\documentclass{amsart}
\usepackage{amsmath}
\usepackage{amsthm}
\usepackage{amssymb,verbatim,appendix}
\usepackage{amscd}
\usepackage{mathtools}
\usepackage{enumerate}
\usepackage{url}
\usepackage{wasysym}
\usepackage{color}
\usepackage{bm}

\makeatletter
\@namedef{subjclassname@2020}{\textup{2020} Mathematics Subject Classification}
\makeatother

\newtheorem{theorem}{Theorem}[section]
\newtheorem{lemma}[theorem]{Lemma}
\newtheorem{corollary}[theorem]{Corollary}
\newtheorem{proposition}[theorem]{Proposition}

\newtheorem{problem}[theorem]{Problem}

\newtheorem{fact}[theorem]{Fact}

\theoremstyle{definition}

\newtheorem{definition}[theorem]{Definition}
\newtheorem{remark}[theorem]{Remark}

\def\Z{\mathbb{Z}}

\def\N{\mathbb{N}}

\def\F{\mathbb{F}}
\def\cF{\mathcal{F}}
\def\cH{\mathcal{H}}
\def\Q{\mathbb{Q}}

\newcommand{\abar}{\bar{a}}
\newcommand{\bbar}{\bar{b}}
\newcommand{\cbar}{\bar{c}}
\newcommand{\dbar}{\bar{d}}
\newcommand{\ebar}{\bar{e}}

\newcommand{\gbar}{\bar{g}}

\newcommand{\mbar}{\bar{m}}
\newcommand{\nbar}{\bar{n}}

\newcommand{\sbar}{\bar{s}}
\newcommand{\tbar}{\bar{t}}
\newcommand{\ubar}{\bar{u}}
\newcommand{\vbar}{\bar{v}}
\newcommand{\wbar}{\bar{w}}
\newcommand{\xbar}{\bar{x}}
\newcommand{\ybar}{\bar{y}}
\newcommand{\zbar}{\bar{z}}

\newcommand{\qftp}{\textnormal{qftp}}
\newcommand{\calL}{\mathcal{L}}

\newcommand{\calH}{\mathcal{H}}
\newcommand{\calC}{\mathcal{C}}
\newcommand{\calM}{\mathcal{M}}
\newcommand{\calN}{\mathcal{N}}
\newcommand{\calA}{\mathcal{A}}
\newcommand{\calB}{\mathcal{B}}
\newcommand{\calD}{\mathcal{D}}

\newcommand{\calE}{\mathcal{E}}

\newcommand{\calI}{\mathcal{I}}
\newcommand{\calJ}{\mathcal{J}}

\newcommand{\calW}{\mathcal{W}}
\newcommand{\calX}{\mathcal{X}}
\newcommand{\calY}{\mathcal{Y}}
\newcommand{\calZ}{\mathcal{Z}}

\newcommand{\M}{\mathbb{M}}
\newcommand{\Span}{\operatorname{Span}}
\def\Aut{\operatorname{Aut}}
\def\tp{\operatorname{tp}}
\def\dom{\operatorname{dom}}

\newcommand{\FOP}{\operatorname{FOP}}
\def\NFOP{\operatorname{NFOP}}
\def\IP{\operatorname{IP}}
\def\NIP{\operatorname{NIP}}
\def\Th{\operatorname{Th}}
\def\ded{\operatorname{ded}}

\newcommand{\cA}{\mathcal{A}}
\newcommand{\cB}{\mathcal{B}}
\newcommand{\cC}{\mathcal{C}}
\newcommand{\cD}{\mathcal{D}}

\newcommand{\cI}{\mathcal{I}}
\newcommand{\cL}{\mathcal{L}}
\newcommand{\cM}{\mathcal{M}}
\newcommand{\cN}{\mathcal{N}}

\newcommand{\cP}{\mathcal{P}}
\newcommand{\bH}{\textnormal{\textbf{H}}}
\newcommand{\bK}{\textnormal{\textbf{K}}}
\newcommand{\uth}{^{\textrm{th}}}
\newcommand{\clqed}{\hfill$\dashv_{\text{\scriptsize{claim}}}$}
\newcommand{\avec}[1]{a_{#1}}
\newcommand{\seq}{\subseteq}
\newcommand{\miff}{\makebox[.4in]{$\Leftrightarrow$}}
\newcommand{\mimp}{\makebox[.4in]{$\Rightarrow$}}
\newcommand{\oless}{\mathbin{\ooalign{$\ocircle$\cr\hspace{1pt}\raisebox{.12ex}{\scalebox{0.84}{$>$}}}}}
\newcommand{\bR}{\mathbf{R}}
\newcommand{\bS}{\mathbf{S}}
\newcommand{\im}{\textnormal{im}}
\newcommand{\inv}{^{\text{-}1}}
\renewcommand{\phi}{\varphi}

\begin{document}
\title{Higher arity stability and the functional order property}

\author{A. Abd Aldaim} 
\address{Department of Mathematics\\
The Ohio State University\\
Columbus, OH, USA}
\email{abdaldaim.1@osu.edu}

\author{G. Conant}
\address{Department of Mathematics\\
The Ohio State University\\
Columbus, OH, USA}
\email{conant.38@osu.edu}

\author{C. Terry}
\address{Department of Mathematics\\
The Ohio State University\\
Columbus, OH, USA}
\email{terry.376@osu.edu}

\begin{abstract}
The $k$-dimensional functional order property ($\FOP_k$) is a combinatorial property of a $(k+1)$-partitioned formula.  This notion arose in work of Terry and Wolf \cite{TW1,TW2}, which identified $\NFOP_2$ as a ternary analogue of stability in the context of two finitary combinatorial problems related to hypergraph regularity and arithmetic regularity. In this paper we show $\NFOP_k$ has equally strong implications in model-theoretic classification theory, where its behavior as a $(k+1)$-ary version of stability is in close analogy to the behavior of $k$-dependence as a $(k+1)$-ary version of $\NIP$. Our results include  several new characterizations of $\NFOP_k$, including a characterization in terms of collapsing indiscernibles, combinatorial recharacterizations, and a characterization in terms of type-counting when $k=2$.  As a corollary of our collapsing theorem, we show $\NFOP_k$ is closed under Boolean combinations, and that $\FOP_k$ can always be witnessed by a formula where all but one variable have length $1$. When $k=2$, we prove a composition lemma analogous to that of Chernikov and Hempel from the setting of $2$-dependence.  Using this,  we provide a new class of algebraic examples of $\NFOP_2$ theories.  Specifically, we show that  if $T$ is the theory of an infinite dimensional vector space over a field $K$, equipped with a bilinear form satisfying certain properties, then $T$ is $\NFOP_2$ if and only if $K$ is stable.  Along the way we provide a corrected and reorganized proof of Granger's quantifier elimination and completeness results for these theories.
\end{abstract}

\date{June 17, 2025}

\thanks{GC was partially supported by NSF grant DMS-2204787. CT was partially supported by NSF grants DMS-2115518 and DMS-2239737.}

\subjclass[2020]{03C45, 03C10}

\maketitle

\section{Introduction}

Combinatorial dividing lines have played a central role in the development of  model theory over the past several decades. 
A rich structure of \emph{stable theories} was developed in the 1970s by Shelah \cite{Shelah}, who also laid the groundwork for most of the other significant dividing lines in modern research, such as the independence property and various tree properties. One of the first major generalizations of stability (in the sense of a ``structure theory") came from work of Kim \cite{KimSST} and Kim and Pillay \cite{KiPi} on \emph{simple theories} in the 1990s. The 2000s then saw an equally fruitful development of NIP theories by Hrushovski, Peterzil, Pillay, and Simon \cite{HPP, HP, HPS} (see also \cite{SiNIP}).  
More recent work includes that of Chernikov and Kaplan \cite{ChKa} on NTP$_2$ theories (see also \cite{ChNTP2}), and of Kaplan and Ramsey \cite{KaRam} on NSOP$_1$ theories (see also \cite{DKR}). For even more recent developments on further dividing lines, see, e.g., \cite{AhnKimLee, Mutch1, Mutch2}.

A common feature of the above dividing lines is that they are  essentially binary, in the sense that they deal with encoding binary configurations in models, using formulas whose variables have been partitioned into two sets. Naturally, this raises the question of whether there exist higher arity analogues of these definitions.  An analogue of NIP, referred to as  ``$k$-dependence",  was defined by Shelah  \cite{Shelah1,Sh2dep}, and studied  by Chernikov, Palac{\'i}n, and Takeuchi \cite{CPT}, Hempel \cite{Hempel},  Chernikov and Hempel \cite{CH,CH2}, and Chernikov and Towsner \cite{ChTow}. 
On the other hand, it has remained an open problem to develop analogues of stability that are essentially higher arity.

  In  \cite{TW1,TW2}, Terry and Wolf addressed this problem from a combinatorial  perspective.  The starting  points for this investigation were two results about stability from the finite setting.  First, following work of Malliaris and Shelah \cite{MS}, it was known that stability characterizes when hereditary classes of finite graphs admit regular decompositions with no irregular pairs.  Second, work of Terry and Wolf \cite{TWa,TWb}, Conant, Pillay, and Terry \cite{CPTstable}, and Conant \cite{CoQSAR}, showed that in the setting of finite groups, stability is closely tied to subgroup structure, reflecting fundamental results from stable group theory going back to the 1980s.  The recent work in \cite{TW1,TW2} takes a systematic approach to finding higher arity analogues of stability by proving higher order analogues of these two types of finitary theorems.   While this investigation gave rise to several competing notions (see \cite[Problem 2.14]{TW2}), one definition  emerged as  particularly robust.  This notion, called the \emph{functional order property}, is defined as follows.\footnote{This is not the original definition from \cite{TW1,TW2} but is easily shown to be equivalent; see Proposition \ref{prop:alleq0} and surrounding discussion.}

\begin{definition}\label{def:FOPk}
Suppose $T$ is a complete $\cL$-theory. An $\cL$-formula $\varphi(x_1,\ldots,x_{k+1})$ has \textbf{$\FOP_k$ in $T$} if there is a model $\cM\models T$, and sequences $(a_f)_{f:\omega^{k-1}\to\omega}$, $(b^1_i)_{i<\omega}$, \ldots, $(b^k_i)_{i<\omega}$, such that each  $a_f\in M^{x_1}$, each $b^t_i\in M^{x_{t+1}}$, and
\[
\cM\models\varphi(a_f,b^1_{i_1},\ldots,b^k_{i_k})\miff i_k\leq f(i_1,\ldots,i_{k-1}).
\]
We say that $T$ is \textbf{NFOP}$_k$ if no $(k+1)$-partitioned formula has $\FOP_k$ in $T$. 
\end{definition}

The work of Terry and Wolf in \cite{TW1,TW2} develops $\NFOP_2$ as a ternary analogue of stability in the finitary settings of hypergraph regularity and quadratic Fourier analysis.
 In analogy to \cite{MS}, they proved in \cite{TW2}  that  the absence of $\FOP_2$ characterizes when hereditary classes of $3$-uniform hypergraphs admit regular decompositions where the irregular triads are controlled using certain lower arity sets.  In analogy to \cite{TWa,TWb}, they  proved in  \cite{TW1} that subsets of $\F_p^n$ whose ternary sum-graphs are $\NFOP_2$ admit quadratic decompositions where the error set can be controlled using linear structure.  This result demonstrates that the functional order property is especially relevant to the arithmetic setting, and to the machinery of higher order Fourier analysis.  Specifically, the results of \cite{TWa,TWb,TW1} establish correspondences between stability and the linear level in this hierarchy, and between $\NFOP_2$ and the quadratic level.  Quadratic structure in this context is fundamentally related to bilinear forms.

 However, the results of \cite{TW1,TW2} say little about $\NFOP_k$ for $k>2$, and do not explore this notion in the infinite setting.  Thus, it was an open question whether the functional order property has interesting implications in the  model-theoretic setting of first-order theories. In this paper, we prove a number of  results about $\NFOP_k$ theories and formulas, which together demonstrate a positive answer. As the reader will see, the previous literature on $k$-dependence (which we will denote by $\NIP_k$) has played a significant role in shaping these theorems. This is especially the case for the work of Chernikov, Palac\'{i}n, and Takeuchi \cite{CPT} and Chernikov  and Hempel \cite{CH}.

We will now summarize  the main theorems of the paper. We work with a fixed complete $\cL$-theory $T$ and a sufficiently saturated monster model $\M$. Note that Definition \ref{def:FOPk} can be equivalently stated to refer just to $\M$. 

 One of the primary tasks of this paper is to characterize $\NFOP_k$ via collapse of indiscernible sequences indexed by a certain Fra\"{i}ss\'{e} limit. This is done for $\NIP_k$ in \cite{CPT} where the choice of Fr\"{a}iss\'{e} class is intuitive. For $\NFOP_k$, however, Definition \ref{def:FOPk} does not clearly suggest an appropriate class, especially due to the role of \emph{functions} rather than relations. (This is relevant since most of the significant tools developed for indiscernible collapse, such as the work of Scow \cite{Scow1,Scow2}, are done in a relational setting.) For these reasons, the first main section of the paper, Section \ref{sec:combochar}, proves several equivalent characterizations of $\FOP_k$ (see  Proposition \ref{prop:alleq}). Of particular relevance to the previous discussion is the following.

\begin{proposition}\label{prop:order}
    A formula $\varphi(x_1,\ldots,x_{k+1})$ has $\FOP_k$ in $T$ if and only if for any linear order $<_*$ on $\omega^k\cup\omega$, there are sequences $(a^1_i)_{i<\omega}$, \ldots, $(a^{k+1}_i)_{i<\omega}$ such that each  $a^t_i\in \M^{x_t}$ and $\M\models\varphi(a^1_{i_1},\ldots,a^{k+1}_{i_{k+1}})$ if and only if $(i_1,\ldots,i_k)<_* i_{k+1}$.
\end{proposition}

The characterization in Proposition \ref{prop:order} leads to a natural definition of a class $\bH_k$ of finite structures in a \emph{relational} language denoted $\calL_k$ (see Definitions \ref{def:Hk} and \ref{def:Tk}, along with surrounding discussion). Roughly speaking, a structure in $\bH_k$ is obtained from a finite half-graph by blowing up one side into the Cartesian product of $k$ disjoint sets.  This process can be seen as a specific example of a general ``blow-up" operation, which we show produces new Fra\"{i}ss\'{e} classes from old (see Proposition \ref{prop:fraisseop}). In our case, we obtain $\bH_k$ as a blow-up of the class of finite linear orders equipped with a unary predicate. With the Fra\"{i}ss\'{e} class $\bH_k$ in hand, we  obtain a Fra\"{i}ss\'{e} limit $\cH_k$.  One of the main results of the paper is that a theory is $\NFOP_k$ if and only if all $\calH_k$-indiscernible sequences collapse to a certain sublanguage of $\cL_k$, which we will call $\calL_k''$ (see Definition \ref{def:L''}).

\begin{theorem}\label{thm:collapse-intro}
$T$ is $\NFOP_k$ if and only if every $\cH_k$-indiscernible sequence in $\M$ is $\cL''_k$-indiscernible.
\end{theorem}

The previous theorem is an abridged version of Theorem \ref{thm:genind}.
A key ingredient required for this result is a proof that $\bH_k$ has the Ramsey property. Here again, we find a significant divergence from the case of $k$-dependence, where this essentially boils down to the Ramsey property for ordered $(k+1)$-uniform hypergraphs, which is a classical result (see \cite{AbHarr, NR77b, NR82, NR83}). For $\bH_k$ however, the verification of the Ramsey property is  more complicated and uses  much more recent theorems of Evans, Hubi\v{c}ka, and Ne\v{s}et\v{r}il \cite{EHN}, and Hubi\v{c}ka and Ne\v{s}et\v{r}il \cite{HN} (see Section \ref{sec:ramsey}).

One of our main applications of Theorem \ref{thm:collapse-intro} is the following one-variable reduction for $\NFOP_k$ (which restates Theorem \ref{thm:FOPkone}).

\begin{theorem}\label{thm:onev-intro}
If $T$ has $\FOP_k$, then there is a formula $\varphi(x_1,\ldots, x_{k+1})$ with $\FOP_k$ in $T$ and with $|x_i|=1$ for all $1\leq i\leq k$. 
\end{theorem}

It is worth emphasizing that the main work required to obtain this result is the proof of Theorem \ref{thm:collapse-intro}, which is extremely complicated, and factors through a collapsing result for a second sublanguage of $\cL_k$.
Beyond this, the proof of Theorem \ref{thm:onev-intro}   essentially follows the same general strategy used for $\NIP_k$ in \cite{CPT} and \cite{CH}.  

 We now briefly summarize some of our other applications of Theorem \ref{thm:collapse-intro}. Perhaps most importantly,   we obtain a short proof of the following fundamental result (which appears later as Theorem \ref{thm:FOPkbool}).

\begin{theorem}
    The collection of $\NFOP_k$ formulas in $T$ is closed under Boolean combinations.
\end{theorem}

This was proved for $k=2$ by Terry and Wolf in \cite{TW2} using a long and complicated hypergraph regularity argument. The result for general $k$ was open.  

In Section \ref{sec:examples} we also give several examples of theories with and without $\FOP_k$, including an easy proof that any complete theory with quantifier elimination in a finite relational language of arity at most $k$ is $\NFOP_k$ (Proposition \ref{prop:arity}). From this we observe that $\NFOP_k$ crosscuts many known dividing lines, including simplicity, NSOP$_1$, and NTP$_2$. Our examples also show that the general implications $\NFOP_k\Rightarrow \NIP_k\Rightarrow \NFOP_{k+1}$ (proved in Proposition \ref{prop:IPk}) are proper.

We now discuss results primarily focused on the $k=2$ case. In Corollary \ref{cor:smooth}, we prove that the theory of any \emph{smoothly approximable} $\aleph_0$-categorical structure is $\NFOP_2$. For example, this shows that the theory of an infinite extraspecial $p$-group is $\NFOP_2$ in the pure group language (see Corollary \ref{cor:extraspecial}; this also follows from Theorem \ref{thm:VSintro} below).
We then characterize $\NFOP_2$ via the following type-counting dichotomy (see Theorem \ref{thm:FOP2dich} for a complete statement with details).

\begin{theorem} Fix an $\cL$-formula $\varphi(x,y,z)$ and an infinite cardinal $\kappa$.
\begin{enumerate}[$(a)$]
\item Assume $\varphi(x,y,z)$ has $\FOP_2$ in $T$. Then for any cardinal $\lambda\geq\ded(\kappa)$ there is an element $a\in\M^x$ and mutually indiscernible sequences $I\seq\M^y$ and $J\seq \M^z$, with $I$ of length $\lambda$ and $J$ of size $\kappa$,  such that for any $\ell<\lambda$, the number of $\varphi$-types over $\{a\}\times J$ realized in $I_{>\ell}$ is at least $\ded(\kappa)$.  
\item Assume $\varphi(x,y,z)$ is $\NFOP_2$ in $T$. Then for any regular cardinal $\lambda>\kappa^{\aleph_0}$, any element $a\in\M^x$, any set $J\seq\M^z$ of size $\kappa$, and any $\lambda$-indexed $J$-indiscernible sequence $I\seq\M^y$, there is some $\ell<\lambda$ such that the number of $\varphi$-types over $\{a\}\times J$ realized in $I_{>\ell}$ is at most $\kappa$.  
\end{enumerate}
\end{theorem}

Using part $(a)$ of the previous result, we prove a
stable/NFOP$_2$ analogue of the NIP/NIP$_k$ ``composition lemma" proved by Chernikov and Hempel in \cite{CH}. Roughly speaking, our version says that if one  takes a formula in a stable structure, and then replaces the free variables with arbitrary binary functions, then the new formula is $\NFOP_2$ (in the theory expanded by the binary functions). See Theorem \ref{thm:FOP2comp} for a precise statement.  We use this theorem to establish a new class of examples of $\NFOP_2$ theories of vector spaces with bilinear forms. In particular, we prove the following result (which restates Theorem \ref{thm:VSFOP2}).

\begin{theorem}\label{thm:VSintro}
    Let $K$ be an arbitrary first-order expansion of a field in a language eliminating quantifiers. Let $T$ be one of the following two-sorted theories:
    \begin{enumerate}[$(i)$]
        \item $T^K_{\infty,\textnormal{alt}}\colon$ the theory of infinite dimensional vector spaces over a model of $\Th(K)$, equipped with a non-degenerate alternating bilinear form.
        \item $T^K_{\infty,\textnormal{sym}}\colon$ the theory of infinite dimensional vector spaces over a model of $\Th(K)$, equipped with a non-degenerate symmetric bilinear form. 
    \end{enumerate}
    In case $(ii)$, assume also that $\operatorname{char}(K)\neq 2$ and $K$ has square roots. Then $T$ is $\NFOP_2$ if and only if $\Th(K)$ is stable. 
\end{theorem}

The theories defined in the previous result were intially studied by Granger \cite{Gr} in his thesis. We also note that Theorem \ref{thm:VSintro} is an analogue of a result of Chernikov and Hempel \cite{CH}, who showed that a theory $T$ as above is $\NIP_2$ if and only if $\Th(K)$ is $\NIP$. Our  proof strategy is roughly the same, in the sense that we apply a quantifier elimination result along with our stable/$\NFOP_2$ version of the composition lemma. However, after initially mapping out this strategy in the same setting as that of \cite{CH}, we learned from \cite{Dobrowolski} that the language used by Granger in his thesis is not sufficient for quantifier elimination. A corrected language was suggested by Macpherson (who initially found the error), however no proof has appeared in the literature as far as we are aware. Thus we take the liberty in Section \ref{sec:QE} of correcting Granger's proof of quantifier elimination (and completeness). Our proof of Theorem \ref{thm:VSintro} then proceeds more or less as in \cite{CH}, although the new language does naturally produce some differences.

Along the way, we prove some  results for $k$-linear forms, including Proposition \ref{prop:fieldFOP} which shows that a certain combinatorial property of a $k$-linear form transforms instability in the field sort to $\FOP_k$ in the vector space. In Section \ref{sec:QE}, we also give a characterization of quantifier elimination for a theory of vector spaces (over some fixed $K$ as above) equipped with a $k$-linear form (see Theorem \ref{thm:rich}).

After posting a preprint of this paper to the arXiv, it came to our attention that unpublished work of Chernikov and Hempel independently obtains some related results.  In particular, they give a corrected quantifier elimination result for vector spaces equipped with non-degenerate bilinear forms. Their work goes further, dealing with $k$-linear forms for $k\geq 2$ and the relationship there to VC$_k$-dimension. These results have since appeared in the preprint \cite{ChHe3}.  In their unpublished draft, Chernikov and Hempel also prove a composition lemma (for $k=2$) similar to Theorem \ref{thm:FOP2comp}, with the stronger conclusion that the resulting formula is $\text{NIFOP}_2$ (see Definition \ref{def:increasingfop} and subsequent discussion).  Combining this stronger composition lemma with the proofs in Section \ref{sec:VS} yields the a stronger dichotomy for the theory of an infinite-dimension $K$-vector space equipped with a  non-degenerate bilinear form: if $K$ is unstable, this theory has $\FOP_2$, and if $K$ is stable, this theory is $\text{NIFOP}_2$ (a stronger conclusion than  $\NFOP_2$).

\subsection*{Acknowledgements}

We thank Greg Cherlin, Dugald Macpherson, Maryanthe Malliaris, Lynn Scow, and Julia Wolf for discussions about various results. We also thank the referee for their careful reading of the paper, along with several useful and interesting remarks.

\subsection*{Notation}

We set out some notation here which we will use throughout the paper.  

 If $(X,<)$ is a linearly ordered set and $x\in X^r$, we say $x$ is \emph{$<$-increasing} if $x_1<\ldots<x_r$.  Given a $t$-ary relation $R$ on a set $X$ and some $Y\subseteq X$, $R|_Y$ denotes the restriction of $R$ to $Y$, i.e. $R\cap Y^t$. In some cases we will use $R{\upharpoonright}Y$ for readability.  Similarly, if $f\colon X\rightarrow Z$ is a function, the notations $f|_Y$ and $f{\upharpoonright}Y$ are used for the restriction of $f$ to domain $Y$. 
 Given a relation $R\seq X\times Y$, let $R^{\textnormal{opp}}$ denote the reverse relation on $Y\times X$. Given an integer $n\geq 1$, let $[n]=\{1,\ldots, n\}$.

As is customary in model theory, we will often omit bars on tuples unless it is useful for understanding.  We will also sometimes conflate tuples with their underlying sets, and use the notation $a\in X$ to denote that $a$ is a \emph{tuple} of elements of $X$ (when there is no possibility of confusion).

Given a first order language $\calL$ and an $\calL$-structure  $\calM$, we let $M$ denote the underlying set of $\calM$, and use $f^{\calM},R^{\calM},c^{\calM}$ to denote the interpretations in $\calM$ of function symbols $f$, relation symbols $R$, and constant symbols $c$ from $\calL$. Often to simplify notation, we will use $R(\cM)$ instead of $R^{\cM}$. In situations where the distinction between substructure and subset is especially relevant, we will write $\cM\leq\cN$ to denote substructure.  Given  some $\cL_0\seq\cL$, we use $\cM{\upharpoonright}\cL_0$ for the reduct of an $\cL$-structure $\cM$ to $\cL_0$. Similarly, if $T$ is a (complete) $\cL$-theory, then $T{\upharpoonright}\cL_0$ is the $\cL_0$-reduct of $T$.

\textbf{Throughout the paper, we let $k\geq 1$ be a fixed integer.}

\section{First results on $\FOP_k$}

Throughout this section, we let $T$ be a complete theory in a language $\cL$, and let $\M$ denote a sufficiently saturated and strongly homogeneous  model of $T$.

\subsection{Combinatorial characterizations}\label{sec:combochar}
 The goal of this subsection is to establish several combinatorial reformulations of $\FOP_k$ as a property of formulas.   We first recall the definition of $\FOP_k$ given in the introduction (but formulated in the monster model $\M$).

 \begin{definition}\label{def:FOPk2}
     An $\cL$-formula $\varphi(x_1,\ldots,x_{k+1})$ has $\bm{{\FOP}_k}$ in $T$ if there are sequences $(a_f)_{f:\omega^{k-1}\to\omega}$, $(b^1_i)_{i<\omega}$, \ldots, $(b^k_i)_{i<\omega}$ such that each $a_f\in \M^{x_1}$, each $b^t_i\in \M^{x_{t+1}}$, and $\M\models\varphi(a_f,b^1_{i_1},\ldots,b^k_{i_k})$ if and only if $i_k\leq f(i_1,\ldots,i_{k-1})$.
     
     The theory $T$ is $\bm{{\NFOP}_k}$ if no formula has $\FOP_k$ in $T$.
 \end{definition}

Roughly speaking a $(k+1)$-partitioned formula has $\FOP_k$ in $T$ if and only if it codes all subsets of $\omega^k$ given by the area under the graph of a function $f\colon\omega^{k-1}\rightarrow \omega$.  One can view this  as an ``ordered variation" of the $k$-ary analogue of the independence property ($\IP_k$), which is about coding \emph{arbitrary} subsets of $\omega^k$. We will elaborate further on these details in Section \ref{sec:IPk}.

 Before discussing the main equivalences we will need for our results, we quickly point out that our definition of $\FOP_k$ is equivalent to the original definition from \cite{TW1,TW2}. This was already observed in in \cite{TW2} (see Appendix D there) and  the proof is short and straightforward. So we emphasize that our change in the definition of $\FOP_k$  is  entirely superficial and done mostly for aesthetic reasons.

\begin{proposition}\label{prop:alleq0}
An $\cL$-formula $\varphi(x_1,\ldots,x_{k+1})$ has $\FOP_k$ in $T$ if and only if there exist an array $(a^f_i)_{i<\omega,f:\omega^k\to\omega}$, and sequences $(b^1_i)_{i<\omega}$, \ldots, $(b^k_i)_{i<\omega}$ such that each  $a^f_i\in\M^{x_1}$, each $b^t_i\in \M^{x_{t+1}}$, and
\[
\M\models\varphi(a^f_j,b^1_{i_1},\ldots,b^k_{i_k})\miff i_k\leq f(j,i_1,\ldots,i_{k-1}).
\] 
\end{proposition}
\begin{proof}
   Suppose $\varphi(\xbar)$ has $\FOP_k$ in $T$ witnessed by $(a_f)_{f:\omega^{k-1}\to w}$ and $(b^t_i)_{t\in[k],i<\omega}$. Given $f\colon\omega^k\to \omega$ and $j<\omega$, let $f_{j}\colon\omega^{k-1}\to \omega$ be the function defined by $f_{j}(i_1,\ldots, i_{k-1})=f(j,i_1,\ldots, i_{k-1})$ and set $a^f_{j}=a_{f_{j}}$. Then  we have 
   \[
   \M\models\varphi(a^f_{j},b^1_{i_1},\ldots, b^k_{i_k})\miff i_k\leq f_j(i_1,\ldots,i_{k-1})=f(j,i_1,\ldots,i_{k-1}).
   \]

Conversely, suppose we have $(a^f_i)_{i<\omega,f:\omega^k\to\omega}$ and $(b^t_i)_{t\in[k],i<\omega}$ as in the statement of the proposition. Given $f\colon\omega^{k-1}\to \omega$, define $g_f\colon\omega^k\to \omega$ by $g_f(i_1,\ldots, i_k)=f(i_2,\ldots, i_k)$, and let $a_f:=a^{g_f}_0$. Then we have 
\[
\M \models \varphi(a_f,b^1_{i_1},\ldots, b_{i_k}^k)\miff i_k\leq g_f(0, i_1,\ldots, i_{k-1}) = f(i_1,\ldots, i_{k-1}).\qedhere
\]
\end{proof}

 Now we move on to the main equivalences, which are given in Proposition \ref{prop:alleq} below.  To help motivate these statements, first recall that a formula $\varphi(x,y)$ has the \emph{order property in  $T$} if  there are sequences $(a_i)_{i\in \omega}$ in $\M^x$ and $(b_i)_{i\in \omega}$ in $\M^y$ so that $\M\models\varphi(a_i,b_j)$ if and only if $i\geq j$.\footnote{Note that the order property is equivalent to $\FOP_1$ (up to identifying $i<\omega$ with the function $\emptyset\mapsto i$ from $\omega^0$ to $\omega$). See also \cite[Fact A.18]{TW1}.} Condition $(ii)$ of Proposition \ref{prop:alleq} characterizes $\FOP_k$ in terms of a version of the  order property in which one of these sequences has been replaced by a Cartesian product of $k$ sequences. The key advantage of this formulation is that it is purely ``relational" in the sense that functions are not explicitly mentioned (unlike the definition of $\FOP_k$). Consequently, this characterization will allow us to define a Fra\"{i}ss\'{e} class used to characterize $\FOP_k$ via ``collapse of  indiscernibles". This is done in Sections \ref{sec:Fraisse} and \ref{sec:collapse}. 

 Condition $(iii)$ of Proposition \ref{prop:alleq} characterizes $\FOP_k$ in terms of a version of the order property where one of the sequences is replaced by an edge-colored $k$-partite  $k$-uniform hypergraph. This characterization plays a crucial role in the applications to higher order stable  regularity appearing in \cite{TW2,TW1} (see, e.g,  \cite[Theorem 6.9]{TW2}). It will also be used to prove $\NFOP_2$ for smoothly approximable structures (Corollary \ref{cor:smooth}), as well as in transferring instability in a field $K$ to $\FOP_k$ in vector spaces over $K$ equipped with a suitably generic $k$-linear form (Proposition \ref{prop:fieldFOP})

\begin{proposition}\label{prop:alleq}
Given an $\cL$-formula $\varphi(x_1,\ldots,x_{k+1})$, the following are equivalent.
\begin{enumerate}[$(i)$]
\item $\varphi(x_1,\ldots,x_{k+1})$ has $\FOP_k$ in $T$.

\item For any linear order $<_*$ on $\omega^k\cup \omega$, there are sequences $(a^1_i)_{i<\omega}$, \ldots, $(a^{k+1}_i)_{i<\omega}$ such that each  $a^t_i\in \M^{x_t}$ and
\[
\M\models\varphi(a^1_{i_1},\ldots,a^{k+1}_{i_{k+1}})\miff  (i_1,\ldots,i_k)<_* i_{k+1}.
\]
\item For any $s<\omega$ and any partition $\omega^k=E_1\cup\ldots\cup E_s$, there are sequences $(a^1_i)_{i<\omega}$, \ldots, $(a^k_i)_{i<\omega}$, $(b_j)_{j=1}^s$ such that each  $a^t_i\in \M^{x_t}$, each  $b_j\in \M^{x_{k+1}}$, and
\[
\textstyle \M\models\varphi(a^1_{i_1},\ldots,a^k_{i_k},b_j)\miff (i_1,\ldots,i_k)\in \bigcup_{\ell\geq j}E_\ell.
\]
\end{enumerate}
\end{proposition}
\begin{proof}
$(i)\Rightarrow (ii)$. Assume $(i)$, and let $(a_f)_{f:\omega^{k-1}\to\omega}$, $(b^t_i)_{t\in[k],i<\omega}$ witness $\FOP_k$ for $\varphi(x_1,\ldots,x_{k+1})$ in $T$. By compactness, it suffices to fix $n<\omega$ and prove $(ii)$ for an arbitrary linear order $<_*$ on $\omega^k\cup [n]$. Without loss of generality, assume that the restriction of $<_*$ to $[n]$ is the reverse of the standard order. Given $i<\omega$, define a function $f_i\colon\omega^{k-1}\to\omega$ such that 
\[
f_i(i_1,\ldots,i_{k-1})={\textstyle \max_{<}}\{j\in[n]:(i,i_1,\ldots,i_{k-1})<_* j \}
\]
(where we adopt the convention $\max\emptyset=0$). For $i<\omega$, let $\hat{a}^1_i=a_{f_i}$. For $1<t\leq k+1$ and $i<\omega$, let $\hat{a}^t_i=b^{t-1}_i$. Then given $i_1,\ldots,i_k<\omega$ and $1\leq i_{k+1}\leq n$,
\begin{align*}
\M\models\varphi(\hat{a}^1_{i_1},\ldots,\hat{a}^{k+1}_{i_{k+1}}) &\miff \M\models\varphi({a_{f}}_{i_1},b^1_{i_2},\ldots,b^k_{i_{k+1}})\\
&\miff i_{k+1}\leq f_{i_1}(i_2,\ldots,i_k)\\
&\miff i_{k+1}\leq\max\{j\in[n]:(i_1,\ldots,i_k)<_*j\}\\
&\miff (i_1,\ldots,i_k)<_*i_{k+1}. 
\end{align*}

$(ii)\Rightarrow (iii)$. Assume $(ii)$. Fix a partition $\omega^k=E_1\cup\ldots\cup E_s$. Let $\prec$ be any linear order on $\omega^k$ so that $E_1\prec\ldots \prec E_s$ and, for each $1\leq \ell\leq s$, $(E_\ell,\prec)$ is arbitrary.  Define a relation $R$ on $\omega^k\times [s]$ as follows.
\[
\textstyle R(i_1,\ldots,i_k,j)\miff (i_1,\ldots,i_k)\not\in \bigcup_{\ell\geq j}E_\ell.
\]
Given $\ubar,\vbar\in \omega^k$ and $i,j\in [s]$, if $\ubar\preceq\vbar$, $R(\vbar,i)$, and $i\leq j$, then $R(\ubar,j)$ holds by construction. It follows that the relation ${\prec'}:={\prec}\cup R\cup {(\neg R)^{\textnormal{opp}}}\cup {<}$ is a linear order on $\omega^k\cup [s]$. Let $\prec_*$ be any linear order on $\omega^k\cup\omega$ extending $\prec'$, and let $<_*$ be the reverse ordering of $\prec_*$. Apply $(ii)$ to $<_*$ to obtain $(a^t_i)_{t\in[k+1],i<\omega}$. For $j\in[s]$, set $b_j=a^{k+1}_j$. Then we have
\begin{multline*}
\M\models\varphi(a^1_{i_1},\ldots,a^k_{i_k},b_j)\miff (i_1,\ldots,i_k)<_* j \miff j\prec_* (i_1,\ldots,i_k)\\
\miff j\prec' (i_1,\ldots,i_k)\miff  \neg R(i_1,\ldots,i_k,j)\miff \textstyle (i_1,\ldots,i_k)\in \bigcup_{\ell\geq j}E_\ell.
\end{multline*}

$(iii)\Rightarrow (i)$.  Assume $(iii)$. By compactness, $(iii)$ holds for any covering of $\omega^k$ by pairwise disjoint sets $(E_i)_{i<\omega}$.  Also by compactness, it suffices to witness $\FOP_k$ with respect to any finite collection $F$ of functions from $\omega^{k-1}$ to $\omega$.  So fix such a finite collection $F$ (we note our proof below would work even if  $F$ were countable). Let $F=\{f_i:i<\omega\}$ be a (repeating) enumeration. Given $i<\omega$, define
\[
E_i=\{(i_1,\ldots,i_k)\in\omega^k:f_{i_1}(i_2,\ldots,i_k)=i\}.
\]
Note that any $(i_1,\ldots,i_k)\in\omega^k$ is contained in $E_{f_{i_1}(i_2,\ldots,i_k)}$. Moreover, if there is some $(i_1,\ldots,i_k)\in E_i\cap E_j$ then $i=f_{i_1}(i_2,\ldots,i_k)=j$. Thus we can apply $(iii)$ with $(E_i)_{i<\omega}$ to obtain $(a^t_i)_{t\in[k],i<\omega}$ and $(b_i)_{i<\omega}$. For $i<\omega$, let $\hat{a}_{f_i}=a^1_{i}$. For $1\leq t<k$ and $i<\omega$, let $b^t_i=a^{t+1}_i$. For $i<\omega$, let $b^{k}_i=b_i$. Then
\begin{align*}
\textstyle \M\models\varphi(\hat{a}_{f_i},b^1_{i_1},\ldots,b^k_{i_k}) &\miff \M\models\varphi(a^1_{i},a^2_{i_1},\ldots,a^k_{i_{k-1}},b_{i_k})\\
  &\miff \textstyle (i,i_1,\ldots,i_{k-1})\in\bigcup_{\ell\geq i_k} E_\ell\\
  & \miff i_k\leq f_i(i_1,\ldots,i_{k-1}).\qedhere
\end{align*}
\end{proof}

\begin{remark}
It is natural to ask whether  $(ii)$ and $(iii)$ can be reformulated in terms of ``existential" conditions in which one only needs to check some sufficiently generic linear order  (in $(ii)$) or partition (in $(iii)$). In Section \ref{sec:Fraisse}, we will construct a certain Fra\"{i}ss\'{e} limit which will lead to reformulations of this kind; see Remark \ref{rem:Gkconditions} for further details.

The interested reader may wish to investigate other variations of the statements in Proposition \ref{prop:alleq}. For example, while the precise form of condition $(iii)$ is motivated by the influence of combinatorics discussed above, the proof of $(ii)\Rightarrow (iii)$ makes it clear that an equivalent variation can be obtained in terms of an arbitrary covering of $\omega^k=\bigcup_{i<\omega}E_i$ by pairwise disjoint sets. But of course this is only one of many such variations (both finitary and infinitary) allowed by compactness. 

It is also worth noting that further combinatorial insights can be obtained by different routes through the proofs of the three equivalent statements.  We leave these as exercises to the reader. 
\end{remark}

As a  consequence of Proposition \ref{prop:alleq}, we observe that $\NFOP_k$ formulas are closed under negation and permutation of the first $k$ variables.

\begin{corollary}\label{cor:easyFOPk}
If $\varphi(x_1,\ldots,x_{k+1})$ is $\NFOP_k$ in $T$ then so is $\neg\varphi(x_1,\ldots,x_{k+1})$ and $\varphi^\sigma(x_{\sigma(1)},\ldots,x_{\sigma(k)},x_{k+1})\coloneqq \varphi(x_1,\ldots,x_{k+1})$ for any permutation $\sigma$ of $[k]$.
\end{corollary}
\begin{proof}
Closure under negations was already shown by Terry and Wolf \cite{TW2} (see Proposition D.2 there), and is fairly clear from  the definition of $\FOP_k$. In any case, condition $(ii)$ makes the claim completely obvious. Closure under permutations of the first $k$ variables is also immediate from $(ii)$  (or $(iii)$). 
\end{proof}

It is also true that $\NFOP_k$ formulas are closed under disjunction and thus arbitrary finite Boolean combinations. This was shown for $k=2$ in \cite{TW2}.  We will provide a  proof for general $k$ later in the paper (see Theorem \ref{thm:FOPkbool}), which will require more elaborate tools on generalized indiscernibles. These results will also demonstrate that in the definition of $\FOP_k$ for a formula $\varphi(x_1,\ldots,x_{k+1})$, the final variable $x_{k+1}$ plays a truly distinct role from the rest. Thus we expect that $\NFOP_k$ is not closed under permutations of variables that move $x_{k+1}$. 

Finally, we show that the definition of $\FOP_k$ can be reformulated using mutually indiscernible sequences. This will be used for our results in Section \ref{sec:FOP2}, which themselves are used in Section \ref{sec:VS}.

\begin{proposition}\label{prop:area}
    If a formula $\varphi(x_1,\ldots,x_{k+1})$ has $\FOP_k$ in $T$, then there is a witness  (as in Definition \ref{def:FOPk2}) in which the sequences $(b^1_i)_{i<\omega}$, ..., $(b^k_i)_{i<\omega}$ are mutually indiscernible.
\end{proposition}
\begin{proof} 
Assume $\varphi(x_1,\ldots,x_{k+1})$ has $\FOP_k$ in $T$. By definition, we can choose $B^t=(b^t_i)_{i<\omega}$ for $t\in[k]$ and $(a_f)_{ f:\omega^{k-1}\rightarrow \omega}$ such that $
\M\models\varphi(a_f,b^1_{i_1},\ldots,b^k_{i_k})$ if and only if $i_k\leq f(i_1,\ldots,i_{k-1})$. Let $x$ be a variable of sort $x_1$ and, for $t\in[k]$ and $i<\omega$, let $y^t_{i}$  be a variable of sort $x_{t+1}$. 
For each $f\colon \omega^{k-1}\rightarrow \omega$ and $n< \omega$, let $\psi_{f,n}((y^1_{i})_{i=1}^n,\ldots (y^k_{i})_{i=1}^n)$ be the formula
\[
\exists x \bigwedge_{(i_1,\ldots,i_k)\in[n]^{k}}  \varphi(x,y^1_{i_1},\ldots,y^k_{i_k})^{i_k\leq f(i_1,\ldots,i_{k-1})}.
\] 
It is not difficult to prove that we can find mutually indiscernible sequences $C^1,\ldots, C^k$ such that for every finite subsequence in $C^i$ and every finite set of formulas $\Delta$, there is some finite subsequence of $B^i$ with the same $\Delta$-type. This is a special case of the \textit{local basedness} of Definition \ref{def:genind}. This fact can be established either directly (e.g. by an argument similar to that in Lemma 1.2 in \cite{ChNTP2}\footnote{The details of such an argument can be found in the first arXiv version of this paper.}) or via the modeling property (Definition \ref{def:genind}), Scow's Theorem (Theorem \ref{thm:scow}) and the fact that the class of linear orders with $k$ named predicates is Ramsey. Now, we write $C^t=(c^{t}_i)_{i<\omega}$ and note that $\M\models \psi_{f,n}(\cbar^{1,n},\ldots,\cbar^{k,n})$ for all $n<\omega$  and $f\colon\omega^{k-1}\to \omega$, where $\cbar^{t,n}=(c^t_1,\ldots,c^t_n)$ for each $1\leq t\leq k$ and $n<\omega$.

Define a type $\Gamma = \Gamma_0\cup\bigcup_{t\in[k]}\Lambda_t$ in variables $(y^t_{i})_{t\in[k],i<\omega}$ and $(x_f)_{f:\omega^{k-1}\to\omega}$, with each $x_f$ of sort $x_1$, as follows:
\begin{enumerate}[\hspace{5pt}$\ast$]
\item $\Gamma_0$ states that $\phi(x_f,y^1_{i_1},\ldots y^k_{i_k})$ holds if and only if $i_k\leq f(i_1,\ldots, i_{k-1})$;
\item $\Lambda_t$ states that $(y^t_{i})_{i<\omega}$ is indiscernible over $(y^s_{t})_{i<\omega,s\in[k]\backslash\{t\}}$. 
\end{enumerate}
If we show $\Gamma$ is satisfiable, we are done. 

Fix some finite $\Delta\subseteq \Gamma$. Let $f_1,\ldots, f_m$ be an enumeration of all functions $f\colon \omega^{k-1}\rightarrow \omega$ such that $x_{f}$ appears in $\Delta$, and let $n<\omega$ be such that for all $1\leq t\leq k$, if $y^t_i$ appears in $\Delta$  then $i\leq n$.

For each $r\in[m]$, we have $\M\models \psi_{f_r,n}(\cbar^{1,n},\ldots,\cbar^{k,n})$, and so we may choose $\hat{a}_r\in \M^{x_1}$ witnessing the existential in $\psi_{f_r,n}$. Since $C^1,\ldots,C^k$ are mutually indiscernible, it follows from the definition of $\psi_{f_r,n}$ that $(\hat{a}_1,\ldots,\hat{a}_m,\cbar^{1,n},\ldots,\cbar^{k,n})$ realizes $\Delta$. By compactness, $\Gamma$ is satisfiable, as desired.
\end{proof}

\subsection{Relation to higher arity independence}\label{sec:IPk}

In this subsection, we discuss the relationship between $\FOP_k$ and  the  higher arity versions of the independence property (defined by Shelah \cite{Shelah1}). As mentioned in the introduction, these often go by the names ``$k$-dependent" and ``$k$-independent". However we will use the acronyms $\NIP_k$ and $\IP_k$ for the sake of consistency.

\begin{definition}
A formula $\varphi(x_1,\ldots,x_{k+1})$ has $\bm{{\IP}_k}$ in $T$ if there are sequences $(a_X)_{X\seq\omega^k}$, $(b^1_i)_{i<\omega}$, \ldots, $(b^k_i)_{i<\omega}$ such that each  $a_X\in \M^{x_1}$, each  $b^t_i\in \M^{x_{t+1}}$, and
\[
\M\models\varphi(a_X,b^1_{i_1},\ldots,b^k_{i_k})\miff (i_1,\ldots,i_k)\in X.
\]
The theory $T$ is $\bm{{\NIP}_k}$ if no formula has $\IP_k$ in $T$.
\end{definition}

\begin{proposition} \label{prop:IPk}
As properties of $T$, 
\[
 \NFOP_k\mimp \NIP_k\mimp \NFOP_{k+1}.
\]
More specifically:
\begin{enumerate}[$(a)$]
\item If a formula $\varphi(x_1,\ldots,x_{k+1})$ has $\IP_k$ in $T$, then it also has $\FOP_k$ in $T$.
\item If a formula $\varphi(x_1,\ldots,x_{k+2})$ has $\FOP_{k+1}$ in $T$, then there is a parameter $b\in\M^{x_{k+2}}$ such that $\varphi(x_1,\ldots,x_{k+1},b)$ has $\IP_k$ in $T$ (with constants  for $b$).
\end{enumerate}
\end{proposition}
\begin{proof}
Part $(a)$. This follows immediately by identifying a function $f\colon\omega^{k-1}\to \omega$ with the set $X_f=\{(i_1,\ldots,i_k)\in\omega^k:i_k\leq f(i_1,\ldots,i_{k-1})\}$.

Part $(b)$. Suppose $(a_f)_{f: \omega^{k}\to\omega}$, $(b^1_i)_{i<\omega}$, \ldots, $(b^{k+1}_i)_{i<\omega}$ witness $\FOP_{k+1}$ for  $\varphi(x_1,\ldots,x_{k+2})$ in $T$. For each $X\seq\omega^{k}$, define  $f_X\colon\omega^{k}\to\omega$ to be the indicator function of $X$, and let $a'_X=a_{f_X}$. Then $(a'_X)_{X\seq\omega^{k}}$, $(b^1_i)_{i<\omega}$, \ldots, $(b^{k}_i)_{i<\omega}$ witness $\IP_{k}$ for $\varphi(x_1,\ldots,x_k,b^{k+1}_1)$ in $T$. 
\end{proof}

The previous result also appears in Fact 2.22 and Lemma 5.11 of \cite{TW2} (for $k=2$).
We will show in Proposition \ref{prop:IPkstrict} that none of the implications  can be reversed. 

In order to further strengthen the analogy between $\NFOP_k$ and $\NIP_k$, we formulate equivalences of $\NIP_k$ suggested by Proposition \ref{prop:alleq}. The analogue of condition $(ii)$ is due to Chernikov, Palac\'{i}n, and Takeuchi \cite{CPT}. Moreover, the combinatorial content of condition $(iii)$ appears in \cite[Proposition 5.6]{TW2} (for $k=2$).

\begin{proposition}\label{prop:IPkeq}
    Given an $\cL$-formula $\varphi(x_1,\ldots,x_{k+1})$, the following are equivalent.
    \begin{enumerate}[$(i)$]
\item $\varphi(x_1,\ldots,x_{k+1})$ has $\IP_k$ in $T$.

\item For any relation $R\seq \omega^{k+1}$, there are sequences $(a^1_i)_{i\in\omega}$, \ldots, $(a^{k+1}_i)_{i<\omega}$ such that each  $a^t_i\in \M^{x_t}$ and
\[
\M\models\varphi(a^1_{i_1},\ldots,a^{k+1}_{i_{k+1}})\miff  R(i_1,\ldots,i_{k+1}).
\]

\item For any $s<\omega$ and partition $\omega^k=\bigcup_{\sigma\seq [s]}E_\sigma$, there are sequences $(a^1_i)_{i<\omega}$, \ldots, $(a^k_i)_{i<\omega}$, $(b_j)_{j\in [s]}$ such that each $a^t_i\in\M^{x_t}$, each $b_j\in \M^{x_{k+1}}$, and
\[
\M\models \varphi(a^1_{i_1},\ldots,a^k_{i_k},b_j)\miff \textstyle (i_1,\ldots,i_k)\in \bigcup_{j\in \sigma}E_\sigma.
\]
    \end{enumerate}
\end{proposition}
\begin{proof}
The equivalence of $(i)$ and $(ii)$ is  \cite[Proposition 5.2]{CPT}. We leave $(ii)\Rightarrow (iii)$ to the reader, since it is  straightforward and easier than the analogous direction of Proposition \ref{prop:alleq}. The proof of $(iii)\Rightarrow (i)$ is also similar to Proposition \ref{prop:alleq}, but in this case the argument is a little more complicated, so we give the details. 

Assume $(iii)$. By compactness, $(iii)$ holds for any covering of $\omega^k$ by pairwise disjoint sets $(E_\sigma)_{\sigma\seq\omega}$. Also by compactness, it suffices to witness $\IP_k$ with respect to some finite collection $F$  of subsets of $\omega^k$. Let $F=\{X_i:i<\omega\}$ be a (repeating) enumeration. Given $\sigma\seq \omega$, define
    \[
E_\sigma=\{(i_1,\ldots,i_k)\in \omega^k:\forall j<\omega,~(i_2,\ldots,i_k,j)\in X_{i_1}\Leftrightarrow j\in\sigma\}.
    \]
   Given $(i_1,\ldots,i_k)\in\omega^k$, if we define $\sigma=\{j<\omega:(i_2,\ldots,i_{k},j)\in X_{i_1}\}$, then $(i_1,\ldots,i_k)\in E_\sigma$.  
Moreover, if $(i_1,\ldots,i_k)\in E_\sigma\cap E_\tau$ then, given $j<\omega$, we have 
\[
j\in\sigma\miff (i_2,\ldots,i_k,j)\in X_{i_1}\miff j\in\tau,
\]
and so $\sigma=\tau$. Thus we can apply $(iii)$ with $(E_\sigma)_{\sigma\seq\omega}$ to obtain $(a^t_i)_{t\in[k],i<\omega}$ and $(b_i)_{i<\omega}$. For $i<\omega$, let $\hat{a}_{X_i}=a^1_{i}$. For $1\leq t<k$ and $i<\omega$, let $b^t_i=a^{t+1}_i$. For $i<\omega$, let $b^k_i=b_i$. Then
\begin{align*}
\M\models\varphi(\hat{a}_{X_i},b^1_{i_1},\ldots,b^k_{i_k}) &\miff \M\models \varphi(a^1_i,a^2_{i_1},\ldots,a^k_{i_{k-1}},b_{i_k})\\
&\miff (i,i_1,\ldots,i_{k-1})\in \textstyle\bigcup_{i_k\in\sigma}E_\sigma.
\end{align*}
Note that the right side of the above equivalences clearly implies $(i_1,\ldots,i_k)\in X_i$. Conversely,  suppose $(i_1,\ldots,i_k)\in X_i$ and define $\sigma=\{j<\omega:(i_1,\ldots,i_{k-1},j)\in X\}$. Then $i_k\in\sigma$ and $(i,i_1,\ldots,i_{k-1})\in E_\sigma$. Thus $\M\models\varphi(\hat{a}_{X_i},b^1_{i_1},\ldots,b^k_{i_k})$ holds if and only if $(i_1,\ldots,i_k)\in X_i$.
\end{proof}

\subsection{Results for $\FOP_2$}\label{sec:FOP2}

In this subsection, we prove some results that are special to the $k=2$ case. First is a basic type-counting result for $\FOP_2$, which we will use to produce some examples.

\begin{proposition}
Suppose $T$ has $\FOP_2$. Then there is a constant $C>0$ so that $|S_n(T)|\geq 2^{Cn^2\log n}$ for arbitrarily large $n$. 
\end{proposition}
\begin{proof}
Suppose $\phi(x,y,z)$ has $\FOP_2$ in $T$. Fix some $n\geq 1$ and let $I$ be the set of functions $f\colon [n]\times[n]\to [n]$. Fix variables $\xbar=(x_1,\ldots,x_n)$, $\ybar=(y_1,\ldots,y_n)$, and $\zbar=(z_1,\ldots,z_n)$. Given $f\in I$, define
\[
\Gamma_{f}(\xbar,\ybar,\zbar)=\left\{\varphi(x_{i_1},y_{i_2},z_j):f(i_1,i_2)\geq j\right\}\cup \left\{\neg\varphi(x_{i_1},y_{i_2},z_j):f(i_1,i_2)<j\right\}.
\]
It follows from the characterization of $\FOP_2$ in Proposition \ref{prop:alleq0} (for $k=2$) that $\Gamma_f$ is satisfiable for all $f\in I$. Further, it's clear that $f\neq f'$ implies that $\Gamma_{f}$ and $\Gamma_{f'}$ are contradictory.  Let $t=|x|+|y|+|z|$. Note $|I|=n^{n^2}$.  Thus $|S_{nt}(T)|\geq n^{n^2}=2^{n^2\log n}$.  Consequently, there is $C>0$ (depending on $t$) so that for arbitrarily large $N$, $|S_{N}(T)|\geq 2^{CN^2\log N}$. 
\end{proof}

The statement of the previous result is only meaningful when $T$ is  $\aleph_0$-categorical since otherwise there is some $n$ such that $S_n(T)$ is infinite (but note that the proof is providing more information about $\varphi$-types that could still be interesting in general). In particular, we will consider  ``smoothly approximable" $\aleph_0$-categorical structures (for definitions, we refer to the reader to the monograph \cite{ChHrFSFT} and the survey article \cite{MacSA}). A celebrated result of Cherlin, Harrington, and Lachlan \cite{CHL} is that any $\omega$-stable $\aleph_0$-categorical structure is smoothly approximable. The converse need not be true since, for example, if $\cM$ is a countably infinite dimensional vector space over a finite field equipped with a non-degenerate alternating bilinear form, then $\cM$ is smoothly approximable and unstable (again, see \cite{ChHrFSFT, MacSA}). However, using the previous proposition and a deep classification theorem from \cite{ChHrFSFT}, we can show that any smoothly approximable structure is ``stable" in our ternary sense.

\begin{corollary}\label{cor:smooth}
If $\calM$ is a smoothly approximable first-order structure, then $\Th(\calM)$ is $\NFOP_2$.
\end{corollary}
\begin{proof}
Suppose towards a contradiction $\calM$ is smoothly approximable but $T=\Th(\calM)$ has $\FOP_2$.  By above, there is $C>0$ so that for arbitrarily large $N$, $|S_N(T)|\geq 2^{CN^2\log N}$. Since $\calM$ is smoothly approximable, Theorem 2(6) of \cite{ChHrFSFT} implies there is a model $\calM^*$ of $T$ of non-standard finite size, so that for all standard $n\in \mathbb{N}$, the number of internal $n$-types over the empty set is at most $c^{n^2}$, for some finite $c$, and where the internal $n$-types and $n$-types over $\calM^*$ coincide.  However, if $N$ is sufficiently large, then $2^{CN^2\log N}>c^{N^2}$, and thus there are strictly more than $c^{N^2}$ many $N$-types over the empty set in $\calM^*$, a contradiction.
\end{proof}

\begin{remark}[pointed out to us by the referee]
Theorem 3 of \cite{ChHrFSFT}  states that a structure $\cM$ is a reduct of a smoothly approximable structure if and only if $\Th(\cM)$ is \emph{quasi-finite}, i.e., there is some $\nu\colon \N\to \N$ such that any finite subset of $\Th(\cM)$ has a finite model with at most $\nu(k)$ $k$-types for all $k$ (see Definition 2.1.1 of \cite{ChHrFSFT}, and also Section 6 of \cite{HruPFD}). Since $\NFOP_2$ is clearly closed under reducts, it then follows from Corollary \ref{cor:smooth} that any quasi-finite theory is $\NFOP_2$.  
\end{remark}

We can use Corollary \ref{cor:smooth} to give an example of an $\NFOP_2$ group with the independence property. Given an odd prime $p$, an infinite group $G$ is called an \emph{extraspecial $p$-group} if every element of $G$ has order $p$, $Z(G)$ is cyclic, and $Z(G)=G'$. The theory of infinite extraspecial $p$ groups is complete, $\aleph_0$-categorical, and unstable, as first shown in \cite{Felgner}. Moreover, this theory is bi-interpretable with the theory $T_{\infty,\textnormal{alt}}^{\mathbb{F}_p}$ of an infinite dimensional vector space over $\mathbb{F}_p$ with a non-degenerate alternating bilinear form, and so it is  supersimple (hence $\IP$). For details see the appendix of \cite{Milliet} (or for different approaches to get supersimplicity, see \cite{MacStein} and \cite{Baudisch}). As noted above, $T^{\mathbb{F}_p}_{\infty,\textnormal{alt}}$ is the theory of smoothly approximable structure, and thus is $\NFOP_2$ (we will reprove this for a more general class of theories of vector spaces with bilinear forms in Section \ref{sec:VS}). So we have the following conclusion.

\begin{corollary}\label{cor:extraspecial}
    For any odd prime $p$, the theory of infinite extraspecial $p$-groups is $\NFOP_2$ and $\IP$.
\end{corollary}

\begin{remark}
    The previous corollary raises the natural question of finding groups whose theories are $\NFOP_k$ and $\FOP_{k-1}$ (or even $\IP_{k-1}$) for arbitrary $k$. In \cite[Section 4]{CH2}, Chernikov and Hempel construct groups with $\NIP_k$ and $\IP_{k-1}$ by showing that $\NIP_k$ is preserved by \emph{Mekler's construction} (a process for interpreting a given structure in a finite relational language inside a pure $2$-step nilpotent group of finite exponent). In addition to general combinatorical facts about Mekler's construction, their proof relies primarily on the characterization of $\NIP_k$ using collapse of indiscernibles. Thus we expect that their argument can be directly modified to produce the same preservation result for $\NFOP_k$ using the indiscernible collapse characterization that we will prove below in Theorem \ref{thm:genind}. So for example, if one applies Mekler's construction to the generic $k$-uniform hypergraph (which is $\NFOP_k$ by Proposition \ref{prop:arity} below), then we conjecture that the resulting group is $\NFOP_k$ (note that this group is $\NIP_k$ and $\IP_{k-1}$ by \cite{CH2}).\footnote{This conjecture has since been  proved by Boissonneau, Papadopoulos, and Touchard  \cite{BoPaTo}.}
\end{remark}

Next we give a \emph{characterization} of $\FOP_2$ in terms of a more elaborate type-counting dichotomy. Recall that for a cardinal $\kappa$, $\ded(\kappa)$ denotes the supremum of the cardinalities of linear orders with a dense subset of size $\kappa$. Clearly $\ded(\kappa)\leq 2^\kappa$. A standard exercise  is that $\ded(\kappa)\geq 2^\lambda$ where $\lambda$ is minimal such that $2^\lambda>\kappa$ (see, e.g., \cite[Lemma 5.2.12]{Marker}). So altogether, $\kappa<\ded(\kappa)\leq 2^\kappa$. We will also need the following result from stability theory.

\begin{fact}[Erd\H{o}s-Makkai \cite{EM}]\label{fact:EM}
Let $X$ be an infinite set and suppose $\cF\seq \cP(X)$ is such that $|\cF|>|X|$. Then there are sequences $(a_i)_{i<\omega}$ from $X$ and $(B_i)_{i<\omega}$ from either $\cF$ or  $\{X\backslash S:S\in\cF\}$ such that, for all $i,j<\omega$, $a_i\in B_j$ if and only if $i\leq j$.
\end{fact}

We can now give our type-counting dichotomy for $\FOP_2$, which will use the following notation. Let $\varphi(x_1,x_2,x_3)$ be an $\cL$-formula. Given $1\leq i\leq 3$ and a set $A\seq\prod_{j\neq i}\M^{x_j}$, we have the local type space $S^\varphi_{x_i}(A)$.  Given $B\seq \M^{x_i}$, we let $S^\varphi_{x_i}(A| B)$ be the set of types in $S^\varphi_{x_i}(A)$ realized in $B$.

\begin{theorem}\label{thm:FOP2dich}
Let $\varphi(x,y,z)$ be a formula, and fix an infinite cardinal $\kappa$.
\begin{enumerate}[$(a)$]
\item Suppose $\varphi(x,y,z)$ has $\FOP_2$ in $T$ and $\lambda\geq\ded(\kappa)$. Then there is some $a\in\M^x$ and mutually indiscernible sequences $I\seq \M^y$ and $J\seq\M^z$ such that $I$ is indexed by $\lambda$, $J$ is indexed by a linear order of size $\kappa$, and for all $\ell<\lambda$,
\[
|S^\varphi_{y}(A|I_{>\ell})|\geq\ded(\kappa),
\]
where $A=\{a\}\times J$.
\item Suppose $\varphi(x,y,z)$ is $\NFOP_2$ in $T$ and $\lambda>\kappa^{\aleph_0}$ is regular. Then for any element $a\in \M^{x}$, any  $\lambda$-indexed sequence $I$ from $\M^y$, and any set $J\subset\M^z$ of size at most $\kappa$, if $I$ is indiscernible over $J$ then  for some $\ell<\lambda$,
\[
|S^\varphi_y(A|I_{>\ell})|\leq\kappa,
\]
where $A=\{a\}\times J$.
\end{enumerate}
\end{theorem}

\begin{proof}
Part $(a)$. Let $(L^*,<)$ be a linear order of size $\ded(\kappa)$ with a dense subset $L\seq L^*$ of size $\kappa$. By Proposition \ref{prop:area} and compactness, there is a sequence $(a_f)_{f:\lambda\to L^*}$ from $\M^{x}$ and mutually indiscernible sequences $(b_i)_{i<\lambda}$ in $\M^{y}$ and $(c_i)_{i\in L^*}$ in $\M^{z}$ such that
\[
\M\models \varphi(a_f,b_{i},c_{j})\miff j\leq f(i).
\]
Let $I=(b_i)_{i<\lambda}$ and let $J$ be the subtuple $(c_i)_{i\in L}$ of $(c_i)_{i\in L^*}$. 

Since $\lambda\geq\ded(\kappa)=|L^*\backslash L|$, we can partition $\lambda=\bigcup_{\alpha\in L^*\backslash L}X_\alpha$ so that each $X_\alpha$ has cardinality $\lambda$, and thus is cofinal in $\lambda$.
From this we can define a function $f\colon \lambda \to L^*\backslash L$ so that for all $i<\lambda$, $i\in X_{f(i)}$. Set $a=a_f$ and $A=\{a\}\times J$. We fix $\ell<\lambda$ and show that $|S^\varphi_{y}(A|I_{>\ell})|\geq\ded(\kappa)$. For each $\alpha\in L^*\backslash L$, since $X_\alpha$ is cofinal in $\lambda$, we may choose some $i_\alpha\in X_\alpha$ with $i_\alpha>\ell$. Set $p_\alpha=\tp^{\varphi}_{y}(b_{i_\alpha}/A)$. Then $p_\alpha\in S^\varphi_{y}(A|I_{>\ell})$ for all $\alpha\in L^*\backslash L$. So it suffices to fix distinct $\alpha,\beta \in L^*\backslash L$, and show $p_\alpha\neq p_\beta$. Without loss of generality, assume $\alpha<\beta$. Fix $j\in L$ such that $\alpha<j<\beta$.  Then
\[
f(i_\alpha)=\alpha<j<\beta=f(i_\beta),
\]
and so $\M\models \varphi(a,b_{i_\alpha},c_{j})\wedge\neg\varphi(a,b_{i_\beta},c_{j})$. Since $(a,c_{j})\in A$, it follows that $p_\alpha\neq p_\beta$, as desired.

Part $(b)$. We prove the contrapositive. Suppose there is some $a\in\M^x$, a $\lambda$-indexed sequence $I$ from $\M^y$, and 
 a set $J\subset\M^z$ of size at most $\kappa$, such that $I$ is $J$-indiscernible and $|S^\varphi_y(A|I_{>\ell})|>\kappa$ for  all $\ell<\lambda$, where $A=\{a\}\times J$. Let $I=(b_i)_{i<\lambda}$ and $J=\{c_i:i<\kappa\}$.

Given $i<\lambda$, let $S_i=\{j<\kappa:\M\models\varphi(a,b_i,c_j)\}$. Note that if $i,i'<\lambda$, then $S_i=S_{i'}$ if and only if $\tp^\varphi_y(b_i/A)=\tp^\varphi_y(b_{i'}/A)$.

We inductively define sequences $(g_\alpha)_{\alpha<\lambda}$ and $(h_\alpha)_{\alpha<\lambda}$ satisfying the following properties:
\begin{enumerate}[$(1)$]
\item Each $g_\alpha$ is a function from $\omega$ to $\lambda$, and each $h_\alpha$ is a function from $\omega$ to $\kappa$.
\item For all $\alpha<\lambda$, one of the following hold.
\begin{enumerate}[$(a)$]
\item For all $m,n<\omega$, $\M\models \varphi(a,b_{g_\alpha(m)},c_{h_\alpha(n)})$ if and only if $n\leq m$. 
\item For all $m,n<\omega$, $\M\models\neg\varphi(a,b_{g_\alpha(m)},c_{h_\alpha(n)})$ if and only if $n\leq m$.
\end{enumerate}
\item If $\alpha<\alpha'<\lambda$ then $g_\alpha(m)<g_{\alpha'}(n)$ for all $m,n<\omega$.  
\end{enumerate}

Fix $\alpha<\lambda$ and suppose $g_{\alpha'}$ and $h_{\alpha'}$ have been defined for all $\alpha'<\alpha$. Define
\[
\ell=\sup\{g_{\alpha'}(n):{\alpha'}<\alpha,~n<\omega\}.
\]
Then $\ell<\lambda$ since $\lambda$ is regular and uncountable. So  $|S^\varphi_y(A|I_{>\ell})|>\kappa$ by assumption. Therefore, if  $\cF=\{S_\gamma:\ell<\gamma<\lambda\}$ then $\cF\seq\cP(\kappa)$ and $|\cF|>\kappa$. So by Fact \ref{fact:EM}, there are $g\colon \omega\to \cF$  and $h\colon \omega\to \kappa$ such that either:
\begin{enumerate}[$(i)$]
\item  for all $i,j<\omega$, $h(i)\in g(j)$ if and only if $i\leq j$, or 
\item for all $i,j<\omega$, $h(i)\not\in g(j)$ if and only if $i\leq j$.
\end{enumerate}
Let $h_\alpha=h$. Define $g_\alpha\colon\omega\to \lambda$ so that $g_\alpha(n)=\min\{\ell<\gamma<\lambda: g(n)=S_{\gamma}\}$. Note that by definition, for each $n<\omega$, $\ell<g_{\alpha}(n)<\lambda$  and $g(n)=S_{g_\alpha(n)}$. Then in case $(i)$, we have that for all $m,n<\omega$, 
\[
\M\models \varphi(a,b_{g_\alpha(m)},c_{h_\alpha(n)})\miff h_\alpha(n)\in S_{g_\alpha(m)}\miff h(n)\in g(m)\miff n\leq m, 
\]
and so property $(2a)$ above holds. 
Similarly, in case $(ii)$, we obtain property $(2b)$. 
This finishes our construction of the sequences.  Let $X_{(a)}$ be the set of $\alpha<\lambda$ where $(2a)$ holds, and let $X_{(b)}$ be the set of $\alpha<\lambda$ where $2(b)$ holds. At least one of these sets has size $\lambda$. Define
\[
(\psi,X)=\begin{cases}
    (\varphi,X_{(a)}) & \text{if $|X_{(a)}|=\lambda$, and}\\
    (\neg\varphi, X_{(b)}) & \text{if $|X_{(a)}|<\lambda$.}
\end{cases}
\]
Then we have  $|X|=\lambda$, and for all $\alpha\in X$ and  $m,n < \omega$, $\M\models \psi(a,b_{g_\alpha(m)},c_{h_\alpha(n)})$ if and only if $n\leq m$. 

 Recall that each $h_\alpha$ is a function from $\omega$ to $\kappa$. Since $\lambda>\kappa^{\aleph_0}$, there is an infinite subset $Y\seq X$ and a function $h\colon\omega\to \kappa$ such that $h_\alpha=h$ for all $\alpha\in Y$. So for all $\alpha\in Y$ and $m,n<\omega$, 
 \[
 \M\models \psi(a,b_{g_\alpha(m)},c_{h(n)})\miff n\leq m.
 \]
For each $m< \omega$, let $c_m^*=c_{h(m)}$.  Let $(\alpha_m)_{m<\omega}$ be an increasing sequence from $Y\seq\lambda$.    

Now fix some function $f\colon \omega\to\omega$. For each  $m<\omega$, define $i_m=g_{\alpha_m}(f(m))$. Note that if $m<m'<\omega$ then $i_m<i_{m'}$ by property $(3)$ above. For each $m<\omega$, set $b^f_m=b_{i_m}$. Observe that for any $m,n<\omega$, we have
\[
\M\models\psi(a,b^f_m,c^*_n)\miff \M\models\psi(a,b_{g_{\alpha_m}(f(m))},c_{h(n)})\miff n\leq f(m).
\]
Recall that $I$ is $J$-indiscernible, and so $(b^f_m)_{m<\omega}\equiv_J (b_m)_{m<\omega}$. Choose $a_f\in\M^x$ such that $a(b^f_m)_{m<\omega}\equiv_J a_f(b_m)_{m<\omega}$. 

Finally, given $f\colon\omega\to\omega$ and $m,n<\omega$, we have
\[
\M\models\psi(a_f,b_m,c^*_n)\miff \M\models \psi(a,b^f_m,c^*_n)\miff n\leq f(m).
\]
So $\psi(x,y,z)$ has $\FOP_2$ in $T$, and hence $\varphi(x,y,z)$ has $\FOP_2$ in $T$ (using Corollary \ref{cor:easyFOPk} in the case that $\psi$ is $\neg\varphi$). 
\end{proof}

As an application of Theorem \ref{thm:FOP2dich}$(a)$, we establish a stable/$\NFOP_2$ ``composition lemma", which says that if one starts with a formula in a stable structure and replaces the free variables by arbitrary binary functions in variables among some fixed $x,y,z$,  then the resulting formula is $\NFOP_2$. This is an analogue of the $\NIP/\NIP_2$ composition lemma proved by Chernikov and Hempel in \cite{CH}. The proof structure  is very much the same as that of \cite{CH}, except  we use Theorem \ref{thm:FOP2dich} (in place of an analogous dichotomy for $\IP_2$) and the fact that if $\psi(x,y)$ is a stable formula then the local type space $S_\psi(B)$ has size at most $|B|+\aleph_0$ for any parameter set $B$ (i.e., the converse of Fact \ref{fact:EM}). We also note that, unlike the $\IP_2$ case, our proof does not require a  set-theoretic absoluteness argument because the relevant cardinal gap, namely $\kappa<\ded(\kappa)$, is true in $\mathsf{ZFC}$  (and we only need it for $\kappa=\aleph_0$).

\begin{theorem}\label{thm:FOP2comp}
Suppose $\cL_0\seq\cL$ is such that $T{\upharpoonright}\cL_0$ is stable. Let $\theta(w_1,\ldots,w_n)$ be an $\cL_0$-formula. Fix disjoint tuples of variables $x,y,z$ and, for each $1\leq i\leq n$,  let $f_i(u_i)=w_i$ be an $\cL$-definable function, where $u_i$ is an ordered pair from $\{x,y,z\}$. Then the $\cL$-formula
\[
\varphi(x,y,z)\coloneqq \theta(f_1(u_1),\ldots,f_n(u_n))
\]
is $\NFOP_2$ in $T$.
\end{theorem}
\begin{proof}
Without loss of generality, assume $T$ is countable. By adding dummy variables and renaming the functions, we may assume each $u_i$ is one of $(x,y)$,  $(x,z)$, or $(y,z)$. Note  that if for some $i_1,\ldots, i_s\leq n$, we have $u_{i_1}=\ldots=u_{i_s}=u$, then we can combine $f_{i_1},\ldots, f_{i_s}$ into a single $\cL$-definable function $g(u)=(f_{i_1}(u),\ldots, f_{i_s}(u))=(w_{i_1},\ldots, w_{i_s})$ in the same tuple of independent variables. Therefore, by 
regrouping and renaming the  variables $w_1,\ldots,w_n$ in $\theta$, we may assume  $\theta$ is of the form $\theta(w_1,w_2,w_3)$ and  
\[
\varphi(x,y,z)\coloneqq \theta(f_1(y,z),f_2(x,z),f_3(x,y)).
\]

Toward a contradiction, suppose $\varphi(x,y,z)$ has $\FOP_2$ in $T$.  
 Let $a\in \M^{x}$, $I\seq\M^y$, and $J\seq\M^z$ be as in Theorem \ref{thm:FOP2dich}$(a)$ with $\kappa=\aleph_0$ and  $\lambda=(2^{\aleph_0})^+$. Recall that $I$ is indexed by $\lambda$ and  $J$ is countable. Enumerate $I=(b_i)_{i<\lambda}$ (so each $b_i$ is in $\M^{y}$). 
 
Recall that $J\seq\M^z$. Let $B=\{f_2(a,c):c\in J\}$ and note that $|B|\leq\aleph_0$. 
Since $I$ and $J$ are mutually indiscernible and $f_1$ is definable, the  array $(f_1(b_i,c))_{c\in J,i<\lambda}$ is indiscernible. So if we let $r_i$  be the sequence $(f_1(b_i,c))_{c\in J}$, then the sequence $(r_i)_{i<\lambda}$ is indiscernible. Note also that each $r_i$ is countable.

 Given $e\in \M^{z}$ and $i<\lambda$, set
\[
J_{e,i}=\{c\in J:\M\models\theta(f_1(b_i,c),f_2(a,c),e)\}.
\]
For $i<\lambda$, set $\cF_i=\{J_{e,i}:e\in \M^{z}\}$.\medskip

\noindent\emph{Claim 1:} For all $i<\lambda$, $|\cF_i|\leq\aleph_0$.

\noindent\emph{Proof:}  Let $\theta^*(w_3;w_1,w_2)$ be $\theta(w_1,w_2,w_3)$. Fix $i<\lambda$. For any $e,e'\in \M^{z}$, if $\tp_{\theta^*}(e/r_iB)=\tp_{\theta^*}(e'/r_iB)$ then $J_{e,i}=J_{e',i}$. So we have a surjective map from $S_{\theta^*}(r_iB)$ to $\cF_i$. Since $|r_iB|\leq\aleph_0$ and $\theta^*$ is stable, it follows that $|\cF_i|\leq\aleph_0$.\clqed\medskip

 Recall that $(r_i)_{i<\lambda}$ is indiscernible. Since $T{\upharpoonright}\cL_0$ is stable (and countable) and $\lambda>
 \aleph_0$ is regular, there is  $\ell<\lambda$ such that $(r_i)_{i>\ell}$ is $\cL_0$-indiscernible over $B$.\footnote{See the ``Claim" in the proof of \cite[Proposition 2.11]{SiNIP}; this only requires $T{\upharpoonright}\cL_0$ to be NIP.}   \medskip

\noindent\emph{Claim 2:} If $\ell< i,j< \lambda$ then $\cF_i=\cF_j$. 

\noindent\emph{Proof:} Fix $\ell< i,j<\lambda$. Without loss of generality, it suffices to show $\cF_i\seq\cF_j$. By $\cL_0$-indiscernibility, there is some $\sigma\in\Aut_{\cL_0}(\M/B)$ such that $\sigma(r_i)=r_j$. Fix $e\in \M^{z}$. Given $c\in J$, we have
\begin{align*}
c\in J_{e,i} &\miff \M\models\theta(f_1(b_i,c),f_2(a,c),e) &\text{(by definition of $J_{e,i}$)}\\
 &\miff \M\models\theta(\sigma(f_1(b_i,c)),\sigma(f_2(a,c)),\sigma(e)) &\text{(since $\sigma$ preserves $\theta$)}\\
 &\miff \M\models\theta(\sigma(f_1(b_i,c)),f_2(a,c),\sigma(e)) &\text{(since $f_2(a,c)\in B$)}\\
 &\miff \M\models \theta(f_1(b_j,c),f_2(a,c),\sigma(e)) &\text{(since $\sigma(r_i)=r_j$)}\\
 &\miff c\in J_{\sigma(e),j} & \text{(by definition of $J_{\sigma(e),j}$).}
\end{align*}
So $J_{e,i}=J_{\sigma(e),j}\in \cF_j$. Therefore $\cF_i\seq \cF_j$, as desired.\clqed\medskip

By the second claim, we can define $\cF\coloneqq\cF_i$ for  $\ell< i<\lambda$. Note $|\cF|\leq\aleph_0$ by the first claim. Given $i<\lambda$, let $K_i=\{c\in J:\M\models \varphi(a,b_i,c)\}$. For any $i<\lambda$ we have $K_i=J_{f_3(a,b_i),i}$. Thus if $\ell<i<\lambda$,  then $K_i\in\cF$. Set $A=\{a\}\times J$. Since $K_i$ determines $\tp^{\varphi}_{y}(b_i/A)$, it follows that $|S^{\varphi}_{y}(A|I_{>\ell})|\leq|\cF|\leq\aleph_0$, which contradicts Theorem \ref{thm:FOP2dich}$(a)$.
\end{proof}

\begin{remark}
As noted in the proof, the previous theorem holds under the weaker assumption that $T{\upharpoonright}\cL_0$ is NIP and $\theta(w_1,\ldots,w_n)$ is stable under any bi-partition of the variables.
\end{remark}

Finally, we remark that an alternate ternary analogue of stability was proposed by Takeuchi \cite{Tak} in a talk at the 2017 Asian Logic Conference in Daejeon, South Korea. We will refer to this  as the \emph{increasing functional order property} ($\text{IFOP}_2$).\footnote{The slides of Takeuchi's talk use the acronym OP$_2$, which has also been used by other authors. However, we will use $\text{IFOP}_2$ to draw contrast with $\FOP_2$, and to emphasize the lack of canonicity of the definition of higher order stability.} 

\begin{definition}[Takeuchi]\label{def:increasingfop}
Let $\omega^{\omega\uparrow}$ denote the set of weakly increasing functions from $\omega$ to $\omega$. 
A formula $\varphi(x_1,x_2,x_3)$  has $\bm{\mathrm{IFOP}_2}$ in $T$ if there are sequences $(a_f)_{f\in\omega^{\omega\uparrow}}$, $(b^1_i)_{i<\omega}$, and $(b^2_i)_{i<\omega}$, such that each $a_f\in\M^{x_1}$, each $b^t_i\in \M^{x_{t+1}}$, and 
\[
\M\models\varphi(a_f,b^1_{i_1},b^2_{i_2})\miff i_2\leq f(i_1).
\]
\end{definition}

 Some basic comparisons between $\textnormal{IFOP}_2$ and $\FOP_2$ are carried out in \cite{TW1}. For example, Proposition A.20 of \cite{TW1} provides a recharacterization of $\textnormal{IFOP}_2$ analogous to Proposition \ref{prop:alleq0} in which one insists that the functions $f\colon \omega^2\rightarrow \omega$ are weakly increasing in the second coordinate. It is clear from the definition that $\FOP_2$ implies $\textnormal{IFOP}_2$. In forthcoming work, the authors show the notions are distinct by exhibiting a complete theory that is $\textnormal{IFOP}_2$ but not $\FOP_2$. This involves a similar characterization of IFOP$_2$ using indiscernible collapse. In particular, one can then use this to establish preservation of NIFOP$_2$ by Boolean combinations (which is one of the results stated in \cite{Tak}).  Finally, we note that there are many potential generalizations of $\textnormal{IFOP}_2$ to $k>2$, since one must make a choice for what it means for a function $f\colon\omega^k\rightarrow \omega$ to be ``increasing."

\section{$\NFOP_k$ and collapsing indiscernibles: Part I}\label{sec:Fraisse}

In this section we show that $\FOP_k$ can be characterized (at the level of formulas) by the ability to code a certain Fra\"{i}ss\'{e} structure (in particular, a $(k+1)$-partite $(k+1)$-uniform hypegraph engineered to be a higher arity generalization of a dense linear order with a dense co-dense predicate). We will also show that this structure satisfies the properties necessary for an application of generalized indiscernibles later in the paper.

\subsection{Preliminaries on generalized indiscernibles}

We begin with some background material on generalized indiscernibles from \cite{Scow1, Scow2}. 

\begin{definition}\label{def:genind}
Let $T$ be a complete $\cL$-theory with monster model $\M$. 
Suppose also that $\cL_0$ is another first-order language and $\cI$ is an $\cL_0$-structure. Let  $(a_i)_{i\in I}$ be a  sequence in $\M^{<\omega}$. 
\begin{enumerate}[$(1)$]
\item  Given $\cL'\seq\cL_0$, we say that $(a_i)_{i\in I}$ is  \textbf{$\cL'$-indiscernible in $\M$} if, for all $n\geq 1$ and $i_1,\ldots, i_n,j_1,\ldots, j_n\in I$, 
\begin{multline*}
\hspace{10pt}\qftp_{\cL'}^{\calI}(i_1,\ldots, i_n)=\qftp_{\cL'}^{\calI}(j_1,\ldots, j_n) \\
\mimp \tp^T_{\cL}(a_{i_1},\ldots,a_{i_n})=\tp^T_{\cL}(a_{j_1},\ldots,a_{j_n}).
\end{multline*}
Note that this implies that for $i,j\in I$, if $\qftp^{\cI}_{\cL'}(i)=\qftp^{\cI}_{\cL'}(j)$ then $|a_i|=|a_j|$.
\item We say that $(a_i)_{i\in I}$ is an \textbf{$\cI$-indexed indiscernible sequence in $\M$} if it is $\cL_0$-indiscernible in $\M$ (for brevity, we often just say  $(a_i)_{i\in I}$ is \textbf{$\cI$-indiscernible}).
\item Suppose $\calJ$ is another $\calL_0$-structure and $(b_j)_{j\in J}$ is another sequence in $\M^{<\omega}$.  We say $(b_j)_{j\in J}$ is \textbf{locally based on $(a_i)_{i\in I}$} if for any finite set $\Delta$ of $\calL$-formulas, any $n\geq 1$, and any $j_1,\ldots, j_n\in J$, there are $i_1,\ldots, i_n\in I$ such that $\qftp_{\calL_0}^{\calJ}(j_1,\ldots, j_n)=\qftp_{\calL_0}^{\calI}(i_1,\ldots, i_n)$ and $\tp^T_{\Delta}(b_{j_1},\ldots, b_{j_n})=\tp^T_{\Delta}(a_{i_1},\ldots, a_{i_n})$.
\end{enumerate}
\end{definition}

\begin{remark}
 The reader may notice that our definition is slightly more general than what appears in \cite{Scow1}, in that we allow our indiscernibles to be tuples of any finite length.  However, the relevant theorem (Theorem \ref{thm:scow} below) is easily extended to this setting, for example by adding explicit new sorts for finite tuples.
\end{remark}

A crucial notion from the theory of generalized indiscernibles is the so-called modeling property.

\begin{definition}\label{def:modeling}
Suppose $\cL_0$ is a first-order language and $\cI$ is an $\cL_0$-structure.  We say  that $\cI$ has the \textbf{modeling property}\footnote{In \cite{Scow2}, this   is defined using the phrase ``$\cI$-indexed indiscernible sets have the modeling property".} if for any $\cL$-theory $T$ (in any language $\cL$) and any sequence $(a_i)_{i\in I}$ in  $\M^{<\omega}$, where $|a_i|$ depends only on the quantifier-free type of $i$, there exists an $\cI$-indexed indiscernible in $\M$ locally based on $(a_i)_{i\in I}$.
\end{definition}

In \cite{Scow1}, Scow proves a characterization of the modeling property in the case of a linearly ordered structure $\cI$ in a finite relational language. This characterization is in terms of the \emph{Ramsey property} for the age of $\cI$ (i.e., the class of all finite structures isomorphic to a substructure of $\cI$).

\begin{theorem}[Scow \cite{Scow1}]\label{thm:scow}
Suppose $\calL_0$ is a finite relational language containing $<$, and $\calI$ is a $\calL_0$-structure linearly ordered by $<$.  Then $\calI$ has the modeling property if and only if the age of $\cI$ has the Ramsey property. 
\end{theorem}

The definition of the Ramsey property is given in Section \ref{sec:ramsey}. 
 In fact, Scow has proved more general results along these lines (see \cite{Scow2}), however the theorem above suffices for our purposes.

The modeling property plays a crucial role in determining whether a given dividing line  can be characterized by collapsing indiscernibles via some indexing structure (depending only on the dividing line).  As is often the case with this type of result, we will use Theorem \ref{thm:scow}, together with the vast literature on Ramsey classes, to deduce that our desired index structure has the modeling property.

\subsection{Our index structure}\label{sec:Hk}

In this subsection, we define the class of finite structures whose Fra\"{i}ss\'{e} limit  will be used as the index model in our collapsing indiscernibles result. The definition of the class, which we will denote by $\bH_k$,  is motivated by characterization $(ii)$ of $\FOP_k$ in Proposition \ref{prop:alleq}. In particular, we start with a half-graph and  replace the linear order on one side by a linear order on a Cartesian product of $k$ sets. 

The proof that $\bH_k$ is a Fra\"{i}ss\'{e} class will be done in a more general setting, which will  clarify aspects of the argument, and is also of independent interest. The reader who does not wish to delve into these details, and is willing to take our word that $\bH_k$ is a Fra\"{i}ss\'{e} class, can safely ignore this general construction. In this case, the reader should skip to the definition of the language $\cL_k$  given just before Definition \ref{def:Tk}, and then take Definition \ref{def:Tk} and Proposition \ref{prop:Hkmod} as the definition of $\bH_k$. 

We now formulate a general way to perform the above-mentioned process of ``blowing up" a predicate into a Cartesian product to get a new Fra\"{i}ss\'{e} class from a base Fra\"{i}ss\'{e} class satisfying mild assumptions. We start with some notation.

Given a $t$-ary relation symbol $R(x_1,\ldots,x_t)$ and some $I\seq[t]$, we define a relation $R_I(\xbar_1,\ldots,\xbar_t)$ were $\xbar_i=x_i$ for each $i\not\in I$ and $\xbar_i=(x_{i,1},\ldots,x_{i,k})$ for each $i\in I$. In particular, $R_I$ is a new relation symbol with arity $k|I|+(t-|I|)$. Let $J$ denote the complement $[t]\backslash I$. Given sequences $(a_i)_{i\in I}$ and $(b_i)_{i\in J}$ of singletons, we let $a_I\oplus b_J$ denote the sequence $(c_i)_{i=1}^t$ where $c_i=a_i$ if $i\in I$ and $c_i=b_i$ if $i\in J$ (so $R(a_I\oplus b_J)$ is a well-defined statement). Analogously, if $\abar_I=(\abar_i)_{i\in I}$ is a tuple of $k$-tuples, we write $\abar_I\oplus b_J$ for the tuple $(\cbar_i)_{i=1}^{t}$ where $\cbar_i=\abar_i$ if $i\in I$ and $\cbar_i=b_i$ if $i\in J$ (so $R_I(\abar_I\oplus b_J)$ is a well-defined statement).

For the rest of this subsection, $\calL$ is a finite relational language containing a binary relation $<$ and a unary relation $U$.

\begin{definition}
Define 
$$
\calL_k(U)=\{P_1,\ldots, P_{k+1}, <\}\cup \{R_I: R(x_1,\ldots, x_t)\in \calL\backslash\{U\},~ I\subseteq[t]\},
$$
where each $P_i$ is a unary predicate, and each $R_I$ is as defined above.
\end{definition}

Our next definition shows us how to obtain an $\calL_k(U)$ structure from an $\calL$-structure.  Given sets $X$, $Y$, a function $f:X\rightarrow Y$, and a tuple $\xbar\in X^k$, let $f(\xbar)=(f(x_1),\ldots, f(x_k))$.  We will also use $Q$ for the $k$-ary relation $P_1\times\cdots\times P_k$. 
 
\begin{definition}\label{def:blowup}
 Suppose $\cA$ is an $\cL$-structure, $\calM$ is an $\calL_k(U)$-structure, and $f$ is a bijection from $Q(\calM)$ to $U(\calA)$.  We say $\calM$ is {\bf obtained from $\calA$ via $f$} if the following hold:
\begin{enumerate}[$(1)$]
\item $P_{k+1}(\calM)=A\backslash U(\calA)$.
\item $P_1(\calM), \ldots , P_{k+1}(\calM)$ partition\footnote{We allow a part to be empty in a partition.} $M$.
\item Given $R(x_1,\ldots,x_t)$ in $\cL\backslash\{U\}$ and $I\seq [t]$ with $J=[t]\backslash I$, if $(\abar_i)_{i\in I}$ is a sequence from $M^k$ and $(b_i)_{i\in J}$ is a sequence from $M$, then $\cM\models R_I(\abar_I\oplus b_J)$ if and only if each $\abar_i\in Q(\cM)$, each $b_j\in P_{k+1}(\cM)$, and $\cA\models R(f(\abar_I)\oplus b_J)$. 
\item $<$ is a linear order on $M$ whose restriction to $P_{k+1}(\calM)$ agrees with $<_{\emptyset}$, and which satisfies $P_1(\calM)<\ldots<P_{k+1}(\calM)$.
\end{enumerate}
\end{definition}

In the previous definition, we think of $\cM$ as being obtained from  $\cA$ by ``blowing up'' the predicate $U$ into a $k$-dimensional Cartesian product. In order for this notion to work well with the properties of Fra\"{i}ss\'{e} classes, we will also need to apply the previous definition with $\cA$ replaced by the empty structure. For the sake of clarity, we make this explicit with the next definition.

\begin{definition}
We say that an $\calL_k(U)$-structure $\calM$ is {\bf trivially obtained} if:
\begin{enumerate}
\item $P_{k+1}(\calM)=\emptyset$ and $Q(\calM)=\emptyset$. 
\item $P_1(\calM),\ldots,P_k(\calM)$ partition $M$.
\item Given $R(x_1,\ldots,x_t)$ in $\cL\backslash\{U\}$ and $I\seq [t]$, $R_I$ does not hold on any tuples in $\calM$. 
\item $<$ is a linear order on $M$  which satisfies $P_1(\calM)<\ldots<P_k(\calM)$.
\end{enumerate} 
\end{definition}

The following fact shows that our blow-up operation behaves well under substructures and isomorphisms.

\begin{fact}\label{fact:HP}$~$
\begin{enumerate}[$(a)$] 
\item Suppose $\calN$ is an $\calL_k(U)$-structure  obtained from an $\cL$-structure $\calB$ via $g$.  Then for any $\calM\leq \calN$, either $\cM$ is trivially obtained or there is some  $\calA\leq \calB$ so that $\calM$ is obtained from $\calA$ via $g|_{\calM}$.
\item Suppose $\calM$ is an $\calL_k(U)$-structure  obtained from an $\cL$-structure $\calA$ via $f$, and also from an $\cL$-stucture $\calB$ via $g$.  Then $\calA$ and $\calB$ are isomorphic.
\end{enumerate}
\end{fact}
\begin{proof}
    Part $(a)$. Using a straightforward verification, one sees that we may choose $\calA$ to be the substructure of $\calB$ with domain $(B\cap P_{k+1}(\calM))\cup (U(\calB)\cap g(Q(\calM)))$\, which is not empty when $\cM$ is not trivially obtained. 

    Part $(b)$. By definition, $A\backslash U(\calA)=P_{k+1}(\calM)=B\backslash U(\calB)$. Define $h\colon A\to B$ so that $h{\upharpoonright}A\backslash U(\cA)$ is the identity (to $B\backslash U(\cB)$) and $h{\upharpoonright}U(\cA)$ is $g\circ f\inv$. Note that $h$ is a bijection.
 We show it is an isomorphism.  Fix a relation $R(x_1,\ldots, x_t)$ of $\calL$ and $\abar\in A^k$.  Let $I=\{i\in [t]: a_i\in U(\calA)\}$ and $J=[t]\setminus I$. Let $\cbar,\dbar$ be such that $\abar=\cbar_I\oplus \dbar_J$. By definition, $\calA\models R(\abar)$ if and only if $\calM\models R_I(f^{-1}(\cbar_I)\oplus \dbar_J)$ if and only if $\calB\models R(g(f^{-1}(\cbar_I))\oplus \dbar_J)$. By definition of $g\circ f^{-1}$, this holds if and only if $\calB\models R(g\circ f^{-1}(\abar))$, as desired.
\end{proof}

\begin{definition}\label{def:fraisseop}
Given a class $\bK$ of $\cL$-structures, define $\bK_k(U)$ to be the class of all $\cL_k(U)$-structures $\cM$ that are either trivially obtained or obtained from an $\cL$-structure $\cA\in \bK$ via some $f$.
\end{definition}

\begin{proposition} \label{prop:fraisseop}
Let $\calL$ be a finite relational language containing a binary relation $<$ and a unary relation $U$. Suppose $\bK$ is a Fra\"{i}ss\'{e} class of finite $\cL$-structures with disjoint amalgamation,   whose elements are linearly ordered by $<$. Moreover, suppose that given any $\calA\in\bK$ and integer $n$, there is a $\calB\in\bK$ with $\calA\leq \calB$ and $|U(\calB)|\geq n$. Then $\bK_k(U)$ is a Fra\"{i}ss\'{e} class.
\end{proposition}
\begin{proof}
First, note that the Hereditary Property is immediate from Fact \ref{fact:HP}$(a)$ above.

\textit{Amalgamation Property:} Suppose $\calM,\calN\in \bK_k(U)$ with $\calD\leq \calM$ and $\calD\leq \calN$, and assume $M\cap N = D$, after a possible relabeling. We will assume that none of these structures is trivially obtained, since in that case the proof is simpler. Let $\calA$, $f$ and $\calB$, $g$ be so that $\calM$ is obtained from $\calA$ via $f$ and $\calN$ is obtained from $\calB$ via $g$.  By Fact \ref{fact:HP}$(a)$, there are  $\calA'\leq \calA$ so that $\calD$ is obtained from $\calA'$ via $f|_{\calD}$ and $\calB'\leq \calB$ so that $\calD$ is obtained from $\calB'$ via $g|_{\calD}$. Thus by Fact \ref{fact:HP}$(b)$, $\calA'\cong \calB'$.  Without loss of generality, we can assume $\calA'=\calB'$, and we call this substructure $\calD'$.  So $\calD$ is obtained from $\calD'$ via $h:=f|_{\calD}=g|_{\calD}$. We will define the amalgamation of $\calM, \calN$ over $\calD$, which we will call $\calC$, as follows.  

First, for each $i\in [k]$, we define the sets 
$$
C_i\coloneqq P_i(\calD)\sqcup (P_i(\calM)\backslash P_i(\calD))\sqcup (P_i(\calN)\backslash P_i(\calD)),
$$
and let $n_i\coloneqq |C_i|$.  Set $n \coloneqq n_1\cdots n_k$, so $n=|C_1\times \ldots \times C_k|$. 

By the disjoint amalgamation property in $\bK$, there is an $\calE\in \bK$ and embeddings $h_1\colon\calA\rightarrow \calE$ and $h_2\colon\calB\rightarrow \calE$ such that for all $d\in \calD'$, we have  $h_1(d)=h_2(d)$, and further, $h_1(A\backslash D')\cap h_2(B\backslash D')=\emptyset$.  By assumption, there exists $\calE'\in \bK$ so that $\calE\leq \calE'$ and $|U(\calE')|\geq n$.  Since $\calL$ is finite relational and $\bK$ is hereditary, we may assume $|U(\calE')|=n$.

Set $C_{k+1}\coloneqq E'\backslash U(\calE')$ and $C\coloneqq C_1\sqcup\ldots \sqcup C_{k+1}$.   We define an injection 
$$
w'\colon Q(\cM)\cup Q(\cN)\to U(\calE')
$$
by letting $w'(\mbar)=h_1(f(\mbar))$ for $\mbar\in Q(\calM)$, and letting $w'(\nbar)=h_2(g(\nbar))$ for $\nbar\in Q(\calN)$. Then $w'$ well-defined because for $\dbar\in Q(\calD)$ we have that $h_1(f(\dbar))=h_1(h(\dbar))=h_1(g(\dbar))=h_2(g(\dbar))$. Moreover, $w'$ is injective since $f$, $g$, $h_1$, and $h_2$ are all injective, and $h_1(A\backslash D')\cap h_2(B\backslash D')=\emptyset$. As $|C_1\times \ldots \times C_k|=n=|U(\calE')|$, we can extend $w'$ to a bijection $w\colon C_1\times \ldots \times C_k\rightarrow U(\calE')$. 

We now define an $\calL_k(U)$-structure $\calC$ so that $\calC$ will be obtained from from $\calE$ via $w$. We let the domain of $\calC$ be $C$, and for each $i\in [k+1]$, we interpret $P_i(\calC)=C_i$.  By above, $w$ is then a bijection from $Q(\calC)$ to $U(\calE')$. Given $R(x_1,\ldots, x_t)$ in $\calL$, $I\subseteq[t]$, $(\cbar_i)_{i\in I}\subseteq Q(\calC)$, and $(c_j')_{j\in J}\subseteq P_{k+1}(\calC)=E'\backslash P(\calE')$, we define $R^\calC_I(\cbar_I\oplus c'_J)$ to hold if and only if $R^{\calE'}(w(\cbar_I)\oplus c'_J)$ holds.  We then define $<^\calC$ as follows.  First, define $<^{\calC}$ on $P_{k+1}(\calC)$ to coincide with $<^\calC_\emptyset$. Given $i\in [k]$, it is straightforward to see that the relation ${\prec_i}={{<^{\calM}}{\upharpoonright}{P_i(\calM)}}\cup {{<^{\calN}}{\upharpoonright}{P_i(\calN)}}$ is a partial order on $P_i(\calC)$, so we can define  ${<^{\calC}}{\upharpoonright}{P_i(\calC)}$ to be any linear order extending $\prec_i$.  Finally, define  $P_1(\calC)<^\calC\cdots<^\calC P_{k+1}(\calC)$.  This finishes our definition of $\calC$. Observe that $\calC$ is obtained from $\calE'$ via $w$ by construction, and so $\calC\in \bK_k(U)$.

Define  $h_1'\colon \calM\rightarrow \calC$ to be the identity on $P_1(\calM)\cup \ldots \cup P_k(\calM)$ and $h_1$ on $P_{k+1}(\calM)$.  We show $h_1'$ is an embedding.  It is clear $h_1'$ is an injection.  By construction, $h_1'(P_i(\calM))=P_i(\calM)=P_i(\calC)\cap M$ for $i\in [k]$.  Also by construction, $h_1'(P_{k+1}(\calM))=h_1(A\backslash U(\calA))=h_1(\calM)\cap P_{k+1}(\calC)$.  The map $h_1'$ preserves $<$ by definition of $\calC$. 

Now fix an $\calL$-relation $R(x_1,\ldots, x_t)$ and an $I\subseteq [t]$.  Suppose $\mbar_I=(\mbar_i)_{i\in I}\subseteq Q(\calM)$ and $m_J'=(m'_j)_{j\in J}\subseteq P_{k+1}(\calM)$. We have that $\calM\models R_I(\mbar_I\oplus m'_J)$ if and only if $\calA\models R(f(\mbar_I)\oplus m'_J)$.  Since $h_1$ is an embedding, we have that $\calA\models R(f(\mbar_I)\oplus m'_J)$ if and only if $\calE'\models R(h_1(f(\mbar_I))\oplus h_1(m'_J))$. Because $w=h_1\circ f$ on $Q(\calM)$, this holds  if and only if $\calC\models R_I(\mbar_I\oplus h_1(m'_J))$, i.e., if and only if $\calC\models R_I(h_1'(\mbar_I\oplus m'_J))$.  This shows $h_1'$ is an embedding of $\calM$ into $\calC$. 

We can similarly define  $h_2'\colon \calN\rightarrow \calC$ to be the identity on $P_1(\calN)\cup \ldots \cup P_k(\calN)$ and $h_2$ on $P_{k+1}(\calN)$. The same argument as above shows $h_2'$ is an embedding of $\calN$ into $\calC$, and it straightforward to check that $h_1',h_2'$ these agree on $\calD$ by construction.

\textit{Joint Embedding Property:} It is an exercise to show that $\bK$ must satisfy a disjoint version of JEP, i.e., that for any two $\calA,\calB\in \bK$, there is $\calC\in \bK$ and embeddings $h_1\colon \calA\rightarrow \calC$ and $h_2\colon \calB\rightarrow \calC$ with disjoint images. Using this, one can mimic the proof above with $\calD$ being empty to show $\bK_k(U)$ has JEP. \footnote{We follow the convention that all structures are nonempty. If we allow for the empty structure then it is in our class and we are done by the AP.}
\end{proof}

Now we are ready to define the class which we will use to obtain the index model for our generalized indiscernibles. 
\begin{definition}\label{def:Hk}
    Let $\calL=\{U, <\}$ where $U$ is a unary predicate and $<$ is a binary relation. Define $\bH_k$ to be $\bK_k(U)$, where $\bK$ is the class of finite $\calL$-structures in which $<$ is a linear order (and $U$ is arbitrary).  
\end{definition}

In order to make the class $\bH_k$ more transparent, we will now give an explicit axiomatization of its universal theory. First, let us make some simplifications to the language and notation in this special case. With $\cL=\{U,<\}$ as in the previous definition, we have  
\[
\calL_k(U)=\{P_1,\ldots, P_{k+1}, <, <_{\emptyset}, <_{\{1\}}, <_{\{2\}},<_{[2]} \}.
\]
For any structure in $\bH_k$, $<_{\emptyset}$ only holds on $P_{k+1}$, where it coincides with $<$. Moreover, $<_{\{1\}}$ is a $(k+1)$-ary relation on $P_1\times\ldots\times P_{k+1}$, and $<_{\{2\}}$ is simply the negation of $<_{\{1\}}$. Finally, $<_{[2]}$ is a linear order order on $P_1\times\ldots\times P_k$. So to simplify notation, we will discard the symbols $<_\emptyset$ and $<_{\{2\}}$, and rename $<_{\{1\}}$ as $R$ and $<_{[2]}$ as $<_k$. Altogether, we denote the resulting simplified language by
\[
\cL_k=\{P_1,\ldots,P_{k+1},<,<_k,R\},
\]
where $P_1,\ldots, P_{k+1}$ are unary relations, $<$ is a binary relation, $<_k$ is a $2k$-ary relation, and $R$ is a $(k+1)$-ary relation. 
From this point on, we view elements of $\bH_k$ as $\cL_k$-structures, and we continue to let $Q$ denote the $k$-ary relation $P_1\times\ldots\times P_k$.

\begin{definition}\label{def:Tk}
 Define $T_k$ to be the $\cL_k$-theory consisting of the following axioms:
\begin{enumerate}[$(1)$]
\item $P_1,\ldots, P_{k+1}$ is a partition.
\item $<$ is a linear order with $P_1<\ldots< P_{k+1}$.
\item $R$ only holds on $P_1\times\ldots\times P_{k+1}$ (which we also view as $Q\times P_{k+1}$).
\item $<_k$ only holds on $Q\times Q$, and is a linear order on $Q$. 
\item For any $\xbar,\ybar\in Q$ and $w,z\in P_{k+1}$, $\big(\xbar\leq_k\ybar \wedge R(\ybar,w)\wedge w\leq z\big)\rightarrow R(\xbar,z)$.
\end{enumerate}
\end{definition}

We note in passing that axiom (5) is a bipartite analogue of the notion of a ``monotone binary relation" on a linearly ordered set (see, e.g., Section 4 of \cite{SiDP}).

\begin{remark}\label{rem:Tkorder}
For any $\cA\models T_k$, the  relation ${<_*}\coloneqq {<_k}\cup R\cup{(\neg R)^{\textnormal{opp}}}\cup {({<}{\upharpoonright}P_{k+1})}$ is a linear order on $Q(\cA)\cup P_{k+1}(\cA)$. 
\end{remark}

\begin{proposition}\label{prop:Hkmod}
    $\bH_k$ is the class of finite models of $T_k$. 
\end{proposition} 
\begin{proof}
Let $\calM \in \bH_k=\bK_k(U)$. If it is trivially obtained, it is clear that it satisfies $T_k$. Say it is obtained from $\calA\in \bK$ via $f$. It is immediate by definition that $\cM$ satisfies axioms (1)-(3) and the first half of (4). That $<_k$ is a linear order on $Q$ holds because ${<_k}={<_{[2]}}$ and $<$ is a linear order on $U(\calA)$. Lastly, given $\xbar, \ybar \in Q$ and $w,z\in P_{k+1}$, $\calM\models \big(\xbar\leq_k\ybar \wedge R(\ybar,w)\wedge w\leq z\big)$ implies $\calA\models f(\xbar)<f(\ybar) <w< z$, which implies $\calA\models f(\xbar)<w$. This implies $\calM\models R(\xbar,w)$ as desired.

Conversely, given a model $\calM\models T_k$, if $Q(\cM)\cup P_{k+1}(\cM) = \emptyset$ we can see that $\calM$ is trivially obtained. Otherwise, we get an $\cL$-structure $\cA$ with domain $Q(\cM)\cup P_{k+1}(\cM)$, where $U(\cA)$ is interpreted as $Q(\cM)$ and $<$ is the linear order $<_*$ from Remark \ref{rem:Tkorder}. By construction, $\cM$ is obtained from $\cA$ via the identity from $Q(\cM)$ to $U(\cA)$.
\end{proof}

Now, it is a well-known folklore fact that the class $\bK$ in Definition \ref{def:Hk} is a Fra\"{i}ss\'{e} class (e.g., see the discussion on page 8 of \cite{TorrTruss}), and it clearly satisfies the extra assumptions in Proposition \ref{prop:fraisseop}. So we obtain the following conclusion.

 \begin{corollary}
$\bH_k$ is a Fra\"{i}ss\'{e} class.
\end{corollary}

\begin{definition}\label{def:HkFL}
Let $\cH_k$ be the Fra\"{i}ss\'{e} limit of $\bH_k$.
\end{definition}

In order to use $\cH_k$ in our characterization of $\FOP_k$ via collapse of indiscernibles, we need to know that this structure has the modeling property (Definition \ref{def:modeling}). Recall that by Theorem \ref{thm:scow}, it suffices to show that $\bH_k$ has the Ramsey property. This will follow from some very general theorems of Hubi\v{c}ka and  Ne\v{s}et\v{r}il \cite{HN}, and Evans, Hubi\v{c}ka, and  Ne\v{s}et\v{r}il \cite{EHN}. In Section \ref{sec:ramsey},  we will state these results and use them to establish the Ramsey property for $\bH_k$. For now, we just state the following corollary, which is all we need for the main body of our work.

\begin{corollary}\label{cor:Gkmod}
$\cH_k$ has the modeling property.
\end{corollary}

\begin{remark}\label{rem:ordextra}
	We now note that we could have defined the order $<$ in our Fra\"{i}ss\'{e} class $\bH_k$ and Fra\"{i}ss\'{e} limit $\cH_k$ to hold only on $P_{k+1}$, and still get a Fra\"{i}ss\'{e} class and limit which allow us to characterize $\FOP_k$ as in Proposition \ref{prop:codeTk}. This essentially says that the order $<_*$ from Remark \ref{rem:Tkorder} encodes all the information we need. However, to be able to use Theorem \ref{thm:scow}, we need our structures to be linearly ordered, which explains why we require $<$ to be a linear order. 
\end{remark}

\subsection{Witnessing $\FOP_k$ with $\cH_k$}

In this subsection, we establish the key connection between $\cH_k$ and $\FOP_k$  via the characterization in Proposition \ref{prop:alleq}$(ii)$. We again work with a complete $\cL$-theory $T$ with monster model $\M$.

\begin{definition}
Given $\cH\models T_k$, we say that an $\cL$-formula $\varphi(x_1,\ldots,x_{k+1})$ \textbf{$R$-codes $\cH$} if there is a sequence $(a_h)_{h\in H}$ from $\M$ such that 
\begin{enumerate}[$(i)$]
\item for all $1\leq i\leq k+1$, if $h\in P_i(\cH)$ then $a_h\in\M^{x_i}$, and
\item for any $(h_1,\ldots,h_{k+1})\in P_1(\cH)\times\ldots\times P_{k+1}(\cH)$, we have $\M\models\varphi(a_{h_1},\ldots,a_{h_{k+1}})$ if and only if $\cH\models R(h_1,\ldots,h_{k+1})$. 
\end{enumerate}
In this case, we also say that $(a_h)_{h\in\cH}$ \textbf{witnesses} that $\varphi(\xbar)$ $R$-codes $\cH$.
\end{definition}

\begin{proposition}\label{prop:codeTk}
Given an $\cL$-formula $\varphi(x_1,\ldots,x_{k+1})$, the following are equivalent.
\begin{enumerate}[$(i)$]
\item $\varphi(\xbar)$ has $\FOP_k$ in $T$.
\item $\varphi(\xbar)$ $R$-codes every finite model of $T_k$.
\item $\varphi(\xbar)$ $R$-codes every small model of $T_k$.
\item $\varphi(\xbar)$ $R$-codes $\cH_k$.
\item $\varphi(\xbar)$ $R$-codes $\cH_k$ witnessed by an $\cH_k$-indexed indiscernible sequence in $\M$.
\end{enumerate}
\end{proposition}
\begin{proof}
First we discuss the implications that are easy or immediate from the results above. Since $T_k$ is a universal theory in a  relational language the equivalence of $(ii)$ and $(iii)$ follows from saturation of $\M$ and standard compactness arguments. Moreover, the equivalence of $(iv)$ and $(v)$ follows from Corollary \ref{cor:Gkmod}. Finally, $(iii)\Rightarrow (iv)$ is trivial, and $(iv)\Rightarrow (ii)$ is immediate since $R$-coding is  hereditary. Altogether, to complete the equivalences, it suffices to show $(iii)\Rightarrow (i)\Rightarrow (ii)$.

$(iii)\Rightarrow (i)$. Assume $(iii)$. We prove $(i)$ via Proposition \ref{prop:alleq}$(ii)$. Fix a linear order $<_*$ on $\omega^k\cup\omega$. Define an $\cL_k$-structure $\cH$ with domain $\omega^k\cup\omega$ as follows.
\begin{enumerate}[\hspace{5pt}$\ast$] 
\item For $1\leq i\leq k$, $(P_i(\cH),<)$ is a copy of $(\omega,<)$.
\item $(P_{k+1}(\cH),<)$ is $(\omega,{<_*}{\upharpoonright}\omega)$.
\item  $P_1(\cH)<\ldots< P_{k+1}(\cH)$.
\item $<_k$ is ${<_*}{\upharpoonright}\omega^k$.
\item $R$ is ${<_*}{\upharpoonright}(\omega^k\times\omega)$.
\end{enumerate}
Then $\cH\models T_k$. So $\varphi(x_1,\ldots,x_{k+1})$ $R$-codes $\cH$ by $(iii)$, as desired.

$(i)\Rightarrow (ii)$. Assume $(i)$ and fix a finite model $\cH\models T_k$. Without loss of generality, assume that for each $i\in [k+1]$, $(P_i(\cH),<)$ is an initial segment of a copy of $(\omega,<)$. With this identification, let $<_*$ be any linear order on $\omega^k\cup\omega$ extending the order on $Q(\cH)\cup P_{k+1}(\cH)$ given by Remark \ref{rem:Tkorder}. Now apply Proposition \ref{prop:alleq}[$(i)\Rightarrow(ii)$] to $<_*$, to obtain  $(a^1_i)_{i<\omega}$, \ldots, $(a^{k+1}_i)_{i<\omega}$  such that $\M\models\varphi(a^1_{i_1},\ldots,a^{k+1}_{i_{k+1}})$ if and only if $(i_1,\ldots,i_k)<_* i_{k+1}$. By construction the initial segments of these sequences corresponding to $P_1(\cH),\ldots,P_{k+1}(\cH)$ witness that $\varphi(\xbar)$ $R$-codes $\cH$. 
 \end{proof}

From this, we obtain one direction of our main result on characterizing $\FOP_k$ via indiscernible collapse.

\begin{definition}
Let $\cL'_k=\cL_k\backslash\{R\}$.
\end{definition}

\begin{corollary}\label{cor:genind}
Let $T$ be a complete $\cL$-theory with monster model $\M$. If every $\cH_k$-indiscernible set in $\M$ is $\cL'_k$-indiscernible, then $T$ is $\NFOP_k$.
\end{corollary}
\begin{proof}
Suppose $T$ has $\FOP_k$. By Proposition \ref{prop:codeTk} we have an $\cH_k$-indiscernible sequence $(a_g)_{g\in \cH_k}$ in $\M$ witnessing that some formula $\varphi(x_1,\ldots,x_{k+1})$ $R$-codes $\cH_k$. Fix $\gbar,\gbar'\in Q(\cH_k)$ and $h\in P_{k+1}(\cH_k)$ such that $\cH_k\models R(\gbar,h)\wedge\neg R(\gbar',h)$. So $(a_g)_{g\in H_k}$ is not $\cL'_k$-indiscernible since $\qftp^{\cH_k}_{\cL'_k}(g_1,\ldots,g_k,h)=\qftp^{\cH_k}_{\cL'_k}(g'_1,\ldots,g'_k,h)$, but $\M\models\varphi(a_{g_1},\ldots,a_{g_k},a_h)\wedge\neg\varphi(a_{g'_1},\ldots,a_{g'_k},a_h)$.
\end{proof}
 
In the next section we will prove the converse of the previous result. We finish this subsection with some comments about the Fra\"{i}ss\'{e} limit $\cH_k$.

\begin{remark}\label{rem:Gkconditions}$~$
\begin{enumerate}[$(1)$]
\item The structure on $\cH_k$, other than $<$ on $P_i$ for $1\leq i<k+1$, is definable from $R$ (using quantifiers). 
\begin{enumerate}[$\ast$]
	\item For $1\leq i\leq k+1$, the  predicate $P_i(x_i)$ is defined by $\exists x_{<i}\exists x_{>i}R(x_1,\ldots,x_{k+1})$. 
	\item The ordering $\xbar <_k\ybar$ on $Q$ is defined by $\exists z(R(\xbar,z)\wedge\neg R(\ybar,z))$.
	\item The ordering $x< y$ on $P_{k+1}$ is defined by $\exists z_1\ldots z_k(R(\zbar,y)\wedge \neg R(\zbar,x))$.
\end{enumerate}
Moreover, we thank the anonymous referee for noting that this part of the structure can also be defined using only universal quantifiers. 
\begin{enumerate}[$\ast$]
	\item For $1\leq i\leq k+1$, $P_i(x)$ is defined by $$\forall y_1\cdots y_k\bigwedge_{\ell=1,\ldots, k+1; \ell\neq i}\lnot R(y_1,\ldots y_{\ell-1},x,y_{\ell}, \ldots,y_{k}).$$ 
	\item The ordering $\xbar <_k\ybar$ on $Q$ is defined by $\xbar\neq \ybar \land \forall z(R(\ybar,z)\to R(\xbar,z))$.
	\item The ordering $x< y$ on $P_{k+1}$ is defined by $x\neq y\land \forall \zbar(R(\zbar,x)\to R(\zbar,y))$.
\end{enumerate}
Along with quantifier elimination, this gives us model completeness of the theory of the modified Fra\"{i}ss\'{e} limit  discussed in Remark \ref{rem:ordextra} (in which $<$ is restricted to $P_{k+1}$).

\item Since $\cH_k$ is countable, we can identify each $P_i(\cH_k)$ with a copy of $\omega$. Thus $Q(\cH_k)\cup P_{k+1}(\cH_k)$ is identified with $\omega^k\cup\omega$ and we obtain a ``canonical/generic" linear order $<_*$ on $\omega^k\cup\omega$ given by Remark \ref{rem:Tkorder}. By Proposition \ref{prop:codeTk}, we then obtain a variation of condition $(ii)$ in Proposition \ref{prop:alleq} in which one only needs to find sequences for this ``canonical/generic" linear order on $\omega^k\cup\omega$. 

\item For any $s<\omega$, we can also use $\cH_k$ to obtain a canonical/generic partition $\omega^k= E_1\cup\ldots \cup E_s$.  Let $\Q_1,\Q_2$ be a dense/co-dense partition of $\Q$, and let $Q(\calH_k)=\{e_q:q\in \Q_1\}$, $P_{k+1}(\calH_k)=\{x_q: q\in \Q_2\}$ be  enumerations so that $\calH_k\models R(e_q,x_{q'})$ if and only if $q<q'$.  Choose any $r_1<\ldots<r_{s-1}\in \Q_2$. Now let $I_1=\Q_1\cap (-\infty,r_1)$, $I_s=\Q_1\cap (r_{s-1},\infty)$, and for each $2\leq j\leq s-1$, let $I_j=\Q_1\cap (r_{j-1},r_j)$.  Setting $E_j=\{e_q: q\in I_j\}$ for each $1\leq j\leq s$, we obtain a partition $\omega^k=E_1\cup \ldots \cup E_s$ in which the edge colors $E_1,\ldots, E_s$ behave ``generically." 
Note then that for $y_1=x_{r_1},\ldots, y_s=x_{r_s}$ we have that $\calH_k\models R(e_q,y_s)$ if and only if $q\in \bigcup_{j<s}E_j$.  We can thus obtain a variant on  Proposition \ref{prop:alleq}$(iii)$ which says such $y_1,\ldots, y_s$ exist with respect to such a single ``canonical/generic" partition $E_1\cup \ldots \cup E_s=\omega^k$. We note that this variant would also differ in that a strict inequality appears in $\bigcup_{j<s}E_j$. This is purely a cosmetic issue arising from our choice of axioms for $\calH_k$, and can be done away with via a compactness argument. 
\end{enumerate}
\end{remark}

\subsection{First applications and examples}\label{sec:examples}

As a quick application of the previous results, we consider the question of preservation of $\NFOP_k$ under Boolean combinations. 
Recall that closure of $\NFOP_k$ formulas under negation was observed by the third author and Wolf in \cite{TW2} (see  Corollary \ref{cor:easyFOPk}). It was also shown in \cite{TW2} that $\NFOP_2$ formulas are closed under arbitrary Boolean combinations. However, unlike negations, the proof for disjunctions was hard, and relied on results in hypergraph regularity only currently available in the $k=2$ case. In this subsection we use generalized indiscernibles to give a much shorter proof for disjunctions, which  works for arbitrary $k$.

\begin{theorem}\label{thm:FOPkbool}
$\NFOP_k$ formulas are closed under Boolean combinations.  In particular, for any complete theory $T$, if  $\varphi(x_1,\ldots,x_{k+1})$ and $\psi(x_1,\ldots,x_{k+1})$ are $\NFOP_k$ in $T$, then so are $\neg\varphi(x_1,\ldots,x_{k+1})$ and  $\varphi(x_1,\ldots, x_{k+1})\vee \psi(x_1,\ldots, x_{k+1})$.
\end{theorem}
\begin{proof}
We only need to deal with the disjunction case. Suppose $(\varphi\vee\psi)(\xbar)$ has $\FOP_k$ in $T$. By Proposition \ref{prop:codeTk}, we have an $\cH_k$-indiscernible sequence $(a_g)_{g\in H_k}$ in $\M$ witnessing that $(\varphi\vee\psi)(\xbar)$ $R$-codes $\cH_k$. Thus, for any $\gbar\in P_1(\cH_k)\times\ldots\times P_{k+1}(\cH_k)$, if $\cH_k\models\neg R(\gbar)$ then $\M\models \neg\varphi(a_{g_1},\ldots,a_{g_{k+1}})\wedge \neg\psi(a_{g_1},\ldots,a_{g_{k+1}})$.

Now fix some $\bar{h}\in P_1(\cH_k)\times\ldots\times P_{k+1}(\cH_k)$ such that $\cH_k\models R(\bar{h})$. Then 
\[
\M\models (\varphi\vee\psi)(a_{h_1},\ldots,a_{h_{k+1}}),
\]
so without loss of generality assume $\M\models \varphi(a_{h_1},\ldots,a_{h_{k+1}})$. For any other $\gbar\in P_1(\cH_k)\times\ldots\times P_{k+1}(\cH_k)$, if $\cH_k\models R(\gbar)$ then $\qftp^{\cH_k}_{\cL_k}(\gbar)=\qftp^{\cH_k}_{\cL_k}(\bar{h})$, and so  by $\cH_k$-indiscernibility of $(a_g)_{g\in H_k}$, we get $\M\models \varphi(a_{g_1},\ldots,a_{g_{k+1}})$. Altogether $\varphi(\xbar)$ $R$-codes $\cH_k$ using $(a_g)_{g\in H_k}$, and thus has $\FOP_k$ by Proposition \ref{prop:codeTk}.  
\end{proof}

To our knowledge, the previous theorem is the first instance in which generalized indiscernibles are used to prove closure under Boolean combinations for a particular dividing line. So we point out that the proof method is quite flexible. For example, a similar argument would work to prove that $\NIP_k$ formulas are closed under Boolean combinations (which was shown by Chernikov, Palac\'{i}n, and Takeuchi \cite{CPT} using a higher-arity version of the Sauer-Shelah lemma).

Next we discuss some examples of theories with and without $\FOP_k$, which demonstrate various hallmark properties.

\begin{proposition}\label{prop:arity}
Suppose $\cL$ is a relational language of arity at most $k$, and $T$ is an $\cL$-theory with quantifier elimination. Then $T$ is $\NFOP_k$. 
\end{proposition}
\begin{proof}
Suppose $(a_g)_{g\in H_k}$ is an $\cH_k$-indexed indiscernible sequence in $\M\models T$.  We will show $(a_g)_{g\in H_k}$ is $\cL'_k$-indiscernible, and thus $T$ is $\NFOP_k$ by Corollary \ref{cor:genind}. Fix arbitrary $g_1,\ldots, g_n,h_1,\ldots,h_n\in \cH_k$ with $\qftp^{\cH_k}_{\calL'_k}(g_1,\ldots, g_n)=\qftp^{\cH_k}_{\calL'_k}(h_1,\ldots, h_n)$. We need to show $\tp^T_{\cL}(a_{g_1},\ldots, a_{g_n})=\tp^T_{\cL}(a_{h_1},\ldots, a_{h_n})$. By quantifier elimination for $T$, and the assumption on $\cL$, the type of a tuple in $\M$ is completely determined by its subtuples of length $k$. So it suffices to 
 fix $1\leq i_1\leq \ldots\leq i_k\leq n$, and show $\tp^T_{\cL}(a_{g_{i_1}},\ldots, a_{g_{i_k}})=\tp^T_{\cL}(a_{h_{i_1}},\ldots, a_{h_{i_k}})$. For this, it suffices by assumption on $(a_g)_{g\in H_k}$, to show $\qftp^{\cH_k}_{\calL_k}(g_{i_1},\ldots, g_{i_k})=\qftp^{\cH_k}_{\calL_k}(h_{i_1},\ldots, h_{i_k})$. But since $R$ is a $(k+1)$-ary relation that only holds on tuples of \emph{distinct} elements, this follows from the assumption that $\qftp^{\cH_k}_{\calL'_k}(g_{i_1},\ldots, g_{i_k})=\qftp^{\cH_k}_{\calL'_k}(h_{i_1},\ldots, h_{i_k})$.
\end{proof}

Of course, one could give an alternate proof of the previous result by observing that a relation of arity at most $k$ must be $\NFOP_k$, and then applying preservation of $\NFOP_k$ under Boolean combinations. However, said preservation result also uses generalized indiscernibles (via Proposition \ref{prop:codeTk}), and thus the previous proof makes this connection more explicit.

As a consequence of Proposition \ref{prop:arity}, we see that $\FOP_k$ crosscuts many of the standard binary dividing lines (except of course for stability and NIP). For example, the Henson graph is $\NFOP_k$ for all $k\geq 2$, but has $\text{TP}_2$ and $\text{SOP}_3$ (and hence is not simple); while on the other hand the generic $(k+1)$-uniform hypergraph is simple but has $\FOP_k$ (in fact, $\IP_k$). Recall also  that we have the general implications $\NFOP_k\Rightarrow\NIP_k\Rightarrow \NFOP_{k+1}$ (see  Proposition \ref{prop:IPk}). We can now show that these implications are strict. We will need the following lemma.

\begin{lemma}\label{lem:IPorder}
    Let $T$ be a complete theory and suppose $\varphi(\xbar;\ybar)$ is $\NIP$ in $T$, with $|\xbar|=k=|\ybar|$. Let $z_1,\ldots,z_{k+1}$ be a partition of the $2k$ free variables in $\varphi(\xbar;\ybar)$. Then $\theta(z_1,\ldots,z_{k+1})\coloneqq \varphi(\xbar,\ybar)$ is $\NIP_k$ in $T$.
\end{lemma}
\begin{proof}
We first claim that there are $i,j\in [k+1]$ such that $z_i\seq \xbar$ and $z_j\seq \ybar$. 
    Since $\varphi$ only contains $2k$ free variables, it follows that there are distinct $i,j\in [k+1]$ such that $|z_i|=|z_j|=1$. If $z_i\in \xbar$ and $z_j\in \ybar$ then the claim follows. So suppose without loss of generality that $z_i,z_j\in\xbar$. Then $\xbar\backslash\{z_i,z_j\}$ contains $k-2$ variables, while there are $k-1$ remaining indices in $[k+1]\backslash\{i,j\}$. So there is some $j'\in [k+1]\backslash\{i,j\}$ such that $z_{j'}\seq \ybar$, and we have our claim.

Now for a contradiction suppose $\theta(z_1,\ldots,z_{k+1})$ has $\IP_k$. Then, with $i,j\in[k+1]$ as above, we can find $\cbar\in \prod_{\ell\in[k+1]\backslash\{i,j\}}M^{x_\ell}$ such that, if $\theta(z_i,z_j;\cbar)$ is obtained from $\theta(z_1,\ldots,z_{k+1})$ by instantiating $z_\ell$  with $c_\ell$ for all $\ell\in[k+1]\backslash\{i,j\}$, then $\theta(z_i,z_j;\cbar)$ has the independence property. But since $z_i\seq \xbar$ and $z_j\seq \ybar$, this would yield the independence property for $\varphi(\xbar;\ybar)$, which is a contradiction.    
\end{proof}

\begin{proposition}\label{prop:IPkstrict}$~$
\begin{enumerate}[$(a)$]
\item The theory of the generic $(k+1)$-uniform hypergraph is $\NFOP_{k+1}$ and $\IP_k$.
\item $\Th(\cH_k)$ is $\NIP_k$ and $\FOP_k$.
\end{enumerate}
\end{proposition}
\begin{proof}
Part $(a)$. $\NFOP_{k+1}$ follows from Proposition \ref{prop:arity} and, as noted above, $\IP_k$ is well-known (and follows essentially by definition of $\IP_k$).

Part $(b)$. In $\cH_k$, the relation  $R(x_1,\ldots,x_{k+1})$ obviously $R$-codes $\cH_k$, and thus has $\FOP_k$ by Proposition \ref{prop:codeTk}. To verify that $\Th(\cH_k)$ is $\NIP_k$, we first recall  that $\NIP_k$ formulas are closed under Boolean combinations and arbitrary permutations of variables (see Corollaries 3.15 and 5.3 in \cite{CPT}). So by quantifier elimination, it suffices to just check that the relations $R$ and $<_k$ (under any $(k+1)$-partition of the $2k$ variables) are $\NIP_k$. For $<_k$, this follows from Lemma \ref{lem:IPorder} and the fact that any linear order is NIP. Finally, if  $R(x_1,\ldots,x_{k+1})$ were $\IP_k$ then by Proposition \ref{prop:IPkeq} it would embed any finite $(k+1)$-partite $(k+1)$-hypergraph as an induced subgraph; but this is not the case since, for example, there are no $\abar,\bbar\in Q(\cH_k)$ and $c,d\in P_{k+1}(\cH_k)$ such that $\cH_k\models R(\abar,c)\wedge R(\bbar,d)\wedge\neg R(\abar,d)\wedge\neg R(\bbar,c)$.
\end{proof}

Next we discuss another natural example of a structure with $\FOP_k$ and $\NIP_k$.

\begin{definition}\label{def:LQ}
Let $\cL^Q_k=\{P_1,\ldots,P_k,<,<_k\}$ and let $\cF_k$ be the $\cL^Q_k$-structure obtained from $\cH_k$ by forgetting $R$ and removing $P_{k+1}(\cH_k)$, i.e. $\cF_k$ has universe $\cH_k\backslash P_{k+1}(\cH_k)$. 
\end{definition}

Alternatively, one can describe $\cF_k$ as the Fra\"{i}ss\'{e} limit of the class of finite $\cL^Q_k$-structures in which $P_1,\ldots,P_k$ yields an $<$-ordered partition in the usual way, and $<_k$ is a $2k$-ary relation inducing an arbitrary linear order on $P_1\times\ldots\times P_k$. Thus $\cF_k$ can be viewed as the ``generic order on $k$-tuples" (with a base order $<$).

\begin{proposition}\label{prop:FkFOPk}
$\Th(\cF_k)$ is $\FOP_k$ and $\NIP_k$.
\end{proposition}
\begin{proof}
$\NIP_k$ for $\Th(\cF_k)$ follows from Lemma \ref{lem:IPorder}, quantifier elimination, and \cite[Proposition 3.15]{CPT} (preservation of $\NIP_k$ under Boolean combinations). For $\FOP_k$, let $\varphi(x_1,\ldots,x_k,x_{k+1})$ be the formula $(x_1,\ldots,x_k)<_k x_{k+1}$, where $x_{k+1}=(y_1,\ldots,y_k)$. We show that $\varphi(\xbar)$ $R$-codes $\cH_k$, and so $\Th(\cF_k)$ has $\FOP_k$ by Proposition \ref{prop:codeTk}. Define an $\cL^Q_k$-structure $\cH$ as follows. 
\begin{enumerate}[\hspace{5pt}$\ast$]
\item For $i\in[k]$, $P_i(\cH)=P_i(\cH_k)\sqcup E_i$, where $E_i$ is a countably infinite set, and the $E_i$'s are disjoint.
\item $<^{\cH}$ is a linear order such that $P_1(\cH)<^{\cH}\ldots<^{\cH}P_k(\cH)$. 
\item Let $f\colon E_1\times\ldots\times E_k\to P_{k+1}(\cH_k)$ be a bijection, and extend $f$ to $W\coloneqq Q(\cH_k)\cup (E_1\times\ldots\times E_k)$ so that $f(\gbar)=\gbar$ for all $\gbar\in Q(\cH_k)$. Note that $W$ is a subset of $Q(\cH)$. Let $<^{\cH}_k$ be any linear order on $Q(\cH)$ such that, for $\alpha,\beta\in W$, $\cH\models \alpha<_k\beta$ if and only if $\cH_k\models f(\alpha)<_* f(\beta)$, where $<_*$ is defined as ${<_k}\cup R\cup{(\neg R)^{\textnormal{opp}}}\cup {({<}{\upharpoonright}P_{k+1})}$ as in Remark \ref{rem:Tkorder}. 
\end{enumerate}
By construction, $\cH$ is a countable model of the universal theory of $\cF_k$, and so we may assume $\cH\seq \cF_k$. For $1\leq i\leq k$ and $h\in P_i(\cH_k)$, let $a_h=h\in P_i(\cH)$. For $h\in P_{k+1}(\cH)$, let $a_h=f^{\text{-}1}(h)\in Q(\cH)$. Then for $(h_1,\ldots,h_k)\in P_1(\cH_k)\times\ldots\times P_{k+1}(\cH_k)$, we have
\begin{multline*}
\cF_k\models \varphi(a_{h_1},\ldots,a_{h_k},a_{h_{k+1}})\miff (h_1,\ldots,h_k)<^{\cH}_k f^{\text{-}1}(h_{k+1})\\
\miff \cH_k\models (h_1,\ldots,h_k)<_* h_{k+1}\miff \cH_k\models R(h_1,\ldots,h_k,h_{k+1}).
\end{multline*}
So $\varphi(x_1,\ldots,x_{k+1})$ $R$-codes $\cH_k$ in $\cF_k$.
\end{proof}

\begin{remark}\label{rem:Fk}
Regarding the connection between $\FOP_k$ and $\cF_k$, we do have some behavior similar to $\cH_k$. In particular, if $T$ is an $\cL$-theory with monster model $\M$, then the following holds.
\begin{enumerate}[$(i)$]
\item Let  $\varphi(x_1,\ldots,x_k,y_1,\ldots,y_k)$ be an $\cL$-formula. Suppose that for any linear order $<_k$ on $\omega^k$, there are sequences $(a^1_i)_{i<\omega}$, \ldots, $(a^k_i)_{i<\omega}$ such that
\[
\M\models \varphi(a^1_{i_1},\ldots,a^k_{i_k},a^1_{j_1},\ldots,a^k_{j_k})\miff (i_1,\ldots,i_k) <_k (j_1,\ldots,j_k).
\]
Then $\varphi(x_1,\ldots,x_k,x_{k+1})$ has $\FOP_k$ in $T$, where $x_{k+1}=(y_1,\ldots,y_k)$.
\item If $T$ is $\NFOP_k$ then every $\cF_k$-indexed indiscernible sequence in $\M$ is $\cL^Q_k\backslash\{<_k\}$-indiscernible.
\end{enumerate}
Indeed, $(i)$ follows using the same coding of $\cH_k$ with $\varphi$ as in the proof of Proposition \ref{prop:FkFOPk}. As for $(ii)$, the argument is similar to what we will do in the next section to obtain the characterization of $\NFOP_k$ using indiscernible collapse with $\cH_k$. We omit details since the statement of $(ii)$ is not used in our work and, as the reader is about to see, the proof for $\cH_k$ is already complicated enough.

On the other hand, we expect that neither $(i)$ nor $(ii)$ leads to a characterization of $\FOP_k$. In particular, the theory of the generic $(k+1)$-uniform hypergraph is $\IP_k$ but should not contain a formula as in $(i)$. It also seems likely that this theory satisfies the conclusion of $(ii)$. The theory of $\cH_k$ should also satisfy the conclusion of $(ii)$. But again, we will not pursue these details here.  
\end{remark}

\section{$\NFOP_k$ and collapsing indiscernibles: Part II}\label{sec:collapse}

\subsection{Characterizing $\FOP_k$}
The goal of this subsection is to prove the characterization of $\NFOP_k$ via collapse of indiscernibles (Theorem \ref{thm:genind} below). Recall that in Corollary \ref{cor:genind}, we showed that a certain kind of indiscernible collapse implies $\NFOP_k$. So it remains to show that such a collapse happens in \emph{any} $\NFOP_k$ theory. Given Corollary \ref{cor:genind}, a natural choice for the ``collapsed sublanguage" is $\cL'_k$. However, our main application of the results in this subsection will be to prove a reduction to one variable for $\FOP_k$ (see Theorem \ref{thm:FOPkone}). For this to work, we will need to involve a different sublanguage of $\cL_k$.

\begin{definition}\label{def:L''}
Let $\cL''_k\coloneqq \cL_k\backslash\{<_k,R\}=\{P_1,\ldots,P_{k+1},<\}$.
\end{definition}

We use $\cong$ and $\cong''$ for isomorphism in $\cL_k$ and $\cL''_k$, respectively. Recall that any model $\cA\models T_k$ admits a canonical linear order $<_*$ on $Q(\cA)\cup P_{k+1}(\cA)$, as described in Remark \ref{rem:Tkorder}.  Moreover, we have the following converse.

\begin{remark}\label{rem:addorder}
Let $T''_k$ be the $\cL''_k$-reduct of $T_k$, and suppose $\cA\models T''_k$. Let $<_*$ be a linear order on $Q(\cA)\cup P_{k+1}(\cA)$ such that $<_*$ coincides with $<$ on $P_{k+1}(\cA)$. Expand $\cH$ to an $\cL_k$-structure $\cA^*$ by interpreting $<_k$ as ${<_*}{\upharpoonright}(Q(\cA)\times Q(\cA))$ and $R$ as ${<_*}{\upharpoonright}(Q(\cA)\times P_{k+1}(\cA))$. Then $\cA^*\models T_k$.
\end{remark}

The overall strategy of this subsection will be reminiscent of the proof that every indiscernible sequence in a model of a stable theory is an indiscernible set.  Said proof uses the fact that any permutation of a finite set  can be generated by consecutive transpositions.  We will use a similar idea, but with the notion of ``transposition" applied to a pair of  models of $T_k$.

\begin{definition}\label{def:transpo}
Suppose $\cA_1$ and $\cA_2$ are two models of $T_k$ with $\cA_1{\upharpoonright}\cL''_k=\cA_2{\upharpoonright}\cL''_k$ (so $(P_i(\cA_1),<)=(P_i(\cA_2),<)$ for all $1\leq i\leq k$). Let $P_i=P_i(\cA_1)=P_i(\cA_2)$ and $Q=Q(\cA_1)=Q(\cA_2)$. 
\begin{enumerate}[$(1)$]
\item $(\cA_1,\cA_2)$ is a \textbf{$QP$-transposition} if there is $(\ebar,v)\in Q\times P_{k+1}$ such that
\begin{enumerate}[$(i)$]
\item $\ebar$ and $v$ are $<_*$-adjacent in both $\cA_1$ and $\cA_2$,
\item $<^{\cA_1}_*$ and $<^{\cA_2}_*$ agree on all pairs except $(\bar{e},v)$.
\item $\cA_1\models \bar{e}<_*v$ if and only if $\cA_2\models v<_*\bar{e}$. 
\end{enumerate}
\item $(\cA_1,\cA_2)$ is a \textbf{$QQ$-transposition} if there is $(\dbar,\ebar)\in Q\times Q$ such that
\begin{enumerate}[$(i)$]
\item $\dbar$ and $\ebar$ are $<_*$-adjacent in both $\cA_1$ and $\cA_2$,
\item $<^{\cA_1}_*$ and $<^{\cA_2}_*$ agree on all pairs except $(\dbar,\ebar)$.
\item $\cA_1\models \dbar<_*\ebar$ and $\cA_2\models \ebar <_* \dbar$. 
\end{enumerate}
\item $(\cA_1,\cA_2)$ is a \textbf{transposition} (\textbf{witnessed by $(\alpha,\beta)$}) if it is a $QP$-transposition (and $(\alpha,\beta)=(\ebar,v)$ in $(1)$) or a $QQ$-transposition (and $(\alpha,\beta)=(\dbar,\ebar)$ in $(2)$).
\end{enumerate}
\end{definition}

\begin{definition}
Fix $\cA,\cB\models T_k$, and suppose $\alpha,\beta\in Q(\cA)\cup P_{k+1}(\cA)$ are $<^{\cA}_*$-adjacent. Then a function $f\colon \cA\to \cB$ is an \textbf{$(\alpha,\beta)$-deficient $\cL_k$-isomorphism} if it is an $\cL''_{k}$-isomorphism that also preserves the $<_*$-ordering among all pairs except possibly $(\alpha,\beta)$. 
\end{definition}

Putting the two previous definitions together, we have the following observation.

\begin{remark}\label{rem:trantocong}
Suppose $(\cA_1,\cA_2)$ is a transposition witnessed by $(\alpha,\beta)$. Then the identity map from $\cA_1$ to $\cA_2$ is an $(\alpha,\beta)$-deficient $\cL_k$-isomorphism. 
\end{remark}

Next we show that for any finite model of $T''_k$, transpositions ``generate" the class of expansions to a model of $T_k$.

\begin{lemma}\label{lem:transposition}
Fix finite $\cA,\cB\models T_k$ with $\cA\cong'' \cB$. Then there are $\cA_1,\ldots,\cA_n\models T_k$ such that $\cA_1=\cA$, $\cA_n\cong \cB$, and each pair $(\cA_i,\cA_{i+1})$ is a  transposition.
\end{lemma}
\begin{proof}
Since $\cA$ and $\cB$ are finite and linearly ordered by $<$, it follows from  $\cA\cong''\cB$ that there is a unique $\cL''_k$-isomorphism $f\colon \cA\to \cB$. Let $\cA'$ be the (unique) $\cL_k$-structure with domain $A$ such that $f\colon \cA'\to \cB$ is an $\cL_k$-isomorphism. Then, in particular, we have $\cA{\upharpoonright}\cL''_k=\cA'{\upharpoonright}\cL''_k$. To simplify notation, let ${<}$ denote ${<^{\cA}}={<^{\cA'}}$, $Q$ denote $Q(\cA)=Q(\cA')$, and $P$ denote $P_{k+1}(\cA)=P_{k+1}(\cA')$. Let ${<_*}={<^{\cA}_*}$ and ${<'_*}={<^{\cA'}_*}$. 

Consider the linear orders $(Q\cup P,<_*)$ and $(Q\cup P,<'_*)$. Note that $<_*$ and $<'_*$ both coincide with $<$ when restricted to $P$. It follows that there is a sequence ${<^1_*},\ldots,{<^n_*}$ of linear orders on $Q\cup P$ such that:
\begin{enumerate}[$(i)$]
\item  ${<^1_*}={<_*}$ and ${<^n_*}={<'_*}$,
\item each $<^i_*$ agrees with $<$ on $P$, and
\item $<^{i+1}_*$ is obtained from $<^i_*$ by swapping an $<^i_*$-adjacent pair from either $Q\times P$ or $Q\times Q$.
\end{enumerate}
Now let $\cA_n$ be the unique $\cL_k$-structure such that $\cA_n{\upharpoonright}{\cL''_k}=\cA{\upharpoonright}\cL_k$ and ${<^{\cA_i}_*}={<^i_*}$. By Remark \ref{rem:addorder}, $\cA_i\models T_k$ for all $1\leq i \leq n$. Moreover, each pair $(\cA_i,\cA_{i+1})$ is a transposition. By construction $\cA_1=\cA$ and $\cA_n=\cA'\cong\cB$.
\end{proof}

When proving that instability arises from an indiscernible sequence which is not an indiscernible set, one first reduces to the case that the failure of set indiscerniblity  is witnessed by an (order) transposition. Then this transposition is removed and replaced by an infinite linear order used to produce the order property. In our setting, we will need to work with an elaborate variation on this idea. In particular, for each of the two types of transpositions of models of $T_k$ defined above, we need to formulate a way to remove the transposition and replace it with some other model of $T_k$. We start with the setting of $QP$-transpositions, since it is a bit easier.

\begin{definition}[The structure $\cA\oslash_{(\bar{e},v)}\cB$]\label{def:Hproduct1}
Fix $\cA\models T_k$ and suppose $(\bar{e},v)\in Q(\cA)\times P_{k+1}(\cA)$ is an $<^{\cA}_*$-adjacent pair. Let $\cB$ be an arbitrary model of $T_k$ with $A\cap B=\emptyset$. We define an $\cL_k$-structure $\cC$ according to the following procedure.
\begin{enumerate}[$(1)$]
\item The domain of $\cC$ is $C=A_0\cup B$, where $A_0=A\backslash\ebar v$.
\item For each $1\leq i\leq k+1$, $(P_i(\cC),<^{\cC})$ is obtained from $(P_i(\cA),<^{\cA})$ by replacing $e_i$ with $(P_i(\cB),<^{\cB})$, where we let  $e_{k+1}=v$. 
\item Extend $<^{\cC}$ to all of $C$ so that $P_1(\cC)<^{\cC}\ldots<^{\cC} P_{k+1}(\cC)$.
\end{enumerate}
Before defining $<^{\cC}_k$ and $R^{\cC}$, we set some notation. Given a tuple $\cbar=(c_1,\ldots,c_k)\in Q(\cC)$, let $I(\cbar)=\{i\in [k]:c_i\in P_i(\cA)\}$. Define the tuple $\cbar^*=(c^*_1,\ldots,c^*_k)$ so that $c^*_i=c_i$ for all $i\in I(\cbar)$, and $c^*_i=e_i$ for all $i\not\in I(\cbar)$. Note that $\cbar^*\in Q(\cA)$. For $\abar \in Q(\cA)$, set $X(\abar)=\{\cbar\in Q(\cC):\cbar^*=\abar\}$. Then $\{X(\abar)\}_{\abar\in Q(\cA)}$ is a partition of $Q(\cC)$. Note that $X(\ebar)=Q(\cB)$. 
\begin{enumerate}[$(4)$]
\item[$(4)$] Next we define  $<^{\cC}_k$.  Let $<^{\cC}_k$ agree with $<^{\cB}_k$ on $X(\ebar)= Q(\cB)$. For $\abar\in Q(\cA)\backslash \ebar$, let $<^{\cC}_k$ be an arbitrary linear order on $X(\abar)$. Then extend $<^{\cC}_k$ to $Q(\cC)=\bigcup_{\abar\in Q(\cA)}X(\abar)$ so that $X(\abar)<^{\cC}_k X(\abar')$ for all $\abar,\abar'\in Q(\cA)$ such that $\abar<^{\cA}_k \abar'$. 

\item[$(5)$] Finally, we define $R^{\cC}$. Fix $\cbar\in Q(\cC)$ and $w\in P_{k+1}(\cC)$. Fix the unique $\abar\in Q(\cA)$ such that $\cbar\in X(\abar)$ We have three cases.
\begin{enumerate}[$(i)$]
\item If $w\in P_{k+1}(\cA)$, then set $R^{\cC}(\cbar,w)$ if and only if $R^{\cA}(\abar,w)$.
\item If $w\in P_{k+1}(\cB)$ and $\abar\neq \ebar$, then set $R^{\cC}(\cbar,w)$ if and only if $R^{\cA}(\abar,v)$.
\item If $w\in P_{k+1}(\cB)$ and $\abar=\ebar$, then set $R^{\cC}(\cbar,w)$ if and only if $R^{\cB}(\cbar,w)$.
\end{enumerate}
\end{enumerate}
This finishes the definition of the $\cL_k$-structure $\cC$, which we will denote by $\cA\oslash_{(\bar{e},v)}\cB$. 
\end{definition}

\begin{lemma}\label{lem:Hproduct}
Fix $\cA\models T_k$ and suppose $(\bar{e},v)\in Q(\cA)\times P_{k+1}(\cA)$ is an $<^{\cA}_*$-adjacent pair. Let $\cB$ be an arbitrary model of $T_k$ with $A\cap B=\emptyset$. 
\begin{enumerate}[$(a)$]
\item The substructure of $\cA$ on $A\backslash\ebar v$ embeds in $\cA\oslash_{(\ebar,v)}\cB$ via the inclusion map.
\item $\cB$ embeds in $\cA\oslash_{(\ebar,v)}\cB$ via the inclusion map.
\item $\cA\oslash_{(\ebar,v)}\cB$ is a model of $T_k$.
\item Given $\bar{e}'\in Q(\cB)$ and $v'\in P_{k+1}(\cB)$, let $\cA_{(\bar{e}',v')}$ denote the substructure of  $\cA\oslash_{(\ebar,v)}\cB$ on $A\backslash\ebar v\cup \bar{e}'v'$. Then the map fixing $A\backslash\ebar v$ pointwise and sending $\bar{e}v$ to $\bar{e}'v'$ is an $(\bar{e},v)$-deficient $\cL_k$-isomorphism from $\cA$ to $\cA_{(\bar{e}',v')}$.  
\end{enumerate}
\end{lemma}
\begin{proof}
Parts $(a)$ and $(b)$ are clear from Definition \ref{def:Hproduct1} (for part $(a)$, note that if $\abar\in Q^{\cA}(A\backslash\ebar v)$ then $X(\abar)=\{\abar\}$). 

To ease notation in parts $(c)$ and $(d)$, we let $\cC=\cA\oslash_{(\ebar,v)}\cB$. We will also make reference to the sets $X(\abar)$ in Definition \ref{def:Hproduct1}.

Part $(c)$.  It is clear that $P_1(\cC),\ldots,P_{k+1}(\cC)$ and $<^{\cC}$ satisfy the necessary axioms, and that $<^{\cC}_k$ is a linear order on $Q(\cC)$. So we just need to verify axiom $(5)$ of Definition \ref{def:Tk}. Fix $\cbar,\dbar\in Q(\cC)$ and $s,t\in P_{k+1}(\cC)$ such that $\cbar\leq ^{\cC}_k \dbar$, $R^{\cC}(\dbar,s)$, and $s\leq ^{\cC}t$. We need to show $R^{\cC}(\cbar,t)$. Fix $\abar,\bbar\in Q(\cA)$ such that $\cbar\in X(\abar)$ and $\dbar\in X(\bbar)$. By definition of $<^{\cC}_k$, it follows from $\cbar\leq^{\cC}_k\dbar$ that $\abar\leq^{\cA}_k\bbar$. We now have a case analysis. 
\begin{enumerate}[$(1)$]
\item $s,t\in P_{k+1}(\cA)$. Then $s\leq^{\cA} t$ and $R^{\cA}(\bbar,s)$. Since $\abar\leq^{\cA}_k\bbar$, we get $R^{\cA}(\abar,t)$, which yields $R^{\cC}(\cbar,t)$.
\item $s\in P_{k+1}(\cA)$ and $t\in P_{k+1}(\cB)$. Then $s<^{\cC} t$ implies $s<^{\cA} v$. So $s<^{\cA}_* v$, which implies $s<^{\cA}_* \ebar$ since $\ebar$ and $v$ are $<^{\cA}_*$-adjacent. Also, $R^{\cC}(\dbar,s)$ implies $R^{\cA}(\bbar,s)$, and so $\bbar<^{\cA}_* s$. Therefore, $\bbar<^{\cA}_*\ebar$, i.e., $\bbar<^{\cA}_k \ebar$. This implies $\abar\neq \ebar$. Since $\abar\leq^{\cA}_k\bbar$, $R^{\cA}(\bbar,s)$, and $s<^{\cA} v$, we have $R^{\cA}(\abar,v)$. Thus, $R^{\cC}(\cbar,t)$.  
\item $s\in P_{k+1}(\cB)$ and $t\in P_{k+1}(\cA)$. Then $s<^{\cC} t$ implies $v<^{\cA} t$. So $v<^{\cA}_* t$, which implies $\ebar<^{\cA}_* t$ since $\ebar$ and $v$ are $<^{\cA}_*$-adjacent. If $\bbar=\ebar$, then $\abar\leq^{\cA}_* \ebar<^{\cA}_* t$, and so $R^{\cA}(\abar,t)$, which implies $R^{\cC}(\cbar,t)$. So assume $\bbar\neq\ebar$. Then $R^{\cC}(\dbar,s)$ implies $R^{\cA}(\bbar,v)$. So   $R^{\cA}(\abar,t)$, which implies $R^{\cC}(\cbar,t)$. 

\item $s,t\in P_{k+1}(\cB)$. There are four subcases.
\begin{enumerate}[$\ast$]
\item  $\abar=\ebar$ and $\bbar=\ebar$. Then $\cbar,\dbar\in Q(\cB)$. So $\cbar<^{\cB}_k \dbar$, $R^{\cB}(\dbar,s)$, and $s<^{\cB} t$. Thus $R^{\cB}(\cbar,t)$, which implies $R^{\cC}(\cbar,t)$.
\item $\abar=\ebar$ and $\bbar\neq \ebar$. Then $\ebar<^{\cA}_* b$ and so $v<^{\cA}_* b$ since $\ebar$ and $v$ are $<^{\cA}_*$-adjacent. But $R^{\cC}(\dbar,s)$ implies $R^{\cA}(\bbar,v)$, which is a contradiction. 
\item $\abar\neq\ebar$ and $\bbar=\ebar$. Then $\abar<^{\cA}_*\ebar$ and so $\abar<^{\cA}_* v$ since $\ebar$ and $v$ are $<^{\cA}_*$-adjacent. Therefore $R^{\cA}(\abar,v)$, which implies  $R^{\cC}(\cbar,t)$. 
\item $\abar\neq\ebar$ and $\bbar\neq\ebar$. Then $R^{\cC}(\dbar,s)$ implies $R^{\cA}(\bbar,v)$. Since $\abar\leq^{\cA}_k\bbar$, we have $R^{\cA}(\abar,v)$, and so $R^{\cC}(\cbar,t)$. 
\end{enumerate}
\end{enumerate}

Part $(d)$. Fix $\ebar'\in Q(\cB)$ and $v'\in P_{k+1}(\cB)$. Let $\cA'=\cA_{(\ebar',v')}$. Define $f\colon \cA\to \cA'$ such that $f(x)=x$ for all $x\in A_0$ and $f(\ebar v)=\ebar'v'$. We need to show that $f$ is an $(\ebar,v)$-deficient $\cL_k$-isomorphism. By construction, it is clear that $f$ preserves the predicates $P_1,\ldots,P_{k+1}$ and is an $<$-order isomorphism. So we need to show that $f$ preserves $<_k$ and all $R$-relations except possibly the one between $\ebar$ and $v$. Note that for any $\abar\in Q(\cA)$, we have $f(\abar)\in Q(\cC)$ and $f(\abar)^*=\abar$ (recall the notation in Definition \ref{def:Hproduct1}), hence $f(\abar)\in X(\abar)$. So it follows from step $(4)$ of Definition \ref{def:Hproduct1} that $f$ preserves $<_k$. Now fix $\abar\in Q(\cA)$ and $t\in P_{k+1}(\cA)$ such that $(\abar,t)\neq (\ebar,v)$.  We want to show $R^{\cA}(\abar,t)$ if and only if $R^{\cC}(f(\abar),f(t))$. Recall that $f(\abar)\in X(\abar)$. If $t\neq v$ then $f(t)=t$, and so we have the desired result by step  $(5)(i)$ of Definition \ref{def:Hproduct1}. If $t=v$ then $\abar\neq\ebar$ and $f(t)=v'\in P_{k+1}(\cB)$, and so we have the desired result by step $(5)(ii)$. 
\end{proof}

Next we need an analogous replacement lemma for $QQ$-transpositions. In this case, the construction is more complicated, because when replacing the transposition by a structure $\cB$, the $P_{k+1}$-predicate of $\cB$ must be inserted in the $P_{i_*}$-predicate of the initial structure for some $i_*\in[k]$ (see  discussion after the definition). 

\begin{definition}[The structure $\cA\oless_{(\dbar,\ebar)}\cB$]\label{def:Hproduct2}
Fix $\cA\models T_k$ and suppose $(\bar{d},\bar{e})\in Q(\cA)\times Q(\cA)$ is an $<^{\cA}_*$-adjacent pair.   Let $\cB$ be an arbitrary model of $T_k$ with $A\cap B=\emptyset$. We define an $\cL_k$-structure $\cC$ according to the following procedure.
\begin{enumerate}[$(1)$]
\item The domain of $\cC$ is $C=A_0\cup B$, where $A_0=A\backslash\dbar e_{i_*}$ and $i_*\in[k]$ is the minimal $i$ such that $d_i\neq e_i$.
\item For  $1\leq i\leq k+1$, $(P_i(\cC),<^{\cC})$ is defined as follows.
\begin{enumerate}[$(i)$]
\item If $i\not\in\{i_*,k+1\}$ then  $(P_i(\cC),<^{\cC})$ is obtained from  $(P_i(\cA),<^{\cA})$ by replacing $d_i$ with  $(P_i(\cB),<^{\cB})$.
\item $(P_{i_*}(\cC),<^{\cC})$ is obtained from $(P_{i_*}(\cA),<^{\cA})$ by replacing $d_{i_*}$ with $(P_{i_*}(\cB),{<^{\cB}})$ and $e_{i_*}$ with $(P_{k+1}(\cB),<^{\cB})$.
\item $(P_{k+1}(\cC),<^{\cC})$ is $(P_{k+1}(\cA),<^{\cA})$.
\end{enumerate}
\item Extend $<^{\cC}$ to all of $C$ so that $P_1(\cC)<^{\cC}\ldots<^{\cC} P_{k+1}(\cC)$.
\end{enumerate}
Before defining $<^{\cC}_k$ and $R^{\cC}$, we set some notation. Given a tuple $\cbar=(c_1,\ldots,c_k)\in Q(\cC)$, let $I(\cbar)=\{i\in [k]:c_i\in P_i(\cA)\}$ as defined in Definition \ref{def:Hproduct1}. Define the tuple $\cbar^*=(c^*_1,\ldots,c^*_k)$ so that $c^*_i=c_i$ for all $i\in I(\cbar)$, and for $i\not\in I(\cbar)$,
\[
c^*_i=
\begin{cases}
d_i & \text{if $i\neq i_*$ or $c_{i_*}\in P_{i_*}(\cB)$}\\
e_{i_*} &\text{if $i=i_*$ and $c_{i_*}\in P_{k+1}(\cB)$}.
\end{cases}
\]
Note that $\cbar^*\in Q(\cA)$. For $\abar\in Q(\cA)$, set $X(\abar)=\{\cbar\in Q(\cC):\cbar^*=\abar\}$. Then $\{X(\abar)\}_{\abar\in Q(\cA)}$ is a partition of $Q(\cC)$. 

\begin{enumerate}[$(1)$]
\item[$(4)$] Next we define $<^{\cC}_k$. We first consider the sets $X(\abar)$.
\begin{enumerate}[$(i)$]
\item  For $\abar\in Q(\cA)\backslash\{\dbar,\ebar\}$, let $<^{\cC}_k$ be an arbitrary linear order on $X(\abar)$. 
\item  We now define $<^{\cC}_k$ on $X(\dbar)\cup X(\ebar)$. Note that $X(\dbar)=Q(\cB)$. Moreover, if for each $t\in P_{k+1}(\cB)$ we set $X_t(\ebar)=\{\cbar\in X(\ebar):c_{i_*}=t\}$, then $\{X_t(\ebar)\}_{t\in P_{k+1}(\cB)}$ partitions $X(\ebar)$. So 
\[
\textstyle X(\dbar)\cup X(\ebar)=Q(\cB)\cup\bigcup_{t\in P_{k+1}(\cB)}X_t(\ebar).
\]
Now $(X(\dbar)\cup X(\ebar),<^{\cC}_k)$ is obtained from $(Q(\cB)\cup P_{k+1}(\cB),<^{\cB}_*)$ by replacing each $t\in P_{k+1}(\cB)$ with $(X_t(\ebar),<_t)$, where $<_t$ is an arbitrary linear order. 
Note that $(X(\dbar),<^{\cC}_k)=(Q(\cB),<^{\cB}_k)$. 
\end{enumerate}
Now extend $<^{\cC}_k$ to all of $Q(\cC)$ so that:
\begin{enumerate}
\item[$(iii)$] $X(\abar)<^{\cC}_k X(\abar')$ for all $\abar,\abar'\in Q(\cA)\backslash\{\dbar,\ebar\}$ such that $\abar<^{\cA}_k\abar'$,
\item[$(iv)$]  $X(\abar)<^{\cC}_k X(\dbar)\cup X(\ebar)$ for all $\abar\in Q(\cA)\backslash\{\dbar,\ebar\}$ such that $\abar<^{\cA}_k\dbar$, and 
\item[$(v)$] $X(\dbar)\cup X(\ebar)<^{\cC}_k X(\abar)$ for all $\abar\in Q(\cA)\backslash\{\dbar,\ebar\}$ such that $\dbar<^{\cA}_k \abar$.
\end{enumerate}
Note that the above linear order is well-defined since $\dbar$ and $\ebar$ are $<^{\cA}_*$-adjacent.
\item[$(5)$] Finally, we define $R^{\cC}$. Fix $\cbar\in Q(\cC)$ and $v\in P_{k+1}(\cC)=P_{k+1}(\cA)$. Fix the unique $\abar\in Q(\cA)$ such that $\cbar\in X(\abar)$. Then we set $R^{\cC}(\cbar,v)$ if and only if $R^{\cA}(\abar,v)$. 
\end{enumerate}
This finishes the definition of $\cL_k$-structure $\cC$, which we will denote by $\cA\oless_{(\dbar,\ebar)}\cB$. 
\end{definition}

Unlike the  construction in Definition \ref{def:Hproduct1}, $\cB$ need not be an $\cL_k$-substructure of $\cA\oless_{(\dbar,\ebar)}\cB$. 
However, as we will see in the next lemma, $\cA\oless_{(\dbar,\ebar)}\cB$  contains substructures  isomorphic to $\cB$ modulo choosing new interpretations of $P_{k+1}$ and $R$. 
First we define some notation needed to make this statement precise.

\begin{definition}
Let $\cA$, $\cB$, $\dbar$, $\ebar$, and $i_*$ be as in Definition \ref{def:Hproduct2}. Let $I=[k]\backslash\{i_*\}$ and let $P_{\dbar,\ebar}(\cB)$ the set of $I$-indexed tuples $(b_i)_{i\in I}$ such that $b_i=e_i$ if $e_i\neq d_i$ and $b_i\in P_i(\cB)$ if $e_i=d_i$. Given $\bbar\in P_{\dbar,\ebar}(\cB)$ and $t\in P_{k+1}(\cB)$, let $\bbar(t)$ be the $k$-tuple obtained from $\bbar$ by setting $b_{i_*}=t$. Note, in particular, that $\bbar(t)\in Q(\cA\oless_{(\dbar,\ebar)}\cB)$. 
\end{definition}

At this point, we recall the language $\cL^Q_k=\{P_1,\ldots,P_k,<,<_k\}$ from Definition \ref{def:LQ}, and we note again that any $\cL_k$-structure $\cB$ determines a canonical $\cL^Q_k$-structure $\cB{\upharpoonright}\cL^Q_k$ obtained by forgetting $P_{k+1}$ and $R$. 

\begin{lemma}\label{lem:Hproduct2}
Fix $\cA\models T_k$ and suppose $(\bar{d},\bar{e})\in Q(\cA)\times Q(\cA)$ is an $<^{\cA}_*$-adjacent pair.   Let $\cB$ be an arbitrary model of $T_k$ with $A\cap B=\emptyset$. Let $i_*$ be the minimal $i\in [k]$ such that $d_i\neq e_i$.
\begin{enumerate}[$(a)$]
\item The substructure of $\cA$ on $A\backslash \dbar e_{i_*}$ embeds in $\cA\oless_{(\dbar,\ebar)}\cB$ via the inclusion map. 
\item 
\begin{enumerate}[$(1)$]
\item $\cB{\upharpoonright}\cL^Q_k$ embeds in $(\cA\oless_{\dbar,\ebar}\cB){\upharpoonright}\cL^Q_k$ via the inclusion map.
\item $(P_{k+1}(\cB),<)$ embeds in $(P_{i_*}(\cA\oless_{(\dbar,\ebar)}\cB),<)$ via the inclusion map.
\item For any $\bbar\in P_{\dbar,\ebar}(\cB)$, $\sbar\in Q(\cB)$, and $t\in P_{k+1}(\cB)$, $\cB\models R(\sbar,t)$  if and only if $\cA\oless_{(\dbar,\ebar)}\cB\models \sbar<_k \bbar(t)$.
\end{enumerate}
\item $\cA\oless_{\dbar,\ebar}\cB$ is a model of $T_k$.
\item Given $\sbar\in Q(\cB)$ and $t\in P_{k+1}(\cB)$, let $\cA_{(\sbar,t)}$ denote the substructure of $\cA\oless_{(\dbar,\ebar)}\cB$ on $A\backslash \dbar e_{i_*}\cup\sbar t$. Then the map fixing $A\backslash \dbar e_{i_*}$ pointwise and sending $\dbar e_{i_*}$ to $\sbar t$ is a $(\dbar,\ebar)$-deficient $\cL_k$-isomorphism from $\cA$ to $\cA_{(\sbar,t)}$. 
\end{enumerate}
\end{lemma}
\begin{proof}
Parts $(a)$, $(b)(1)$, and $(b)(2)$ are clear from Definition \ref{def:Hproduct2}. To ease notation in the remaining parts, we let $\cC=\cA\oless_{\dbar,\ebar}\cB$. We will also make reference to the sets $X(\abar)$ and $X_t(\ebar)$ from Definition \ref{def:Hproduct2}.

Part $(b)(3)$. Fix $\bbar\in P_{\dbar,\ebar}(\cB)$, $\sbar\in Q(\cB)$, and $t\in P_{k+1}(\cB)$. Recall from Definition \ref{def:Hproduct2} that $X(\dbar)=Q(\cB)$, and so we have $\sbar\in X(\dbar)$. By definition of $\bbar(t)$, we  have $\bbar(t)\in X_t(\ebar)$. So from step $(4)(ii)$ of Definition \ref{def:Hproduct2}, $\sbar<^{\cC}_k \bbar(t)$ if and only if $\sbar <^{\cB}_* t$ (i.e., $R^{\cB}(\sbar,t)$), as desired. 

Part $(c)$. It is clear that $P_1(\cC),\ldots,P_{k+1}(\cC)$ and $<^{\cC}$ satisfy the necessary axioms, and we already observed in step $(4)$ of Definition \ref{def:Hproduct2} that $<^{\cC}_k$ is a linear order on $Q(\cC)$. So we just need to verify axiom $(5)$ of Definition \ref{def:Tk}. Fix $\cbar,\cbar'\in Q(\cC)$ and $t,t'\in P_{k+1}(\cC)$ such that $\cbar\leq^{\cC}_k \cbar'$, $R^{\cC}(\cbar',t)$, and $t\leq^{\cC} t'$. We need to show $R^{\cC}(\cbar,t')$. Fix $\abar,\abar'\in Q(\cA)$ such that $\cbar\in X(\abar)$ and $\cbar'\in X(\abar')$. Then $R^{\cC}(\cbar',t)$ implies $R^{\cA}(\abar',t)$. So $R^{\cA}(\abar',t')$ holds since $t\leq^{\cA} t'$ (recall step $(2)(iii)$ of Definition \ref{def:Hproduct2}) and $\cA\models T$. We need to show $R^{\cA}(\abar,t')$. First note that if $\{\abar,\abar'\}=\{\dbar,\ebar\}$ then from $R^{\cA}(\abar',t')$ and the assumption that $\dbar$ and $\ebar$ are $<^{\cA}_*$-adjacent, we have $R^{\cA}(\abar,t')$, as desired. On the other hand, if $\{\abar,\abar'\}\neq \{\dbar,\ebar\}$ then by step $(4)$ of Definition \ref{def:Hproduct2}, $\cbar\leq^{\cC}_k \cbar'$ implies $\abar\leq^{\cA}_k \abar'$, and so $R^{\cA}(\abar,t')$.

Part $(d)$. Fix $\sbar\in Q(\cB)$ and $t\in P_{k+1}(\cB)$. Let $\cA'=\cA_{(\sbar,t)}$. Define $f\colon \cA\to\cA'$ such that $f(x)=x$ for all $x\in A\backslash\dbar e_{i_*}$ and $f(\dbar e_{i_*})=\sbar t$. We need to show that $f$ is a $(\dbar,\ebar)$-deficient $\cL_k$-isomorphism. By construction, it is clear that $f$ preserves the predicates $P_1,\ldots,P_{k+1}$ and is an $<$-order isomorphism. So we need to show that $f$ preserves $R$ and the $<_k$-order between any two elements from $Q(\cA)$ except possibly $\dbar$ and $\ebar$. Note that for any $\abar\in Q(\cA)$, we have $f(\abar)\in Q(\cC)$ and $f(\abar)^*=\abar$ (recall the notation from Definition \ref{def:Hproduct2}), hence $f(\abar)\in X(\abar)$. So it follows from step $(4)$ of Definition \ref{def:Hproduct2} that $f$ preserves the $<_k$-order between any two elements from $Q(\cA)$ except possibly $\dbar$ and $\ebar$. Now fix $\abar\in Q(\cA)$ and $v\in P_{k+1}(\cA)$. Note that $v\in P_{k+1}(\cC)$ and $f(v)=v$. So we want to show $R^{\cA}(\abar,v)$ if and only if $R^{\cC}(f(\abar),v)$. But this follows from step $(5)$ of Definition \ref{def:Hproduct2} and the fact that $f(\abar)\in X(\abar)$.
\end{proof}

Now we need to start working in the Fra\"{i}ss\'{e} limit $\cH_k$. Note that if $(\cA_1,\cA_2)$ is a transposition, then by definition $\cA_1$ and $\cA_2$ are \emph{distinct} models of $T_k$ on the same domain $A$. If $A$ is finite then of course we can embed $\cA_1$ and $\cA_2$ as substructures of $\cH_k$ with (necessarily) distinct domains in $H_k$. Next we describe how to accomplish this so these domains overlap as much as possible.

\begin{proposition}\label{prop:movetrans}
Let $\cA_1$ be a finite substructure of $\cH_k$ with domain $A$.
\begin{enumerate}[$(a)$]
\item Suppose $(\cA_1,\cA_2)$ is a $QP$-transposition, and let $(\ebar,v)$ be as in Definition \ref{def:transpo}$(1)$. Then there is a substructure $\cA^*_2$ of $\cH_k$ such that $\cA_2$ is isomorphic to $\cA^*_2$ via a map fixing $A\backslash\{v\}$ pointwise.
\item Suppose $(\cA_1,\cA_2)$ is a $QQ$-transposition, and let $(\dbar,\ebar)$ be as in Definition \ref{def:transpo}$(2)$. Set $i_*=\min\{i\in [k]:d_i\neq e_i\}$. 
Then there is a substructure $\cA^*_2$ of $\cH_k$ such that $\cA_2$ is isomorphic to $\cA^*_2$ via a map fixing $A\backslash\{e_{i_*}\}$ pointwise.
\end{enumerate}
\end{proposition}
\begin{proof}
We prove both parts simultaneously, using $(\alpha,\beta)$ and $t$ as names for $(\ebar,v)$ and $v$ in part $(a)$, and as names for $(\dbar,\ebar)$ and $e_{i_*}$ in part $(b)$. By Remark \ref{rem:trantocong}, the identity map from $\cA_1$ to $\cA_2$ is an $(\alpha,\beta)$-deficient $\cL_k$-isomorphism. Therefore, if $A_0=A\backslash\{t\}$ and  $\cA_0$ is the substructure of $\cA_1$ with domain $A_0$, then $\beta$ no longer appears in $\cA_0$, and so the inclusion map from $\cA_0$ to $\cA_2$ is an $\cL_k$-isomorphism. By universality and homogeneity of  $\cH_k$, we have an embedding $f\colon \cA_2\to \cH_k$ whose restriction to $A_0$ is the inclusion map. Now let $\cA_2^*$ be the image of $f$.
\end{proof}

\begin{definition}
A pair $(\cA_1,\cA^*_2)$ of finite substructures of $\cH_k$ is called a \textbf{$QP$-transposition in $\cH_k$} if there is some $\cA_2$ such that $(\cA_1,\cA_2)$ is a $QP$-transposition and $\cA_2^*$ is obtained from $\cA_2$ as in Proposition \ref{prop:movetrans}$(a)$. The notions of a \textbf{$QQ$-transposition in $\cH_k$} and a \textbf{transposition in $\cH_k$} are defined analogously.
\end{definition}

We are now ready to prove the main result of this subsection, which shows that any $\NFOP_k$ theory admits a suitable collapse of indiscernibles. Before stating this result, we observe that for any $1\leq i\leq k+1$, all singletons in $P_i(\cH_k)$ have the same $\cL_k$-type.  As noted in Definition \ref{def:genind}$(1)$, it follows that for any $\cH_k$-indexed indiscernible sequence $(a_g)_{g\in H_k}$ in some $\M\models T$, and any $1\leq i\leq k+1$, there is an integer $\ell_i$ such that $|a_g|=\ell_i$ for all $g\in P_i(\cH_k)$.  We call $(\ell_1,\ldots,\ell_{k+1})$ the \emph{length assignment} of $(a_g)_{g\in H_k}$.

\begin{lemma}\label{lem:collapse1}
Let $T$ be a complete theory with monster model $\M$, and suppose $\M$ contains an $\cH_k$-indexed indiscernible sequence $(a_g)_{g\in H_k}$ that is not $\cL''_k$-indiscernible. Then there  is an $\cL$-formula $\varphi(x_1,\ldots,x_{k+1},y)$, a tuple $\gamma\in H_k^y$, and an $\cL_k$-structure $\cH^*$ satisfying the following properties.
\begin{enumerate}[$(i)$]
\item $\cH^*$ is isomorphic to $\cH_k$ and $H^*$ is a subset of $H_k$ disjoint from $\gamma$.
\item For any $g_1,\ldots,g_n,g'_1,\ldots,g'_n\in H^*$ if $\qftp^{\cH^*}_{\cL''_k}(g_1,\ldots,g_n)=\qftp^{\cH^*}_{\cL''_k}(g'_1,\ldots,g'_n)$ then $\qftp^{\cH_k}_{\cL''_k}(g_1,\ldots,g_n,\gamma)=\qftp^{\cH_k}_{\cL''_k}(g'_1,\ldots,g'_n,\gamma)$.
\item If $(\ell_1,\ldots,\ell_k)$ is the length assignment of $(a_g)_{g\in H_k}$, then $|x_i|=|\ell_i|$ for all $1\leq i\leq k$ and $|x_{k+1}|=\ell_{i_*}$ for some $1\leq i_*\leq k+1$.
\item $\varphi(x_1,\ldots,x_{k+1},a_\gamma)$ $R$-codes $\cH^*$ witnessed by $(a_g)_{g\in H^*}$.
\end{enumerate}
\end{lemma}
\begin{proof}
Given a tuple $\gbar=(g_1,\ldots,g_n)$ from $\cH_k$, let $a_{\gbar}=(a_{g_1},\ldots,a_{g_n})$. 
Note that any $<$-increasing tuple $\gbar$ from $\cH_k$ determines a canonical finite model of $T_k$. Also, if $\gbar,\gbar'$ are $<$-increasing tuples then $\qftp^{\cH_k}_{\cL_k}(\gbar)=\qftp^{\cH_k}_{\cL_k}(\gbar')$ if and only if $\gbar\cong\gbar'$, and similarly for $\cL''_k$.

To ease notation, we write $a\equiv b$ for $\tp^T_{\cL}(a)=\tp^T_{\cL}(b)$.
Since $(a_g)_{g\in H_k}$ is not $\cL''_k$-indiscernible, there is some $n\geq 1$ and $\gbar,\gbar'$ from $\cH_k$ such that $\qftp^{\cH_k}_{\cL''_k}(\gbar)=\qftp^{\cH_k}_{\cL''_k}(\gbar')$, but $a_{\gbar}\not\equiv a_{\gbar'}$. Since $\qftp^{\cH_k}_{\cL''_k}(\gbar)=\qftp^{\cH_k}_{\cL''_k}(\gbar')$, we may assume, after relabeling, that $\gbar$ and $\gbar'$ are $<$-increasing tuples.\medskip

\noindent\textit{Claim:} Without loss of generality, we may assume $(\gbar,\gbar')$ is a transposition in $\cH_k$.

\noindent\textit{Proof:} By assumption,  $g\cong''\gbar'$. So by Proposition \ref{prop:movetrans} and Lemma \ref{lem:transposition}, there is a sequence $\gbar^1,\ldots,\gbar^n$ of tuples in $\cH_k$ such that $\gbar^1=\gbar$, $\qftp^{\cH_k}_{\cL_k}(\gbar^n)=\qftp^{\cH_k}_{\cL_k}(\gbar')$, and $(\gbar^i,\gbar^{i+1})$ is a transposition in $\cH_k$ for all $i$. By Remark \ref{rem:trantocong}, we have $\qftp^{\cH_k}_{\cL''_k}(\gbar^i)=\qftp^{\cH_k}_{\cL''_k}(\gbar)=\qftp^{\cH_k}_{\cL''_k}(\gbar')$ for all $1\leq i\leq n$.

Since $\qftp^{\cH_k}_{\cL_k}(\gbar^n)=\qftp^{\cH_k}_{\cL_k}(\gbar')$, we have $a_{\gbar_n}\equiv a_{\gbar'}$. On the other hand, $a_{\gbar^1}=a_{\gbar}$. So $a_{\gbar}\not\equiv a_{\gbar'}$ implies $a_{\gbar^1}\not\equiv a_{\gbar^n}$. Therefore, there is some $1\leq i<n$ such that $a_{\gbar^i}\not\equiv a_{\gbar^{i+1}}$. So if we replace $(\gbar,\gbar')$ with $(\gbar^i,\gbar^{i+1})$, then $(\gbar,\gbar')$ is  a transposition in $\cH_k$, and we still have $\qftp^{\cH_k}_{\cL''_k}(\gbar)=\qftp^{\cH_k}_{\cL''_k}(\gbar')$ and $a_{\gbar}\not\equiv a_{\gbar'}$. \clqed\medskip

We now need to consider two cases depending on  whether $(\gbar,\gbar')$ is a $QP$-transposition in $\cH_k$ or a $QQ$-transposition in $\cH_k$.\medskip 

\noindent \textit{Case 1:} $(\gbar,\gbar')$ is a $QP$-transposition in $\cH_k$. 

By construction, there are $\ebar\in Q(\gbar)$, $v\in P_{k+1}(\gbar)$, and $v'\in P_{k+1}(\gbar')$ such that $\gbar\backslash\{v\}=\gbar\backslash\{v'\}$ (so, in particular $Q(\gbar)=Q(\gbar')$ and $\ebar\in Q(\gbar')$), $\ebar$ and $v$ are $<_*$-adjacent in $\gbar$, and $\ebar$ and $v'$ are $<_*$-adjacent in $\gbar'$.
Moreover, we have $\ebar<_*v$ (i.e., $R(\ebar,v)$) if and only if $v'<_*\ebar$ (i.e., $\neg R(\ebar,v')$).  So without loss of generality, we can assume $v'<_*\ebar<_*v$.

For each $1\leq i\leq k$, choose  tuples $g^+_i$ and $g^-_i$ so that $g^-_ie_ig^+_i$ is the $<$-enumeration of $P_{i}(\gbar)=P_{i}(\gbar')$. Choose   tuples $h^-$ and $h^+$ so that $h^-vh^+$ and $h^-v'h^+$ are the $<$-enumerations of $P_{k+1}(\gbar)$ and $P_{k+1}(\gbar')$. Altogether:
\begin{enumerate}[\hspace{5pt}$\ast$]
\item $\gbar=g^-_1e_1g^+_1\ldots g^-_k e_kg^+_k h^-v h^+$,
\item $\gbar'=g^-_1e_1g^+_1\ldots g^-_k e_kg^+_k h^-v' h^+$,
\item $\ebar$ and $v$ are $<_*$-adjacent  in $\gbar$, with $\ebar<_*v$, and 
\item $\ebar$ and $v'$ are $<_*$-adjacent in $\gbar'$, with $v'<_*\ebar$.
\end{enumerate}
Let $\gamma$ be the tuple $g^-_1g^+_1\ldots g^-_kg^+_k h^-h^+$ (i.e., $\gamma=\gbar\backslash \ebar v=\gbar'\backslash \ebar v'$).

Let $\cC=\gbar\oslash_{\ebar,v}\cH^*$, where  $\cH^*$ is an isomorphic copy of $\cH_k$ (with domain disjoint from $\gbar$). By Lemma \ref{lem:Hproduct}$(c)$ and universality of $\cH_k$, we may assume $\cC$ is a substructure of $\cH_k$ extending $\gamma$. Note that $\cH^*$ satisfies condition $(i)$ of the lemma. Moreover,  by Lemma \ref{lem:Hproduct}$(b)$, $\cH^*$ is a substructure of $\cH_k$.  By construction of $\cC$ (in Definition \ref{def:Hproduct1}), it follows that for any $1\leq i\leq k+1$, all elements of $P_i(\cH^*)$ have the same $<$-cut in $\gamma$, which yields condition $(ii)$.

Given  $\sbar=(s_1,\ldots,s_k)\in Q(\cH^*)$ and $t\in P_{k+1}(\cH^*)$, let 
\[
g_{\sbar,t}=g^-_1 s_1 g^+_1\ldots g^-_k s_k g^+_k h^- t h^+.
\]
By construction, $g_{\sbar,t}$ $<$-enumerates the substructure of $\cC$ on $\gamma \sbar t$. 
By Lemma \ref{lem:Hproduct}$(d)$, the map fixing $\gamma$ pointwise and sending $\bar{e}v$ to $\bar{s}t$ is an $(\bar{e},v)$-deficient $\cL_k$-isomorphism from $\gbar$ to $g_{\sbar,t}$; moreover, if $R(\sbar,t)$ (i.e., $\bar{s}<_*t)$ then this map is a genuine $\cL_k$-isomorphism. Likewise, the map fixing $\gamma$ pointwise and sending $\bar{e}v'$ to $\bar{s} t$ is a $(\ebar,v')$-deficient $\cL_k$-isomorphism from $\gbar'$ to $g_{\sbar,t}$; moreover, if $\neg R(\sbar,t)$ (i.e., $t<_*\bar{s}$) then this map is a genuine $\cL_k$-isomorphism. Altogether: 
\begin{enumerate}[\hspace{5pt}$(i)$]
\item If $R(\sbar,t)$ then $\qftp^{\cH_k}_{\cL_k}(g_{\sbar,t})=\qftp^{\cH_k}_{\cL_k}(\gbar)$.
\item If $\neg R(\sbar,t)$ then $\qftp^{\cH_k}_{\cL_k}(g_{\sbar,t})=\qftp^{\cH_k}_{\cL_k}(\gbar')$.
\end{enumerate}

Now we return to the $\cH_k$-indiscernible sequence $(a_g)_{g\in H_k}$. 
Recall that $\avec{\bar{g}}\not\equiv \avec{\bar{g}'}$. So there is an $\cL$-formula $\theta(w)$ such that $\M\models\theta(\avec{\gbar})\wedge\neg\theta(\avec{\gbar'})$. Recall also that $(\ell_1,\ldots,\ell_{k+1})$ is the length assignment of $(\avec{g})_{g\in H_k}$. We now partition the free variables $w$ according to the enumerations of $\gbar$ and $\gbar'$ described above. In particular, we can write
\[
w=y^-_1x_1y^+_1\ldots y^-_kx_ky^+_kz^-x_{k+1}z^+,
\]
where:
\begin{enumerate}[\hspace{5pt}$\ast$]
\item for $1\leq i\leq k$, $|y^-_i|=|\avec{g^-_i}|$ and $|y^+_i|=|\avec{g^+_i}|$,
\item  $|z^-|=|\avec{h^i}|$ and $|z^+|=|\avec{h^+}|$, 
\item for $1\leq i\leq k$, $|x_i|=|\avec{e_i}|=\ell_i$, and
\item $|x_{k+1}|=|a_{v}|=\ell_{k+1}$.
\end{enumerate}
Let $y=y^-_1y^+_1\ldots y^-_k y^+_k z^- z^+$. Then after permuting the free variables, we can write $\theta(w)$ as $\varphi(x_1,\ldots,x_{k+1};y)$. Note that we have condition $(iii)$ of the lemma. So to finish the proof in this case, we need to show that $\varphi(x_1,\ldots,x_{k+1},a_\gamma)$ $R$-codes $\cH^*$ witnessed by $(\avec{g})_{g\in H^*}$.

First note that for any $\sbar\in Q(\cH^*)$ and $t\in \cP_{k+1}(\cH^*)$, the formula $\theta(\avec{g_{\sbar,t}})$ is  $\varphi(\avec{s_1},\ldots,\avec{s_{k}},\avec{t};\avec{\gamma})$. Now fix $\sbar\in Q(\cH^*)$ and $t\in P_{k+1}(\cH^*)$. By $(i)$ and $(ii)$ above, and $\cH_k$-indiscernibility of $(a_g)_{g\in H^*}$, we have that $R(\sbar,t)$ implies $\avec{g_{\sbar,t}}\equiv \avec{\gbar}$, and $\neg R(\sbar,t)$ implies $\avec{g_{\sbar,t}}\equiv\avec{\gbar'}$. Since $\M\models\theta(\avec{\gbar})\wedge\neg\theta(\avec{\gbar'})$, it follows that $\M\models\theta(\avec{g_{\sbar,t}})$ if and only if $R(\sbar,t)$, i.e., $\M\models \varphi(\avec{s_1},\ldots,\avec{s_k},\avec{t},a_\gamma)$ if and only if $R(\sbar,t)$, as desired.\medskip

\noindent\textit{Case 2:} $(\gbar,\gbar')$ is a $QQ$-transposition in $\cH_k$. 

By construction, there are $\dbar,\ebar\in Q(\gbar)$ such that $\dbar$ and $\ebar$ are $<_*$-adjacent in $\gbar$, with $\dbar<_k\ebar$, and if $i_*=\min\{i\in [k]:d_i\neq e_i\}$, then $\gbar'=\gbar\backslash\{e_{i_*}\}\cup\{e'\}$ and the map from $\gbar$ to $\gbar'$ fixing $\gbar\backslash\{e_{i_*}\}$ pointwise and sending $e_{i_*}$ to $e'$ is a $(\dbar,\ebar)$-deficient $\cL_k$-isomorphism. Moreover, if $\ebar'=(e_1,\ldots,e_{i_*-1},e',e_{i_*+1},\ldots,e_k)$, then $\dbar,\ebar'$ are in $Q(\gbar')$ and are $<_*$-adjacent in $\gbar'$ with $\ebar'<_k\dbar$. To ease notation, we will also assume $d_{i_*}<e_{i_*}$. This is a nontrivial assumption since $\dbar$ and $\ebar$ have asymmetric roles (beyond just their $<_k$-order). However, there is no loss in generality since the relative $<$-order of $d_{i_*}$ and $e_{i_*}$ is  insignificant in the subsequent argument.

Let $g_{k+1}$ be the $<$-enumeration of $P_{k+1}(\gbar)=P_{k+1}(\gbar')$. For each $i\in[k]\backslash\{i_*\}$,  choose tuples $g^+_i$ and $g^-_i$ so that $g^-_id_ig^+_i$ is the $<$-enumeration of $P_i(\gbar)=P_i(\gbar')$. 
Finally, choose tuples $g^+_{i_*}$, $g_i$, $g^-_i$ so that $g^-_id_{i_*}g_ie_{i_*}g_i$ $<$-enumerates $P_{i_*}(\gbar)$ and $g^-_id_{i_*}g_ie'g_i$ $<$-enumerates $P_{i_*}(\gbar')$. Altogether:
\begin{enumerate}[\hspace{5pt}$\ast$]
\item $\gbar = g^-_1d_1g^+_1\ldots g^-_{i_*-1}d_{i_*-1}g^+_{i_*-1}g^-_{i_*}d_{i_*}g_{i_*}e_{i_*}g^+_{i_*}g^-_{i_*+1}d_{i_*+1}g^+_{i_*+1}\ldots g^-_kd_kg^+_k g_{k+1}$,
\item $\gbar' = g^-_1d_1g^+_1\ldots g^-_{i_*-1}d_{i_*-1}g^+_{i_*-1}g^-_{i_*}d_{i_*}g_{i_*}e'g^+_{i_*}g^-_{i_*+1}d_{i_*+1}g^+_{i_*+1}\ldots g^-_kd_kg^+_k g_{k+1}$,
\item $\dbar$ and $\ebar$ are $<_*$-adjacent in $\gbar$ with $\dbar<_k \ebar$, and 
\item $\dbar$ and $\ebar'$ are $<_*$-adjacent in $\gbar'$ with $\ebar'<_k\dbar$.
\end{enumerate}
Let $\gamma$ be the tuple $g^-_1g^+_1\ldots g^-_{i_*-1}g^+_{i_*-1}g^-_{i_*}g_{i_*}g^+_{i_*}g^-_{i_*+1}g^+_{i_*+1}\ldots g^-_kg^+_kg_{k+1}$ (i.e., $\gamma=\gbar\backslash \dbar e_{i_*}=\gbar'\backslash\dbar e'$). 

 Let $\cC=\gbar\oless_{\dbar,\ebar}\cH^*$, where $\cH^*$ is a copy of $\cH_k$ (with domain disjoint from $\gbar$). Note also that by the above description of $\gbar$ and $\gbar'$, $\cC$ is identical to $\gbar'\oless_{\dbar,\ebar'}\cH^*$. By Lemma \ref{lem:Hproduct2}$(c)$ and  universality of $\cH_k$, we may assume $\cC$ is a substructure of $\cH_k$ extending $\gamma$. Note that $\cH^*$ satisfies condition $(i)$ of the lemma.  Moreover, by Lemma \ref{lem:Hproduct2}$(b)$, and since $\cC$ is a substructure of $\cH_k$, we have the following properties:
 \begin{enumerate}[\hspace{5pt}$(1)$]
 \item $\cH^*{\upharpoonright}\cL^Q_k$ embeds in $\cH_k{\upharpoonright}\cL^Q_k$ via the inclusion map.
 \item $(P_{k+1}(\cH^*),<)$ embeds in $(P_{i_*}(\cH_k),<)$ via the inclusion map. 
 \item For any $\cbar\in P_{\dbar,\ebar}(\cH^*)$, $\sbar\in Q(\cH^*)$, and $t\in P_{k+1}(\cH^*)$, $\cH^*\models R(\sbar,t)$ if and only if $\cH_k\models \sbar<_k \cbar(t)$.
 \end{enumerate}
By construction of $\cC$ (in Definition \ref{def:Hproduct2}), we also have that for any $1\leq i\leq k+1$, all elements of $P_i(\cH^*)$ have the same $<$-cut in $\gamma$, which yields condition $(ii)$.

 Given $\sbar\in Q(\cH^*)$ and $t\in P_{k+1}(\cH^*)$, let
 \[
 g_{\sbar,t}=g^-_1s_1g^+_1\ldots g^-_{i_*-1}s_{i_*-1}g^+_{i_*-1}g^-_{i_*}s_{i_*}g_{i_*}tg^+_{i_*}g^-_{i_*+1}s_{i_*+1}g^+_{i_*+1}\ldots g^-_ks_kg^+_k g_{k+1}.
 \]
 By construction, $g_{\sbar,t}$ $<$-enumerates the substructure of $\cC$ on $\gamma\sbar t$. \medskip
 
 \noindent\emph{Claim.} Suppose $\sbar\in Q(\cH^*)$ and $t\in P_{k+1}(\cH^*)$.
 \begin{enumerate}[$(a)$]
 \item If $\cH^*\models R(\sbar,t)$ then $\qftp^{\cH_k}_{\cL_k}(g_{\sbar,t})=\qftp^{\cH_k}_{\cL_k}(\gbar)$.
 \item If $\cH^*\models \neg R(\sbar,t)$ then $\qftp^{\cH_k}_{\cL_k}(g_{\sbar,t})=\qftp^{\cH_k}_{\cL_k}(\gbar')$.
 \end{enumerate}
 
 \noindent\textit{Proof.} Part $(i)$. By Lemma \ref{lem:Hproduct2}$(d)$, the map $f$ from $\gbar$ to $\gbar_{\sbar,t}$, which fixes $\gamma$ pointwise and sends $\dbar e_{i_*}$ to $\sbar t$, is an $(\dbar,\ebar)$-deficient $\cL_k$-isomorphism. Note also that $f(\ebar)=\bbar(t)$ where $\bbar\in P_{\dbar,\ebar}(\cH^*)$ is such that, for $i\in [k]\backslash\{i_*\}$, $b_i=e_i$ if $e_i\neq d_i$ and $b_i=s_i$ if $e_i=d_i$. Now suppose $\cH^*\models R(\sbar,t)$. Then $\cH_k\models \sbar<_k \bbar(t)$ by $(3)$, i.e., $\cH_k\models f(\dbar)<_k f(\ebar)$. Since $\cH_k\models\dbar<_k \ebar$, it follows that $f$ is an $\cL_k$-isomorphism.
 
 Part $(ii)$. The argument is similar to $(i)$, keeping in mind that $\cC$ is also equal to $\gbar'\oless_{\dbar,\ebar'}\cH^*$. By Lemma \ref{lem:Hproduct2}$(d)$, the map $f$ from $\gbar'$ to $\gbar_{\sbar,t}$, which fixes $\gamma$ pointwise and sends $\dbar e'$ to $\sbar t$, is an $(\dbar,\ebar')$-deficient $\cL_k$-isomorphism. Note also that $f(\ebar')=\bbar(t)$ where $\bbar\in P_{\dbar,\ebar'}(\cH^*)=P_{\dbar,\ebar}(\cH^*)$ is such that, for $i\in [k]\backslash\{i_*\}$, $b_i=e_i$ if $e_i\neq d_i$ and $b_i=s_i$ if $e_i=d_i$.   Now suppose $\cH^*\models \neg R(\sbar,t)$. Then $\cH_k\models \bbar(t)<_k \sbar$ by $(3)$ (and the fact that $\sbar$ and $\bbar(t)$ disagree at the $i_*\uth$ coordinate, hence are distinct tuples), i.e., $\cH_k\models f(\ebar')<_k f(\dbar)$. Since $\cH_k\models \ebar'<_k \dbar$, it follows that $f$ is an $\cL_k$-isomorphism.\clqed\medskip
 
 Now we return to the $\cH_k$-indiscernible sequence $(a_g)_{g\in H_k}$. Recall that $a_{\gbar}\not\equiv a_{\gbar'}$.  So there is an $\cL$-formula $\theta(w)$ such that $\M\models \theta(a_{\gbar})\wedge\neg\theta(a_{\gbar'})$. Recall also that $(\ell_1,\ldots,\ell_{k+1})$ is the length assignment of $(a_g)_{g\in H_k}$. We now partition the free variables $w$ according to the enumerations of $\gbar$ and $\gbar'
 $ described above. In particular, we can write
 \[
 w= y^-_1x_1y^+_1\ldots y^-_{i_*-1}x_{i_*-1}y^+_{i_*-1}y^-_{i_*}x_{i_*}y_{i_*}x_{k+1}y^+_{i_*}y^-_{i_*+1}x_{i_*+1}y^+_{i_*+1}\ldots y^-_kx_ky^+_k y_{k+1}
 \]
 where:
 \begin{enumerate}[\hspace{5pt}$\ast$]
 \item for $1\leq i\leq k$, $|y^-_i|=|a_{g^-_i}|$ and $|y^+_i|=|a_{g^+_i}|$,
 \item $|y_{i_*}|=|a_{g_{i_*}}|$ and $|y_{k+1}|=|a_{g_{k+1}}|$,
 \item for $1\leq i\leq k$, $|x_i|=|a_{d_i}|=\ell_i$, and
 \item $|x_{k+1}|=|a_{e_{i_*}}|=\ell_{i_*}$.
 \end{enumerate}
 Let $y=y^-_1y^+_1\ldots y^-_{i_*-1}y^+_{i_*-1}y^-_{i_*}y_{i_*}y^+_{i_*}y^-_{i_*+1}y^+_{i_*+1}\ldots y^-_ky^+_k y_{k+1}$. Then after permuting the free variables, we can write $\theta(w)$ as $\varphi(x_1,\ldots,x_{k+1};y)$. Note that we have condition $(iii)$ of the lemma. So to finish the proof in this case, we need to show that $\varphi(x_1,\ldots,x_{k+1};a_\gamma)$ $R$-codes $\cH^*$ witnessed by $(a_g)_{g\in H^*}$. 
 
First note that for any $\sbar\in Q(\cH^*)$ and $t\in P_{k+1}(\cH^*)$, the formula $\theta(a_{g_{\sbar,t}})$ is $\varphi(a_{s_1},\ldots,a_{s_k},a_{t};a_\gamma)$. Now fix $\sbar\in Q(\cH^*)$ and $t\in P_{k+1}(\cH^*)$. By parts $(i)$ and $(ii)$ of the claim, and $\cH_k$-indiscernibility of $(a_g)_{g\in H^*}$, we have that $\cH^*\models R(\sbar,t)$ implies $a_{g_{\sbar,t}}\equiv a_{\gbar}$, and $\cH^*\models \neg R(\sbar,t)$ implies $a_{g_{\sbar,t}}\equiv a_{\gbar'}$. Since $\M\models \theta(a_{\gbar})\wedge\neg\theta(a_{\gbar'})$, it follows that $\M\models a_{g_{\sbar,t}}$ if and only if $\cH^*\models R(\sbar,t)$, i.e., $\M\models \varphi(a_{s_1},\ldots,a_{s_k},a_t;a_\gamma)$ if and only if $\cH^*\models R(\sbar,t)$, as desired. 
\end{proof}

The statement of the previous lemma includes extra details that will be used in our main application (in the next subsection). But in any case, we now have our desired characterization of $\NFOP_k$ theories via collapse of indiscernibles.

\begin{theorem}\label{thm:genind}
Let $T$ be a complete theory with monster model $\M$. The following are equivalent.
\begin{enumerate}[$(i)$]
\item $T$ is $\NFOP_k$.
\item Every $\cH_k$-indexed indiscernible sequence in $\M$ is $\cL'_k$-indiscernible.
\item Every $\cH_k$-indexed indiscernible sequence in $\M$ is $\cL''_k$-indiscernible.
\end{enumerate}
\end{theorem}
\begin{proof}
 $(i)\Rightarrow (iii)$ follows from Lemma \ref{lem:collapse1} and Proposition \ref{prop:codeTk}. $(iii)\Rightarrow (ii)$ follows from the fact that $\cL''_k$-indiscernibility implies $\cL'_k$-indiscernibility. $(ii)\Rightarrow (i)$ is Corollary \ref{cor:genind}.
\end{proof}

\begin{remark}
The previous theorem implies that if $T$ is $\NFOP_k$ then every $\cH_k{\upharpoonright}\cL'_k$-indexed indiscernible sequence in $\M$ is $\cL''_k$-indiscernible (in particular, any $\cH_k{\upharpoonright}\cL'_k$-indexed indiscernible sequence canonically determines an $\cH_k$-indexed indiscernible sequence). However, we expect that the converse fails, with $\cH_k$ itself being a counterexample. This hypothesis is along the same lines as the  discussion of $\cF_k$ in Remark \ref{rem:Fk} and, once again, we will not pursue any details at this time.
\end{remark}

\subsection{Application: reduction of $\NFOP_k$ to one variable}

In this subsection, we use indiscernible collapse to show that when proving  a complete theory $T$ is $\NFOP_k$, it suffices to consider formulas in which the first $k$ variables are singletons. Lemma \ref{lem:collapse1} will provide the main technical result needed for the proof. That lemma aside, the coarse proof structure is analogous to similar results for $\IP_k$ from \cite{CPT} and \cite{CH}. Throughout this subsection, $T$ is a complete $\cL$-theory with monster model $\M$.

 We will be comparing $\cL_k$-structures to $\cL_{k-1}$-structures. So for the sake of clarity, we assume $k\geq 2$ and we let $R_k$ and $R_{k-1}$ denote the hypergraph relations in $\cL_k$ and $\cL_{k-1}$. Given a $\cL_k$-structure $\cH$ and some $e\in P_1(\cH)$,  the \textbf{$\cL_{k-1}$-structure induced by $e$ in $\cH$} is the $\cL_{k-1}$-structure $\cH^e$ defined as follows:
\begin{enumerate}[\hspace{5pt}$\ast$]
\item For $1\leq i\leq k$, $(P_i(\cH^e),<)$ is $(P_{i+1}(\cH),<)$.
\item For $\gbar,\gbar'\in Q(\cH^e)$, $\cH^e\models \gbar<_{k-1}\gbar'$ if and only if $\cH\models e\gbar <_k e\gbar' $.
\item For $\gbar\in P(\cH^e)\times\ldots\times P_k(\cH^e)$, $\cH^e\models R_{k-1}(\gbar)$ if and only if $\cH\models R_k(e\gbar)$. 
\end{enumerate} 
It is straightforward to check that if $\cH\models T_k$ then $\cH^e\models T_{k-1}$. In the case of the Fra\"{i}ss\'{e} limit $\cH_k$, we note that for any choice of $e$, $\cH_k^e$ satisfies the one-point extension axioms for $\Th(\cH_{k-1})$ (see \cite[Section 2]{KruckAM} for a general exposition of this topic). Thus by $\aleph_0$-categoricity we can identify $\cH^e_k$ and $\cH_{k-1}$.

\begin{lemma}\label{lem:var}
Given an $\cL$-formula $\varphi(x_0,\ldots,x_k)$, the following are equivalent.
\begin{enumerate}[$(i)$]
\item $\varphi(x_0,\ldots, x_k)$ has $\FOP_k$ in $T$.
\item There is a parameter $c\in \M^{x_0}$ such that $\varphi(c,x_1,\ldots,x_k)$ $R_{k-1}$-codes $\cH_{k-1}$ witnessed by an $\cL''_{k-1}$-indiscernible sequence $(a_g)_{g\in H_{k-1}}$ in $\M$. 
\item Condition $(ii)$ holds where $(a_g)_{g\in H_{k-1}}$ is  also  $\cL_{k-1}$-indiscernible over $c$. 
\end{enumerate}
\end{lemma}
\begin{proof}
Let $\qftp_k=\qftp^{\cH_k}_{\cL_k}$ and $\qftp''_k=\qftp^{\cH_k}_{\cL''_k}$ (and similarly for $k-1$).

$(i)\Rightarrow (ii)$. Assume $(i)$. Then by Proposition \ref{prop:codeTk}, $\varphi(x_0,\ldots,x_k)$ $R_k$-codes $\cH_k$ witnessed by an $\cH_k$-indiscernible sequence $(a_g)_{g\in H_k}$ in $\M$. 

Fix some $e\in P_1(\cH_k)$, and let $\cH_{k-1}$ be identified with the $\cL_{k-1}$-structure induced by $e$ in $\cH_k$, as described above. 
Note that the domain of $\cH_{k-1}$ is a subset of that of $\cH_k$. Setting $c=a_e$, it follows by assumption on $(a_g)_{g\in H_k}$, and our construction of $\cH_{k-1}$, that $(a_g)_{g\in H_{k-1}}$ witness that $\varphi(c,x_1,\ldots,x_k)$ $R_{k-1}$-codes $\cH_{k-1}$. Finally, we show that $(a_g)_{g\in H_{k-1}}$ is $\cL''_{k-1}$-indiscernible. Fix $g_1,\ldots,g_n,h_1,\ldots,h_n\in H_{k-1}$ such that $\qftp''_{k-1}(g_1,\ldots,g_n)=\qftp''_{k-1}(h_1,\ldots,h_n)$. Then we immediately have $\qftp''_k(g_1,\ldots,g_n)=\qftp''_k(h_1,\ldots,h_n)$. Moreover, in $\cH_k$, the relations $<_k$ and $R_k$ only hold on tuples whose first coordinate is from $P_1(\cH_k)$. Since $H_{k-1}$ does not intersect $P_1(\cH_k)$, it then follows that $\qftp_k(g_1,\ldots,g_n)=\qftp_k(h_1,\ldots,h_n)$, and thus $\tp^T_{\cL}(a_{g_1},\ldots,a_{g_n})=\tp^T_{\cL}(a_{h_1},\ldots,a_{h_n})$ since $(a_g)_{g\in H_k}$ is $\cL_k$-indiscernible.

$(ii)\Rightarrow (iii)$. Assume $(ii)$. We work in $T$ with constants named for $c$.  By the modeling property for $\cH_{k-1}$ (Corollary \ref{cor:Gkmod}), there is an $\cH_{k-1}$-indiscernible sequence $(a'_g)_{g\in H_{k-1}}$  in $\M$ locally based on $(a_g)_{g\in H_{k-1}}$. So, in particular, $(a'_g)_{g\in H_k}$ is $\cL_{k-1}$-indiscernible over $c$. Moreover, since $(a'_g)_{g\in H_{k-1}}$ is locally based on $(a_g)_{g\in H_{k-1}}$, it follows that $(a'_g)_{g\in H_{k-1}}$ also satisfies the conditions of $(ii)$, namely, it is $\cL''_{k-1}$-indiscernible and witnesses that $\varphi(c,x_1,\ldots,x_k)$ $R_{k-1}$-codes $\cH_{k-1}$.

$(iii)\Rightarrow (i)$. Assume $(iii)$, witnessed by $c\in\M^{x_0}$ and $(a_g)_{g\in H_{k-1}}$.
Fix a finite model $\cH\models T_k$. We will show that $\varphi(x_1,\ldots,x_{k+1})$ $R_k$-codes $\cH$, and so $(i)$ follows by Proposition \ref{prop:codeTk}. We may assume $(P_i(\cH_{k-1}),{<})=(P_{i+1}(\cH_k),{<})$ for all $1\leq i\leq k$. We also view $\cH$ as a substructure of $\cH_k$.
For $1\leq i\leq k$, let $\bar{g}_i$ $<$-enumerate $P_{i+1}(\cH)$.  For each $e\in P_1(\cH)$, let $\cH^e$ be the $\cL_{k-1}$-structure induced by $e$ in $\cH$, and note that $\cH^e\models T_{k-1}$. Moreover the $<$-order in $\cH^e$ is induced by the $<$-order in $\cH_k$, which agrees with the $<$-order on $\cH_{k-1}$. In particular, given $1\leq i\leq k$, $\bar{g}_i$ $<$-enumerates $P_i(\cH^e)$ for all $e\in P_k(\cH)$. 

Fix $e\in P_k(\cH)$. Note that  the domain of $\cH^e$ is a subset of $H_{k-1}$, but $\cH^{e}$ need not coincide with the substructure induced from $\cH_{k-1}$ on this set. However, there is some set $H^e_*\seq H_{k-1}$ such that if $\cH^e_*$ is the substructure of $\cH_{k-1}$ on $H^e_*$, then $\cH^e$ and $\cH^e_*$ are isomorphic via the unique $<$-order preserving map from  the domain of $\cH^e$ to $H^e_*$, which we call $f_e$. Now we have  $\qftp''_{k-1}(\bar{g}^1,\ldots,\bar{g}^k)=\qftp''_{k-1}(f_e(\bar{g}^1),\ldots,f_e(\bar{g}^k))$, and so $\tp^T_{\cL}(a_{\bar{g}^1},\ldots,a_{\bar{g}^k})=\tp^T_{\cL}(a_{f_e(\bar{g}_1)},\ldots,a_{f_e(\bar{g}_k)})$. Pick some $c_e\in \M^{x_0}$ such that $\tp^T_{\cL}(c_e,a_{\bar{g}^1},\ldots,a_{\bar{g}^k})=\tp^T_{\cL}(c,a_{f_e(\bar{g}^1)},\ldots,a_{f_e(\bar{g}^k)})$. 

Now, given $e\in P_1(\cH)$ and $g_t\in \bar{g}^t$ for $1\leq t\leq k$, we have
\begin{align*}
\M\models \varphi(c_e,a_{g_1},\ldots,a_{g_k}) &\miff \M \models \varphi(c,a_{f_e(g_1)},\ldots,a_{f_e(g_k)})\\
&\miff \cH^e_*\models R_{k-1}(f_e(g_1),\ldots,f_e(g_k))\\
&\miff \cH^e\models R_{k-1}(g_1,\ldots,g_k)\\
&\miff \cH\models R_k(e,g_1,\ldots,g_k).
\end{align*}
So $((c_e)_{e\in P_1(\cH)},a_{\bar{g}^1},\ldots,a_{\bar{g}^k})$ witnesses that $\varphi(x_0,\ldots,x_k)$ $R_k$-codes $\cH$. 
\end{proof}

\begin{lemma}\label{lem:var2}
Fix partitioned variables $(x_0,\ldots,x_k)$. Suppose there is a parameter $c\in \M^{x_0}$ and a sequence $(a_g)_{g\in H_{k-1}}$ such that:
\begin{enumerate}[\hspace{5pt}$\ast$]
\item $a_g\in \M^{x_i}$ for all $1\leq i\leq k$ and $g\in P_i(\cH_{k-1})$,
\item $(a_g)_{g\in H_{k-1}}$ is $\cL_{k-1}$-indiscernible over $c$, but not $\cL''_{k-1}$-indiscernible over $c$, and 
\item $(a_g)_{g\in H_{k-1}}$ is $\cL''_{k-1}$-indiscernible over $\emptyset$.
\end{enumerate}
Then there is a formula $\varphi(x_0,\ldots,x_{k-1},y_k)$ with $\FOP_k$ in $T$.
\end{lemma}
\begin{proof}
Applying Lemma \ref{lem:collapse1} with $k-1$ and with $c$ added as a constant, we get a formula $\varphi(x_0,x_1,\ldots,x_{k-1},x^*_k;y)$, a tuple $\gamma$ from $H_{k-1}$ and an $\cL_{k-1}$-structure $\cH^*$ satisfying the following properties.
\begin{enumerate}[$(i)$]
\item $\cH^*$ is isomorphic to $\cH_{k-1}$ and $H^*$ is a subset of $H_{k-1}$ disjoint from $\gamma$.
\item For any $g_1,\ldots,g_n,h_1,\ldots,h_n\in H^*$,
\begin{multline*}
\qftp^{\cH^*}_{\cL''_{k-1}}(g_1,\ldots,g_n)=\qftp^{\cH^*}_{\cL''_{k-1}}(h_1,\ldots,h_n)\\
\mimp \qftp^{\cH_{k-1}}_{\cL''_{k-1}}(g_1,\ldots,g_n,\gamma)=\qftp^{\cH_{k-1}}_{\cL''_{k-1}}(h_1,\ldots,h_n,\gamma).
\end{multline*}
\item $|x^*_k|=|x_{i_*}|$ for some $1\leq i_*\leq k$.
\item $\varphi(c,x_1,\ldots,x_{k-1},x^*_k,a_\gamma)$ $R$-codes $\cH^*$ witnessed by $(a_g)_{g\in H^*}$.
\end{enumerate}
By $(ii)$, and since $(a_g)_{g\in H_{k-1}}$ is $\cL''_{k-1}$-indiscernible over $\emptyset$, it follows that $(a_g)_{g\in H^*}$ is $\cL''_{k-1}$-indiscernible over $a_\gamma$.  Therefore, working with constants added for $a_\gamma$, the formula $\varphi(x_0,\ldots,x_{k-1},x^*_k,a_\gamma)$ satisfies condition $(ii)$ of Lemma \ref{lem:var}. Thus $\varphi(x_0,\ldots,x_{k-1},x^*_k,a_\gamma)$ has $\FOP_k$ in $T$ (with constants for $a_\gamma$). Setting $y_k=x^*_ky$, it follows that $\varphi(x_0,\ldots,x_{k-1},y_k)$ has $\FOP_k$ in $T$. 
\end{proof}

\begin{corollary}\label{cor:var}
Suppose $\varphi(x_0,\ldots,x_k)$ has $\FOP_k$ in $T$. Then there is a formula $\psi(z_0,x_1,\ldots,x_{k-1},y_k)$ with $\FOP_k$ in $T$ where $|z_0|=1$.
\end{corollary}
\begin{proof}
We proceed by induction on $|x_0|$, where the base case $|x_0|=1$ is trivial. Assume now $|x_0|>1$ and  the claim holds for any formula $\varphi'(x'_0,x'_1,\ldots,x'_k)$ with $|x'_0|<|x_0|$. Since $\varphi(x_0,x_1,\ldots,x_k)$ has $\FOP_k$ in $T$ we can apply Lemma \ref{lem:var} to obtain some $c\in \M^{x_0}$ such that $\varphi(c,x_1,\ldots,x_k)$ $R_{k-1}$-codes $\cH_{k-1}$ witnessed by an $\cL''_{k-1}$-indiscernible sequence $(a_g)_{g\in H_{k-1}}$, which is also $\cL_{k-1}$-indiscernible over $c$. Moreover, it follows from  $R_{k-1}$-coding that $(a_g)_{g\in H_{k-1}}$ is not $\cL''_{k-1}$-indiscernible over $c$ (or even $\cL'_{k-1}$-indiscernible over $c$, as in the proof of Corollary \ref{cor:genind}). Say $c=c'c''$ with $|c'|,|c''|\geq 1$, and let $x_0=x'_0x''_0$ be the corresponding partition of $x_0$. Note that $(a_g)_{g\in H_{k-1}}$ is $\cL_{k-1}$-indiscernible over $c''$.  

First suppose that $(a_g)_{g\in H_{k-1}}$ is not $\cL''_{k-1}$-indiscernible over $c''$. Then $c''$ and $(a_g)_{g\in H_{k-1}}$ satisfy condition $(ii)$ of Lemma \ref{lem:var2}, which implies there is a formula $\varphi'(x''_0,x_1,\ldots,x_{k-1},y'_k)$ with $\FOP_k$ in $T$. Since $|x''_0|<|x_0|$, we apply induction to obtain a formula $\psi(z_0,x_1,\ldots,x_{k-1},y_k)$ with $\FOP_k$ in $T$ and $|z_0|=1$. 

Finally, suppose that  $(a_g)_{g\in H_{k-1}}$ is $\cL''_{k-1}$-indiscernible over $c''$. For each $g\in P_k(\cH_{k-1})$, let $a^*_g=a_gc''$ and set $a^*_g=a_g$ for all other $g$. Then $(a^*_g)_{g\in H_{k-1}}$ is:
\begin{enumerate}[\hspace{5pt}$\ast$]
\item $\cL''_{k-1}$-indiscernible over $\emptyset$ (since $(a_g)_{g\in H_{k-1}}$ is $\cL''_{k-1}$-indiscernible over $c''$),
\item $\cL_{k-1}$-indiscernible over $c'$ (since $(a_g)_{g\in H_{k-1}}$ is $\cL_{k-1}$-indiscernible over $c$), and
\item not $\cL''_{k-1}$-indiscernible over $c'$ (since $(a_g)_{g\in H_{k-1}}$ is not $\cL''_{k_1}$-indiscernible over $c$).
\end{enumerate}
Thus  $c'$ and $(a^*_g)_{g\in H_{k-1}}$ satisfy condition $(ii)$ of Lemma \ref{lem:var2}. So there is a formula $\varphi'(x'_0,x_1,\ldots,x_{k-1},y'_k)$ with $\FOP_k$ in $T$. Since $|x'_0|<|x_0|$, we apply induction to obtain a formula $\psi(z_0,x_1,\ldots,x_{k-1},y_k)$ with $\FOP_k$ in $T$ and $|z_0|=1$.
\end{proof}

\begin{theorem}\label{thm:FOPkone}
If $T$ has $\FOP_k$ then there is a formula $\varphi(x_1,\ldots, x_{k+1})$ with $\FOP_k$ in $T$ and with $|x_i|=1$ for all $1\leq i\leq k$.
\end{theorem}
\begin{proof}
Suppose $\varphi'(x'_1,\ldots, x'_{k+1})$ has $\FOP_k$ in $T$.  Iteratively applying Corollary \ref{cor:var}, and using Corollary \ref{cor:easyFOPk} to exchange the roles of the  variables $x'_1,\ldots, x'_k$, we obtain a formula $\varphi(x_1,\ldots, x_{k+1})$ with $\FOP_k$ in $T$ and  $|x_i|=1$ for all $1\leq i\leq k$.
\end{proof}

\section{Vector spaces with a bilinear form}

The primary goal of this section is to prove that certain theories of infinite dimensional vector spaces equipped with a non-degenerate bilinear form (satisfying further properties) are $\NFOP_2$ if and only if the associated field is stable. This is a stability/$\FOP$-analogue of the corresponding result due to Chernikov and Hempel \cite{CH}, who showed that the theories of such  vector spaces  are $\NIP_2$ if and only if the associated field is NIP. The key tools used in \cite{CH} for that result are the $\NIP_2$ ``composition lemma" and a quantifier elimination result due to Granger \cite{Gr}. Our proof strategy is similar to that of \cite{CH}, although we are able to simplify some of the technical step through a slightly different approach (see Remark \ref{rem:terms}). 

Moreover, various mistakes in Granger's quantifier elimination and completeness results have recently come to light.  Correcting these errors requires one to work in a different language. Thus in Section \ref{sec:VS}, we will set out the new language, and prove our main results about theories of vector spaces with bilinear forms, assuming quantifier elimination and completeness in the corrected language. In Section \ref{sec:QE}, we will provide a corrected proof of Granger's quantifier elimination result and completeness results, along with some further results in the general $k$-linear case.

\subsection{Vector spaces}\label{sec:VS}
We start with some basic definitions.

\begin{definition}
Suppose $V$ is a vector space over a field $F$, and $k\geq 2$ is an integer.  A \textbf{$k$-linear form} is a function $f:V^k\rightarrow F$ such that for any $1\leq i\leq k$ and $v_1,\ldots,v_{k-1}\in V$, the map $f(v_1,\ldots,v_{i-1},x,v_i,\ldots,v_{k-1})$ from $V$ to $F$ is linear.
\end{definition}

When $k=2$ we will use the (standard) terminology ``bilinear form". In this case, we have the following terminology. 

\begin{definition}
Let $V$ be a vector space over a field $F$ and suppose $\beta\colon V\times V\to F$ is a bilinear form.
\begin{enumerate}[$(1)$]
\item $\beta$ is \textbf{symmetric} if $\beta(v,w)=\beta(w,v)$ for all $v,w\in V$.
\item $\beta$ is \textbf{alternating} if $\beta(v,v)=0$ for all $v\in V$.
\item Assume $\beta$ is symmetric or alternating. Then $\beta$ is \textbf{non-degenerate} if for all nonzero $v\in V$ there is some $w\in V$ such that $\beta(v,w)\neq 0$.
\end{enumerate}
\end{definition}

Rather than delve into a precise generalization of non-degeneracy for arbitrary $k$, we will instead define a weaker ``genericity" property, which suffices for our purposes (and, in the $k=2$ case follows from non-degeneracy; see Fact \ref{fact:generic}).

\begin{definition}\label{def:generic}
Suppose $V$ is a vector space over a field $F$, and $k\geq 2$ is an integer. Then a $k$-linear map $f\colon V^k\to F$ is \textbf{generic} if for any $n\geq 1$ and any function $\sigma\colon [n]^k\to F$, there is an array $(v_{i,j})_{i\in[k],j\in[n]}$ from $V$ such that $f(v_{1,j_1},\ldots,v_{k,j_k})=\sigma(j_1,\ldots,j_k)$ for all $(j_1\ldots,j_k)\in [n]^k$.
\end{definition}

The next result shows that non-degeneracy is a much stronger assumption than genericity for a bilinear form. 

\begin{fact}\label{fact:generic}
 Suppose $\beta$ is a non-degenerate bilinear form on an infinite dimensional vector space $V$ over a field $F$. Then for any $n\geq 1$, any linearly independent $w_1,\ldots,w_n\in V$, and any $a_1,\ldots,a_n\in F$, there is some $v\in V$ such that $\beta(v,w_i)=a_i$ for all $1\leq i\leq n$. In particular, $\beta$ is generic. 
\end{fact}
\begin{proof}
Note that the first statement implies genericity since, given $\sigma\colon[n]^2\to F$, we can let $v_{2,j}=w_j$ for $1\leq j\leq n$, and then, for fixed $1\leq i\leq n$, apply the statement with $a_j=\sigma(i,j)$ for $1\leq j\leq n$ to find the desired $v_{1,i}$. So now suppose $w_1,\ldots,w_n\in V$ and $a_1,\ldots,a_n\in F$ are as in the statement.  Let $w_1,\ldots,w_n$ be linearly independent vectors in $V$. For each $i$, consider the  linear map $L_i\colon V\to F$ such that $L_i(x)=\beta(x,w_i)$. Then since $\beta$ is non-degenerate, linear independence of $w_1,\ldots,w_n$ says precisely that $L_1,\ldots,L_n$ are linearly independent (as elements of the dual vector space of all linear functions from $V$ to $F$). In this case, it follows that the linear map $L\colon V\to F^n$ given by $(L_1,\ldots,L_n)$ has rank $n$, i.e., is surjective. So we can pick $v\in V$ such that $L(v)=(a_1,\ldots,a_n)$, as desired. 
\end{proof}

Next we establish the first-order setting for studying vector spaces with multilinear forms. We will work in a two-sorted setting, where $F$ denotes the field sort, and $V$ denotes the vector space sort. Let 
\[
\calL_{\textit{VF}}=\{V,F;+_F,\cdot_F,-_F, ()\inv_F, 0_F,1_F,+_V,0_V, \cdot_V\},
\]
where $+_F,\cdot_F$ are binary functions on the $F$-sort, $-_F, ()\inv_F$ are unary function symbols on the $F$-sort,  $+_V$ is a binary function symbol on the $V$-sort, $0_F,1_F,0_V$ are constant symbols in the indicated sorts, and $\cdot_V$ is a binary function symbol from $F\times V$ to $V$. Let $T_{\textit{VF}}$ be the $\calL_{\textit{VF}}$-theory that says $V$ is a vector space over the field $F$, where $+_F$, $\cdot_F$, $-_F$, $()\inv_F$, $0_F$, and $1_F$ are together interpreted as the field structure, $+_V$ and $0_V$ are together interpreted as the additive structure of the vector space, and $\cdot _V$ is interpreted as scalar multiplication.

Now let $\cL$ be an arbitrary expansion of $\cL_{\textit{VF}}$ which contains a $k$-ary function symbol $f$ from $V^k$ to $F$. Let $T$ be a complete $\cL$-theory which contains $T_{\textit{VF}}$ along with an axiom saying that $f$ is a $k$-linear form. Let $\cL{\upharpoonright}F$ be the one-sorted language consisting of (one-sorted versions of) the symbols in $\cL$ only involving the $F$ sort. Let $T{\upharpoonright}F$ be the complete $\cL{\upharpoonright}F$-theory induced by $T$ on the $F$ sort. In particular, $T{\upharpoonright}F$ is an expansion of the theory of fields. 

Note also that the genericity of $f$ is first-order expressible in $\cL_{\textit{VF}}\cup\{f\}$ (using an infinite axiom scheme).

\begin{proposition}\label{prop:fieldFOP}
If $T$ implies that $f$ is generic and $T{\upharpoonright}F$ is unstable, then $T$ has $\FOP_k$.  
\end{proposition}
\begin{proof}
Assume $T$ implies that $f$ is generic, and suppose $T{\upharpoonright}F$ is unstable. Then there is a formula $\varphi(\ybar,z)$ with the order property in $T$, where $\ybar$ and $z$ are variables of the $F$ sort and $z$ is a singleton  (see \cite[Theorem II.2.13]{Shelah}). We will show that the formula $\varphi(x_1,\ldots,x_k,\ybar) \coloneqq \psi(\ybar,f(x_1,\ldots,x_k))$ has $\FOP_k$ in $T$ by showing part $(iii)$ of Proposition \ref{prop:alleq}. By compactness, it is enough to show that for any $s<\omega$, arbitrarily large $N$, and partition $[N]^k=E_1\cup\ldots\cup E_s$, there is a model $\cM\models T$ and sequences $(c^1_i)_{i=1}^N$, \ldots, $(c^k_i)_{i=1}^N$, $(\dbar_j)_{j\in[s]}$ such that each  $c^t_i\in \cM^{x_t}$, each  $\dbar_j\in \cM^{\ybar}$, and
\[
\textstyle \cM\models\varphi(c^1_{i_1},\ldots,c^k_{i_k},\dbar_j)\miff (i_1,\ldots,i_k)\in \bigcup_{\ell\geq j}E_\ell.
\]

 Fix $s$, $N$ and a partition as above. Let $u(i_1,\ldots, i_k)$ be the $u\in[s]$ so that $(i_1,\ldots, i_k)\in E_u$. By the order property for $\varphi$, there is $\cM\models T$ with $(\abar_i)_{i=1}^N\subseteq \cM^{\ybar}$ and $(b_j)_{j=1}^N\subseteq \cM^z$ such that $\psi(\abar_i,b_j)$ holds if and only if $i\leq j$. By the fact that $f$ is generic, we get $(v_{j,i})_{j\in[k],i\in[n]}$ so that $f(v_{1,i_1},v_{2,i_2},\ldots, v_{k,i_k})=b_{u(i_1,\ldots,i_k)}$. Let $c^j_i=v_{j,i}$ for $j\in[k]$ and $i\in[N]$, and let $\dbar_j=\abar_j$ for $j\in[s]$. We then have that $\varphi(c^1_{i_1},\ldots,c^k_{i_k},\dbar_j) = \psi(\abar_j,f(v_{1,i_1},v_{2,i_2},\ldots, v_{k,i_k})) = \psi(\abar_j, b_{u(i_1,\ldots, i_k)})$ and so
\[
\textstyle \cM\models\varphi(c^1_{i_1},\ldots,c^k_{i_k},\dbar_j)\miff j\leq u(i_1,\ldots, i_k)\miff (i_1,\ldots,i_k)\in \bigcup_{\ell\geq j}E_\ell.\qedhere
\]
\end{proof}

\begin{remark}
    Using Proposition \ref{prop:IPkeq}$(iii)$ in place of Proposition \ref{prop:alleq}$(iii)$,  the same proof shows that if $T$ implies $f$ is generic, and $T{\upharpoonright}F$ has $\IP$, then $T$ has $\IP_k$. 
\end{remark}

Now let $\cL'=\cL\backslash\{f\}$, and let $T'$ be the reduct of $T$ to $\cL'$.  We focus next on the question of transferring tameness from $T'$ to $T$. The key tool needed for this is the composition lemma in Theorem \ref{thm:FOP2comp}, which we only currently have in the $k=2$ case. For this reason, the remaining results in this subsection will be largely restricted to the $k=2$. However, in order to apply Theorem \ref{thm:FOP2comp}, we first need a lemma allowing us to simplify terms involving the bilinear form. Since this is not sensitive to the arity, we will prove this part for arbitrary $k$ (see Lemma \ref{lem:terms} below).

\begin{definition}\label{def:fnbar}
Suppose $\xbar=(x_1,\ldots, x_n)$ is a sequence of singleton variables and $h$ is a function  of arity $m$.  Then we write $\boldsymbol{h}(\xbar)$ to denote the \emph{tuple} 
$$
(h(x_{i_1},\ldots, x_{i_m}))_{(i_1,\ldots, i_m)\in [n]^m}
$$
enumerated by the lexicographic order on $[n]^m$.
\end{definition}

In what follows, we will use this notation where $h$ is our $k$-linear form $f$.

\begin{definition}
A language expanding $\cL_{\textit{VF}}$ is \textbf{$V$-conservative} if it contains no new constant symbols in the $V$-sort and no new function symbols into the $V$-sort.
\end{definition}

\noindent For a tuple of variables $\xbar$, let $\xbar_V$ be the subtuple consisting of vector-sort variables. 

\begin{lemma}    \label{lem:terms}
Assume $\cL$ is $V$-conservative, and let $t(\xbar)$ be an $\calL$-term.  Then there is an $\calL'$-term  $t'(\xbar,\ubar)$ with $|\ubar|=|\boldsymbol{f}(\xbar_V)|$ so that $T\models \forall \xbar(t(\xbar)=t'(\xbar,\boldsymbol{f}(\xbar_V)))$.
\end{lemma}

\begin{proof}
We proceed by induction on term length. If $|t(\xbar)|=1$, then it is a constant or a variable, and thus already an $\calL'$-term, and so we are done by setting $t'=t$. 

Assume now that $|t(\xbar)|>1$, and assume by induction that the claim holds for all $\calL$-terms of length less than $|t(\xbar)|$. Suppose first that $t(\xbar)=h(t_1(\xbar),\ldots, t_m(\xbar))$ where $h$ is an $m$-ary function symbol other than $f$, and $t_1,\ldots, t_m$ are $\calL$-terms. Clearly, each of $t_1(\xbar),\ldots, t_m(\xbar)$ have length shorter than $|t(\xbar)|$, so by induction, for each $1\leq j\leq m$, there is an $\calL'$-term $t'_j(\xbar,\ubar)$ so that $T\models \forall \xbar(t_j(\xbar)=t'_j(\xbar,\boldsymbol{f}(\xbar_V)))$. Note $h\in \calL'$ since $h\neq f$. Therefore, setting $t'(\xbar,\boldsymbol{f}(\xbar_V))=h(t'_1(\xbar,\ubar),\ldots, t'_m(\xbar,\ubar))$, we have that $t'(\xbar,\ubar)$ is an $\calL'$-term and $T\models \forall \xbar(t(\xbar)=t'(\xbar,\boldsymbol{f}(\xbar_V)))$. 

Finally, suppose $t(\xbar) = f(t_1(\xbar),\ldots, t_k(\xbar))$, where $t_1(\xbar),\ldots, t_k(\xbar)$ are vector valued $\calL$-terms.   If $|t_1(\xbar)|=\ldots=|t_k(\xbar)|=1$, then each $t_i(\xbar)$ is either a constant or variable.  Since we have added no new constants to the $V$-sort, this means each $t_i(\xbar)$ is either $0_V$ or a variable.  If one of the $t_i(\xbar)$ is $0_V$, then $T\models \forall \xbar (t(\xbar)=0_F)$, so setting $t'(\xbar,\ubar)=0_F$, we are done.  If none of the $t_i(\xbar)$ are $0_V$, then each of $t_1(\xbar),\ldots, t_k(\xbar)$ are variables, and thus $f(t_1(\xbar),\ldots, t_k(\xbar))$ appears in the tuple $\boldsymbol{f}(\xbar_V)$, say in the $j^{\text{th}}$ index.  Letting $t'(\xbar,\ubar)=u_j$, we have that $T\models \forall \xbar (t(\xbar)=t'(\xbar,\boldsymbol{f}(\xbar_V)))$.

Assume now there is some $1\leq i\leq m$ so that $|t_i(\xbar)|>1$.  Since $+_V$ and $\cdot_V$ are the only vector-valued functions in $\calL$, this means $t_i(\xbar)$ has the form $s_i(\xbar)\cdot_Vt'_i(\xbar)$ for some scalar valued term $s_i(\xbar)$ and vector valued term $t'_i(\xbar)$ or $t_i'(\xbar)+_Vt_i''(\xbar)$ for some vector valued terms $t'_i(\xbar)$ and $t_i''(\xbar)$.  Thus one of the following hold.
\begin{enumerate}[$(1)$]
\item $T\models \forall \xbar(t(\xbar) = s_i(\xbar)\cdot_V f(t_1(\xbar),\ldots, t_i'(\xbar),\ldots, t_k(\xbar)))
$, or 
\item $T\models \forall \xbar(t(\xbar) = f(t_1(\xbar),\ldots, t_i'(\xbar),\ldots, t_k(\xbar))+_Vf(t_1(\xbar),\ldots, t_i''(\xbar),\ldots, t_k(\xbar)))$.
\end{enumerate}

In case (1), both $s_i(\xbar)$ and $f(t_1(\xbar),\ldots, t_i'(\xbar),\ldots, t_k(\xbar))$ have shorter length than $t(\xbar)$, so there are $\calL'$-terms $s''(\xbar,\ubar)$ and $t''(\xbar,\ubar)$ such that
\[
 T\models \forall \xbar (f(t_1(\xbar),\ldots, t_i'(\xbar),\ldots, t_k(\xbar))=t''(\xbar,\boldsymbol{f}(\xbar_V)))
 \]
  and $T\models \forall \xbar(s_i(\xbar)=s''(\xbar,\boldsymbol{f}(\xbar_V)))$. In this case, setting $t'(\xbar,\ubar)=s''(\xbar,\ubar)\cdot_Vt''(\xbar,\ubar)$, we have that $t'$ is an $\calL'$-term and 
$$
T\models \forall \xbar(t(\xbar)=t'(\xbar,\boldsymbol{f}(\xbar_V))),
$$
as desired. 

In case (2), both $f(t_1(\xbar),\ldots, t_i'(\xbar),\ldots, t_k(\xbar))$ and $f(t_1(\xbar),\ldots, t_i''(\xbar),\ldots, t_k(\xbar))$ have shorter length than $t(\xbar)$, so there are $\calL'$-terms $t''(\xbar,\ubar)$ and $t'''(\xbar,\ubar)$ such that
\[
 T\models \forall \xbar (f(t_1(\xbar),\ldots, t_i'(\xbar),\ldots, t_k(\xbar))=t''(\xbar,\boldsymbol{f}(\xbar_V))),
 \]
  and $T\models \forall \xbar (f(t_1(\xbar),\ldots, t_i''(\xbar),\ldots, t_k(\xbar))=t'''(\xbar,\boldsymbol{f}(\xbar_V)))$. In this case, setting $t'(\xbar,\ubar)=t''(\xbar,\ubar)+_Vt'''(\xbar,\ubar)$, we have that $t'$ is an $\calL'$-term and 
$$
T\models \forall \xbar(t(\xbar)=t'(\xbar,\boldsymbol{f}(\xbar_V))),
$$
as desired. 
\end{proof}

\begin{remark}\label{rem:terms}
The previous lemma is a variation on a similar ``term reduction" argument done by Chernikov and Hempel in \cite[Section 6.3]{CH} (in the $k=2$ case).  In fact their conclusion is much stronger, reducing the  behavior of $\cL$-terms to those of the form  $t'(\xbar_F, \boldsymbol{f}(\xbar_V))$ for some term $t'(\ubar,\vbar)$ in the \emph{field sort} (see the proof of  Theorem 6.3(2) in \cite{CH}). They then transfer $\NIP$ in the field sort to $\NIP_2$ in a bilinear vector space via their composition lemma, applied to the pure field structure.  However, their proof takes place in a language which was later discovered is not sufficient for quantifier elimination.  The ``correct" language has infinitely many additional function symbols, resulting in a more complicated term structure (see Definition \ref{def:bilineartheories}). While an analogous reduction to terms in the field sort can be carried out in this new language, it is substantially more complicated.  We have chosen not to include this, as such a strong reduction is not strictly necessary.

In our case, Lemma \ref{lem:terms} only provides a reduction to terms that remove the symbol $f$. Although a weaker conclusion, this will still suffice for our main results since, in the setting Granger's examples, tameness in the field sort can be preserved in the pure vector space without the bilinear form (see Fact \ref{fact:stableVS} below). 
\end{remark}

\textbf{We now assume $k=2$.}  For emphasis, we will use the symbol $\beta$ for the form, rather than $f$. So $\cL$ is an arbitrary expansion of $\cL_{\textit{VF}}$ containing a binary function symbol $\beta:V\times V\to F$, and $T$ is a complete $\cL$-theory containing $T_{\textit{VF}}$ and an axiom saying $\beta$ is a bilinear form. Recall that $T'=T{\upharpoonright}(\cL\backslash{\{\beta\}})$. 
 We first use Lemma \ref{lem:terms} to prove an abstract test for transferring tameness from $T'$ to $T$.

\begin{lemma} \label{lem:fieldNFOP}
Assume $\cL$ is $V$-conservative. Suppose also that $T$ has quantifier elimination and $T'$ is stable. Then $T$ is $\NFOP_2$. 
\end{lemma}
\begin{proof}
Since $T$ has quantifier elimination and $\NFOP_2$ formulas are closed under Boolean combinations, it suffices to show that every atomic formula $\phi(x_1,x_2,x_3)$ is $\NFOP_2$ in $T$.  So assume $\phi(x_1,x_2,x_3)$ is a tri-partitioned atomic formula, where $x_1$, $x_2$, $x_3$ are tuples of variables.  Then it has the form $R(t_1(\xbar),\ldots, t_d(\xbar))$, where $R$ is a relation of $\calL$ (and so of $\calL'$),  $t_1,\ldots, t_d$ are $\calL$-terms, and $\xbar=(x_1,x_2,x_3)$.  By Lemma \ref{lem:terms}, for each $1\leq j\leq d$, there are $\calL'$-terms $t'_j$ so that 
$$
T\models \forall \xbar \bigwedge_{j=1}^d t_j(\xbar)=t'_j(\xbar,\boldsymbol{\beta}(\xbar_V)).
$$
Thus, $R(t_1(\xbar),\ldots, t_d(\xbar))$ is $T$-equivalent to $R(t'_1(\xbar,\boldsymbol{\beta}(\xbar_V)),\ldots, t'_d(\xbar,\boldsymbol{\beta}(\xbar_V)))$.

Let $\theta(\xbar, \ubar)$ be the $\calL'$-formula $R(t'_1(\xbar,\ubar),\ldots, t'_d(\xbar,\ubar))$. Thus, $\theta(\xbar,\ubar)$ is a stable formula by our assumption that $T'$ is stable. We also have that $\varphi(x_1,x_2,x_3)$ is $T$-equivalent to $\theta(\xbar,\boldsymbol{\beta}(\xbar_V))$. 

Observe that an element the tuple $\boldsymbol{\beta}(\xbar)$ has the form $\beta(u,v)$ for some singleton variables $u,v$ appearing in $x_1x_2x_3$.   If say $u\in x_i$ and $v\in x_j$, then, by adding dummy variables, $\beta(u,v)$ can be thought of as a (definable) function of the form $g(x_i,x_j)$, if $i\neq j$, or $g(x_i)$, if $i=j$. Thus we can write $\boldsymbol{\beta}(\xbar_V)=(g_1(u_1),\ldots, g_n(u_n))$, where each $u_i\in\{x_1,x_2,x_3,(x_1,x_2),(x_2,x_1),(x_1,x_3),(x_3,x_1),(x_2,x_3),(x_3,x_2)\}$.

The formula  $\theta(\xbar,\boldsymbol{\beta}(\xbar_V))$ thus has the form $\theta(x_1,x_2,x_3,g_1(u_1),\ldots, g_n(u_n))$.  Since $\theta(\xbar,\ubar)$ is stable, we have that $\varphi(x_1,x_2,x_3)$ is $\NFOP_2$ in $T$ by Theorem \ref{thm:FOP2comp}.  
\end{proof}

For the rest of this subsection,  let  $K$ be an arbitrary first-order structure expanding a field.

\begin{definition}[The theory $\widetilde{T}^K_\infty$]\label{def:TK}
 Define the language $\widetilde{\cL}^K_{\textit{VF}}=\cL_{\textit{VF}}\cup \cL_K$ where $\cL_K$ is a two-sorted language which copies the language of $K$ onto the $F$-sort (and has no symbols involving the $V$-sort). 
 Let $\widetilde{T}^{K}_{\infty}$ be the $\widetilde{\cL}^K_{\textit{VF}}$-theory consisting of $T_{\textit{VF}}$, the axioms of $\Th(K)$ as $\cL_K$-sentences in the $F$ sort, and axioms specifying that $V$ has infinite dimension over $F$.
\end{definition}

We now prove $\widetilde{T}^K_\infty$ is complete, and is stable if and only if $\Th(K)$ is stable.  The latter result follows from work of d'Elb\'{e}e, Kaplan, and Neuhauser \cite{DKN} on theories of algebraically closed fields with a predicate for a distinguished subfield. 

\begin{fact}\label{fact:stableVS}$~$
\begin{enumerate}[$(a)$]
\item $\widetilde{T}^K_\infty$ is complete. 
\item $\widetilde{T}^K_\infty$ is stable if and only if $\Th(K)$ is stable.
\end{enumerate}
\end{fact}
\begin{proof}
Part $(a)$. It suffices to show that any two saturated models of the same sufficiently large cardinality are isomorphic (without loss of set theoretic generality; see \cite{HalKap}). So fix saturated $M,N\models \widetilde{T}^K_\infty$ with $|M|=|N|=\kappa\geq|\widetilde{T}^K_\infty|$. Then $F(M)$ and $F(N)$ are saturated models of  $\Th(K)$ of  cardinality $\kappa$, and thus are isomorphic. We also have $\dim_{F(M)}(V(M))=\kappa=\dim_{F(N)}(V(N))$, and so an $\cL_K$-isomorphism from $F(M)$ to $F(N)$ determines a vector space isomorphism from $V(M)$ to $V(N)$, and hence an $\cL^K_{\textit{VF}}$-isomorphism from $M$ to $N$.  

Part $(b)$. Clearly, if $\widetilde{T}^K_\infty$ is stable then so is $\Th(K)$. So suppose $\Th(K)$ is stable. Let $L$ be an algebraically closed extension of $K$ with infinite transcendence degree over $K$. Let $(L,K)$ be the expansion of $L$ obtained adding a predicate naming $K$ along with its full structure. Then $\Th(L,K)$ is stable by \cite[Theorem 5.16]{DKN}. Since $\Th(L,K)$ interprets $\widetilde{T}^K_\infty$, we conclude that $\widetilde{T}^K_\infty$ is stable. 
\end{proof}

We note that the results of \cite{DKN} can similarly be used to prove the statement of Fact \ref{fact:stableVS}$(b)$ with ``stable" replaced by ``simple", ``$\NIP$'', or ``NSOP$_1$".

\begin{remark}
     Corollary \ref{cor:VSqe} below provides completeness and quantifier elimination for a suitable definitional expansion $T^K_\infty$ of $\widetilde{T}^K_\infty$. We also remark that in the case that $K$ is a \emph{definitional} expansion of the language of fields, completeness of $\widetilde{T}^K_\infty$ was proved by Kuzichev in the 1990s  using a relative quantifier elimination result (see \cite[Corollary 2]{Kuzichev}).
\end{remark}

We now describe the theories of vector spaces with bilinear forms studied in Granger's thesis \cite{Gr}. We will take some care here, as it has been discovered that the enriched two-sorted language which has traditionally been used to study vector spaces with a bilinear form is not sufficient for quantifier elimination.

Recall that $K$ is a fixed first-order expansion of a field. \textbf{For the rest of the subsection, we assume  $\Th(K)$ eliminates quantifiers.}  

\begin{definition}[The theories $T^K_{\infty,\textnormal{alt}}$ and $T^K_{\infty,\textnormal{sym}}$]\label{def:bilineartheories}
First recall the language $\cL^K_{\textit{VF}}$ and the  $\cL^K_{\textit{VF}}$-theory $\widetilde{T}^K_\infty$ from Definition \ref{def:TK}. 
\begin{enumerate}[$(1)$]
\item Define the two-sorted language
\[
\cL^K_\beta=\cL^K_{\textit{VF}}\cup\{\beta\}\cup\{g_{n,i}:1\leq i\leq n<\infty\}
\]
where $\beta$ is a binary function symbol from $V\times V$ to $F$ and each $g_{n,i}$ is an $(n+1)$-ary function symbol from $V^{n+1}$ to $F$. 
\item For each $1\leq i\leq n<\infty$, let $\Gamma_{n,i}$ be an $\cL^K_\beta$-sentence saying that for any $v_1,\ldots, v_{n+1}$ in $V$, $g_{n,i}(v_1,\ldots, v_{n+1})$ is $0_F$ if either $v_1,\ldots, v_n$ are linearly dependent or $v_1,\ldots, v_{n+1}$ are linearly independent, and otherwise equals the unique $a_i\in F$ so that $v_{n+1}=\sum_{j=1}^n a_jv_j$ for some $a_j\in F$.  
\item Let $T^K_{\infty,\beta}$ be the $\cL^K_\beta$-theory consisting of $\widetilde{T}^K_\infty$,   an axiom asserting that $\beta$ is a  bilinear form, and the axioms $\Gamma_{n,i}$ for each $1\leq i\leq n<\infty$.
\item Let $T^K_{\infty,\textnormal{alt}}$ be the expansion of $T^K_{\infty,\beta}$ which says that $\beta$ is non-degenerate and alternating.
\item Assume $\operatorname{char}(K)\neq 2$ and $K$ has square roots. Let $T^K_{\infty,\textnormal{sym}}$ be the expansion of $T^K_{\infty,\beta}$ which says that $\beta$ is non-degenerate and symmetric.
\end{enumerate}
\end{definition}

\begin{remark}\label{rem:infdim}
Note that one has the obvious finite-dimensional analogues $T^K_{m,\textnormal{alt}}$ and $T^K_{m,\textnormal{sym}}$ for $m\geq 1$. These are also studied by Granger but are less interesting for our purposes, since they are interpretable in a suitable expansion of $\Th(K)$ by constants (this is a straightforward exercise; see Section 6.1 of \cite{CH} for details).
\end{remark}

The following is one of the main results from Granger's thesis \cite{Gr}.

\begin{theorem}[Granger]\label{thm:QE}
$T^K_{\infty, \textnormal{alt}}$ and $T^K_{\infty,\textnormal{sym}}$ are complete and have quantifier elimination.
\end{theorem}

We mention again that an error in Granger's proof of quantifier elimination was later found by Macpherson. In Section \ref{sec:QE} we will provide a corrected and reorganized version of Granger's proof.  In any case, the results above come together to yield our main result.

\begin{theorem}\label{thm:VSFOP2}
Let $T$ be $T^K_{\infty, \textnormal{alt}}$ or $T^K_{\infty,\textnormal{sym}}$. Then $T$ is $\NFOP_2$ if and only if $\Th(K)$ is stable.
\end{theorem}
\begin{proof}
 If $T$ is $\NFOP_k$ then $\Th(K)$ is stable by Proposition \ref{prop:fieldFOP}, Fact \ref{fact:generic}, and since $T$ implies $\beta$ is non-degenerate. Conversely, suppose $\Th(K)$ is stable. Then $\widetilde{T}^K_\infty$ is stable by Fact \ref{fact:stableVS}. Note that  $T'=T{\upharpoonright}(\cL\backslash\{\beta\})$ is interpretable in $\widetilde{T}^K_\infty$ since each $g_{n,i}$  is $\cL_{\textit{VF}}$-definable. So $T'$ is stable. Since $T$ has quantifier elimination (Theorem \ref{thm:QE}), we can apply   Lemma \ref{lem:fieldNFOP} to conclude that $T$ is $\NFOP_2$. 
\end{proof}

\subsection{Quantifier Elimination}\label{sec:QE}
The goal of this subsection is to correct Granger's proof of Theorem \ref{thm:QE}. Let us first recall all the needed definitions. As in Section \ref{sec:VS}, we let $K$ be a fixed first-order expansion of a field  such that  $\Th(K)$ eliminates quantifiers. We are interested in theories of vector spaces over models of $\Th(K)$ equipped with bilinear forms. While our focus in Section \ref{sec:VS} was on the case of infinite dimensional vector spaces, here we will include the finite dimensional cases as well for completeness. 
\medskip

\noindent\textbf{The languages.} First we recall the basic vector space and field languages.
\begin{enumerate}[\hspace{5pt}$\ast$]
\item $\cL_{\textit{VF}}=\{V,F;+_F,\cdot_F,-_F,  ()\inv_F,0_F,1_F,+_V,0_V,\cdot_V\}$ is a  two-sorted language with sorts $V$ (vector space) and $F$ (field), unary function symbols $-_F$ and  $()\inv_F$ on $F$, binary function symbols $+_F$ and $\cdot _F$ on $F$, a binary function symbol $+_V$ on $V$, a binary function symbol $\cdot_V$ from $F\times V$ to $F$, constant symbols $0_F$ and $1_F$ in $F$, and a constant symbol $0_V$ in $V$. 
\item $\cL_K$ is a copy of the language of $K$ relativized to the $F$-sort (where we identify the field language with the corresponding symbols in $\cL_{\textit{VF}}$).  
\end{enumerate}
We then define:
\begin{enumerate}[\hspace{5pt}$\ast$]
\item $\cL^g_{\textit{VF}}=\cL_{\textit{VF}}\cup\{g_{n,i}:1\leq i\leq n<\infty\}$ where each $g_{n,i}$ is an $(n+1)$-ary function symbol from $V^{n+1}$ to $F$.
\item $\cL^{K}_{\textit{VF}}=\cL^g_{\textit{VF}}\cup \cL_K$. 
\item $\cL^K_f=\cL^{K}_{\textit{VF}}\cup\{f\}$ where $f$ is a $k$-ary function symbol from $V^k$ to $F$
\end{enumerate}
Note that in the case $k=2$, $\calL_f^K$ agrees with how $\cL^K_\beta$ was defined in Definition \ref{def:bilineartheories}.\medskip

\noindent\textbf{The theories.}  
First define $T^g_{\textit{VF}}$ to be the $\cL^g_{\textit{VF}}$-theory with the following axioms:
\begin{enumerate}[\hspace{5pt}$\ast$]
\item $V$ is a vector space over the field $F$ (i.e., the $\cL_{\textit{VF}}$-theory $T_{\textit{VF}}$ from Section \ref{sec:VS}).
 \item For each $1\leq i\leq n<\infty$, a sentence saying that for any $v_1,\ldots, v_{n+1}\in V$, $g_{n,i}(v_1,\ldots, v_{n+1})=0_F$ if either $v_1,\ldots, v_n$ are linearly dependent or $v_1,\ldots, v_{n+1}$ are linearly independent, and otherwise $g_{n,i}(v_1,\ldots, v_{n+1})$ is the unique $a_i\in F$ so that $v_{n+1}=\sum_{j=1}^n a_jv_j$ for some $a_j\in F$ with $j\in[n]\backslash\{i\}$.
\end{enumerate}
Note that $T^g_{\textit{VF}}$ is a definitional expansion of $T_{\textit{VF}}$. 

Next, given $m\in\Z^+\cup\{\infty\}$, define $T^{K}_m$ to be the $\cL^{K}_{\textit{VF}}$-theory consisting of $T^g_{\textit{VF}}$ together with the  following additional axioms:
\begin{enumerate}[\hspace{5pt}$\ast$]
\item $F$ satisfies $\Th(K)$ (relativized to the $F$ sort).
\item $V$ has dimension $m$ over $F$.
\end{enumerate}
Then define $T^K_{m,f}$ to be the $\cL^K_f$-theory consisting of $T^{K}_m$ together with an axiom asserting that $f$ is a $k$-linear form. 

Finally, we define the primary theories of interest in the $k=2$ case (where again we use $\beta$ in place of $f$ for emphasis).

\begin{definition}\label{def:bilinears}
Fix $m\in\Z^+\cup\{\infty\}$, and let $K$ be as above.
\begin{enumerate}[$(1)$]
    \item Assume $m$ is infinite or even. Let $T^K_{m,\textnormal{alt}}$ be the expansion of $T^K_{m,\beta}$ which says that $\beta$ is a non-degenerate alternating bilinear form.
    \item Assume $\operatorname{char}(K)\neq 2$ and $K$ has square roots. Let $T^K_{m,\textnormal{sym}}$ be the expansion of $T^K_{m,\beta}$ which says that $\beta$ is a non-degenerate symmetric bilinear form.
    \end{enumerate}
\end{definition}

The main result of this subsection is quantifier elimination and completeness for the theories defined in $(1)$ and $(2)$. As previously mentioned, this is an amendment to a result from Granger's thesis \cite{Gr} (Corollary \ref{cor:bilinearQE} below). So before continuing, we summarize the key issues in \cite{Gr}. The first of these (part $(1)$ of Remark \ref{rem:errors}) came to the authors' attention at a later stage in the preparation of this manuscript, which necessitated the addition of this subsection. In our efforts to correct Granger's proof, a few more issues emerged.

\begin{remark}\label{rem:errors}$~$
\begin{enumerate}[$(1)$]
\item In \cite{Gr}, Granger defines $T^K_{m,\textnormal{alt}}$ and $T^K_{m,\textnormal{sym}}$ in a  weaker language where, instead of the functions $g_{n,i}$, one has an $n$-ary relation $\theta_n(x_1,\ldots,x_n)$ on the $V$ sort interpreted as linear independence of $n$ vectors.\footnote{Note that the relation $\theta_n(x_1,\ldots,x_n)$ is equivalent mod $T^g_{\textit{VF}}$ to the quantifier-free formula  $g_{n,1}(x_1,\ldots,x_n,x_1)=1_F$.} However, it was pointed out by Macpherson that this language is not sufficient for quantifier elimination in general. A counterexample is described in the discussion after Definition 2.7 in the published version of \cite{Dobrowolski}. It is worth noting that  the $g_{n,i}$ functions are necessary to obtain quantifier elimination even at the level of the pure vector space language (see Corollary \ref{cor:VSqe}).

\item Granger uses a quantifier elimination test that only works for complete theories. More to the point, \cite[Proposition 9.2.8]{Gr} is false without the assumption of completeness. Since Granger derives completeness from quantifier elimination, this introduces a circularity in the proofs. See Remark \ref{rem:QEerror} for details.

\item The key algebraic tool needed for quantifier elimination is a result of Witt (given in Theorem \ref{thm:witt} below). However, Granger does not check one of the assumptions of this theorem, which concerns the existence of a certain $\cL^K_{\beta}$-isomorphism between two fixed models (of the theory in question). His work does suggest a strategy, which involves working with subspaces of countable dimension. But this introduces certain subtleties, which will we need to address. See Remark \ref{rem:wittissue} for further discussion.  
\end{enumerate}
\end{remark}

We now start working toward the proof of Corollary \ref{cor:bilinearQE}. The general ideas  follow Granger's work. However, in contrast to \cite{Gr}, we will start in the case of a general $k$-linear form, and provide a characterization of quantifier elimination in terms of a relatively straightfoward extension property  (see Definition \ref{def:rich} and Theorem \ref{thm:rich} below). We will then show that this extension property holds for the bilinear theories given in Definition \ref{def:bilinears}.  

Recall that $K$ is a fixed first-order expansion a field for which $\Th(K)$ has quantifier elimination. We also fix $m\in\Z^+\cup\{\infty\}$.

For an $\cL_f^K$-structure $M$ and $\zbar\in M$, we write $\langle \zbar \rangle$ for the $\cL_f^K$-substructure generated by $\zbar$. Similarly, if $\zbar\in F(M)$, we write $\langle \zbar \rangle_K$ for the $\cL_K$-substructure of $F(M)$ generated by $\zbar$. Recall from Definition \ref{def:fnbar} that for any $n$-ary function $h$ and a tuple $\xbar$, we write $\boldsymbol{h}(\xbar)$ for the tuple $(h(x_{i_1},\ldots, x_{i_n}))_{(i_1,\ldots, i_n)\in [\ell(\xbar)]^n}$ enumerated according to the lexicographic order on $[\ell(\xbar)]^n$. We first prove a few lemmas to understand the structure of finitely generated substructures of models of $T^K_{m,f}$, illustrating the leverage obtained from  the  $g_{n,i}$ functions.

\begin{lemma}  \label{lem:substruct}$~$
\begin{enumerate}[$(a)$]
    \item Suppose $M\models T^K_{m,f}$. For any tuples $\vbar\in V(M)$ and $\abar\in F(M)$, we have $V(\langle \vbar\abar\rangle)=\Span_{F(\langle \vbar\abar\rangle)}(\vbar)$.
    \item Suppose $M\models T^g_{\textit{VF}}$. If $M'$ is an $\cL^g_{\textit{VF}}$-substructure of $M$ and $\vbar\in V(M')$ is a  tuple, then $\ebar\subseteq \vbar$ is a basis for $\Span_{F(M')}(\vbar)$ if and only if it is a basis for $\Span_{F(M)}(\vbar)$. In particular, $\vbar$ is $F(M')$-linearly independent if and only if it is $F(M)$-linearly independent, and in this case  $$\dim_{F(M)}(\Span_{F(M)}(\vbar))=\dim_{F(M')}(\Span_{F(M')}(\vbar)).$$
\end{enumerate}
\end{lemma} 
\begin{proof}
    Part $(a)$. This can be shown by induction on terms and is left as an exercise to the reader.

    Part $(b)$. Every dependence relation between elements of $\vbar$ with coefficients in $F(M')$ is trivially a dependence relation with coefficients in $F(M)$. Conversely, assume we have $v_j=\sum_{i\in I_0}c_iv_i$ for some $c_i\in F(M)$ and $I_0$ a finite subset of the index set of $\vbar$ with $j\notin I_0$. Then we can choose an $F(M)$-linearly independent subset $\ebar=e_1\ldots e_n\subseteq (v_i)_{i\in I_0}$ so that $v_j=\sum_{i=1}^nd_ie_i$. Since $M\models T^g_{\textit{VF}}$, we have that $g_{n,i}(\ebar, v_j)=d_i$. Thus, we get that $d_i\in M'$, and so we have a  dependence relation in $F(M')$ between $v_j(v_i)_{i\in I_0}$. 

    Now, let $\ebar\subseteq \vbar$ be an $F(M')$-basis for $\Span_{F(M')}(\vbar)$. By the above, there is no dependence relation in $F(M)$ between elements of $\ebar$, and so $\ebar$ is $F(M)$-linearly independent. Moreover, every $v_j\in\vbar\backslash \ebar$ can be written as $\sum_{i\in I}c_ie_i$ for some finite subset $I$ of the index set of $\ebar$, with $c_i\in F(M')\subseteq F(M)$. Therefore, $\ebar$ also spans $\Span_{F(M)}(\vbar)$. 

    Conversely, let $\ebar\subseteq\vbar$ be an $F(M)$-basis for $\Span_{F(M)}(\vbar)$. By the first paragraph of the proof, $\ebar$ is also $F(M')$-linearly independent. Finally, each $v_j\in \vbar\backslash \ebar$  can be written as $\sum_{i\in I}c_ie_i$ with $c_i\in F(M)$ for some finite $I$. However, by the first paragraph again, we can find $d_i\in F(M')$ for $i\in I$ so that $v_j=\sum_{i\in I}d_ie_i$. Hence, $\ebar$ also spans $\Span_{F(M')}(\vbar)$.
\end{proof}

Part $(b)$ of the previous lemma tells us that linear independence and dimension is preserved under substructures and superstructures when working with models of $T^g_{\textit{VF}}$  (and so also of $T^K_{m,f}$).

In the next proof, we will use the following notation. Let $X$ be a set in the $V$-sort. We will write $f(X)$ to denote the image of $X^k$ under $f$, and similarly we write $g_{t,i}(X)$ to denote the image of $X^{t+1}$. Moreover, we will write $g_{\leq n}(X)$ to denote the set $\bigcup_{1\leq i\leq t\leq n}g_{t,i}(X^{t+1})$. On the other hand, for a tuple $\xbar$, we write $\boldsymbol{g}_{\leq n}(\xbar)$ for the tuple $(\boldsymbol{g}_{t,i}(\xbar))_{1\leq i\leq t\leq n}$ (with $(t,i)$ ordered lexicographically).

\begin{lemma}\label{lem:fieldsort}
    Suppose $M\models T^K_{m,f}$, $\vbar\in V(M)$ is a finite tuple, and $\abar\in F(M)$.  Then $F(\langle \vbar\abar\rangle)=\langle \boldsymbol{f}(\vbar)\boldsymbol{g}_{\leq n}(\vbar)\abar\rangle_K$. In particular, if $\vbar$ is linearly independent, then $F(\langle \vbar\abar\rangle)=\langle \boldsymbol{f}(\vbar)\abar\rangle_K$.
\end{lemma}
\begin{proof}
Let $n=\ell(\vbar)$.
    Let $N=\langle \vbar\abar\rangle$ and let $L=\langle \boldsymbol{f}(\vbar)\boldsymbol{g}_{\leq n}(\vbar)\abar\rangle_K$. The inclusion $F(N)\supseteq L$ is trivial. It is now enough to show that letting $F(N_0) = L$ and $V(N_0) = \Span_{L}(\vbar)$ gives an $\cL^K_f$-structure $N_0$, which means we must have $N\subseteq N_0$ and so $F(N)\subseteq F(N_0)=L$ as desired. However, it is by construction that $F(N_0)$ is a field and $V(N_0)$ is an $F(N_0)$-vector space. Moreover, it is also clear that $f(\wbar)\in F(N_0)$ and $g_{m,i}(\wbar')\in F(N_0)$ for any $\wbar,\wbar'\in V(N_0)$, $1\leq i\leq m$ and $1\leq m\leq n$. Thus, $F(\langle \vbar\abar\rangle)=F(N) = L = \langle \boldsymbol{f}(\vbar)\boldsymbol{g}_{\leq n}(\vbar)\abar\rangle_K$. 

    For the last part of the conclusion, if $\vbar$ is linearly independent then each coordinate of $\boldsymbol{g}_{\leq n}(\vbar)$ is either $0_F$ or $1_F$, and so 
    \[
    F(\langle \vbar\abar\rangle)=\langle \boldsymbol{f}(\vbar)\boldsymbol{g}_{\leq n}(\vbar)\abar\rangle_K=\langle \boldsymbol{f}(\vbar)\abar\rangle_K.\qedhere
    \]
\end{proof}

Lemma \ref{lem:fieldsort} along with part $(a)$ of Lemma \ref{lem:substruct} give us a complete description of the finitely generated substructures of a model of $T^K_{m,f}$. Now, we will turn our attention to a special class of maps between substructures of models of $T^K_{m,f}$, which will provide the majority of our isomorphisms.

\begin{definition}
    Let $M_1,M_2\models T^K_{m,f}$ and let $N_1\leq M_1$ and $N_2\leq M_2$ be $\calL^K_f$-substructures so that $V(N_1)$ and $V(N_2)$ have the same dimension. Suppose $\sigma_F\colon F(N_1)\to F(N_2)$ is an $\cL_K$-isomorphism and $\vbar\in V(N_1)$, $\wbar\in V(N_2)$ are bases.  We define $\sigma^{\vbar,\wbar}_F:V(N_1)\to V(N_2)$ by setting $\sigma^{\vbar,\wbar}_F(\lambda v_i)=\sigma_F(\lambda)w_i$, and then extending additively. 
\end{definition}

We show that as long as $\sigma_F$ respects the forms, $\sigma_F^{\vbar,\wbar}$ will be an $\cL^K_f$-isomorphism.

\begin{lemma} \label{lem:semisim}
    Suppose $M_1, M_2\models T^K_{m,f}$, with forms $f_1,f_2$ respectively, and $N_1\leq M_1$, $N_2\leq M_2$ are substructures such that $\dim V(N_1)=\dim V(N_2)$. Let $\vbar\in V(N_1)$ and $\wbar\in V(N_2)$ be bases. If $\sigma_F:F(N_1)\to F(N_2)$ is an $\calL_K$-isomorphism such that $\sigma_F(\boldsymbol{f}_1(\vbar))= \boldsymbol{f}_2(\wbar)$, then $(\sigma_F,\sigma^{\vbar,\wbar}_F)\colon N_1\to N_2$ is an $\calL^K_f$-isomorphism. 
\end{lemma}
\begin{proof}
    Let $\sigma=(\sigma_F,\sigma^{\vbar,\wbar}_F)$. Observe that $\sigma$ is a bijection on both sorts. Moreover, $\sigma$ trivially respects all constants and functions solely in the $F$-sort.  We also have $\sigma(0_V)=0_V$ and $\sigma(x+y)=\sigma(x)+\sigma(y)$ for $x,y$ in the $V$-sort by definition. Next, it is easy to check that for $\lambda\in F(M)$ and $x\in \Span(\vbar)$ we have $\sigma(\lambda\cdot x) = \sigma(\lambda)\cdot\sigma(x)$, since $\sigma(\lambda\cdot v_i)=\sigma_F(\lambda)\cdot w_i = \sigma(\lambda)\cdot\sigma(v_i)$. It then follows that $\sigma$ respects $g_{n,i}$ for all $1\leq i\leq n<\infty$. 
    Finally, we show that $\sigma$ repsects $f$ on $V(N_1)$. By Lemma \ref{lem:substruct}$(a)$, $V(N_1)=\Span_{F(N_1)}(\vbar)$, so it suffices to show
    \[
    f_2(\sigma(x_1),\ldots,\sigma(x_k))=\sigma(f_1(x_1,\ldots,x_k))
    \]
    for all $x_1,\ldots,x_k\in \Span_{F(N_1)}(\vbar)$. So fix $x_1,\ldots,x_k\in \Span_{F(N_1)}(\vbar)$ and for $t\in [k]$, write $x_t=\sum_{i\in I} c_{t,i}v_i$ with each $c_{t,i}\in F(N_1)$ and $I$ a sufficiently large finite subset of the indices of $\vbar$. Then
    \begin{align*}
        & \textstyle f_2\left(\sigma\left(\sum_{i\in I}c_{1,i}v_i\right),\ldots,  \sigma\left(\sum_{i\in I}c_{k,i}v_i\right)\right) \\
        &\textstyle \hspace{75pt}= f_2\left(\sum_{i\in I}\sigma(c_{1,i})w_i,\ldots,  \sum_{i\in I}\sigma(c_{k,i})w_i\right)\\
    &\textstyle \hspace{75pt} = \sum_{i_1\in I}\cdots\sum_{i_k\in I}\sigma(c_{1,i_1})\cdots\sigma(c_{k,i_k})f_2(w_{i_1},\ldots,w_{i_k})\\
    &\textstyle \hspace{75pt}=\sum_{i_1\in I}\cdots\sum_{i_k\in I_0}\sigma(c_{1,i_1})\cdots\sigma(c_{k,i_k})\sigma(f_1(v_{i_1},\ldots,v_{i_k}))\\
   &\textstyle \hspace{75pt}= \sigma\left(f_1\left(\sum_{i\in I}c_{1,i}v_i,\ldots, \sum_{i\in I}c_{k,i}v_i\right)\right). \qedhere
    \end{align*} 
\end{proof}

Our next goal is to formulate an extension property for expansions of $T^K_{m,f}$ which characterizes quantifier elimination (in the language $\cL^K_f$).

\begin{definition}\label{def:rich}
    Fix an $\cL^K_f$-theory $T\supseteq T^K_{m,f}$.
\begin{enumerate}[$(1)$]
\item  Let $N$ be a  model of $T$.
\begin{enumerate}[$(a)$]
\item Call a finite tuple $(\dbar,\wbar)$ from $N$ an \textbf{extension pair} if $\wbar$ is a linearly independent tuple from $V(N)$, and $\dbar$ is a tuple from $F(N)$ so that $|\dbar \boldsymbol{f}(\wbar)| = (|\wbar|+1)^k$. We say the extension pair $(\dbar,\wbar)$ is \textbf{realized in $N$} if there is some $y\in V(N)$ so that $y\wbar$ is linearly independent and $\boldsymbol{f}(y\wbar)=\dbar \boldsymbol{f}(\wbar)$.

\item Given an extension pair $(\dbar,\wbar)$ from $N$, let $q_{(\dbar,\wbar)}(x\vbar)$ be the quantifier-free type saying that $x\vbar$ is  linearly independent  and $\boldsymbol{f}(x\vbar)\models \qftp_{F(N)}(\dbar \boldsymbol{f}(\wbar))$. We say that $(\dbar,\wbar)$ is \textbf{consistent in $T$} if $T\cup q_{(\dbar,\wbar)}(x\vbar)$ is consistent.  
\end{enumerate}
\item  We say $T$ is \textbf{rich} if for any $\aleph_0$-saturated model $N\models T$ and any extension pair $(\dbar,\wbar)$ from $N$, if $(\dbar,\wbar)$ is consistent in $T$ then it is realized in $N$.
 \end{enumerate}
\end{definition}

The next theorem is our main result for the general $k$-linear case.

\begin{theorem}\label{thm:rich}
   Given an $\cL^K_f$-theory $T\supseteq T^K_{m,f}$, the following are equivalent.
   \begin{enumerate}[$(i)$]
        \item $T$ has quantifier elimination.
        \item $T$ is complete and has quantifier elimination.
        \item $T$ is rich.
   \end{enumerate}
\end{theorem}
\begin{proof}
    $(i)\Rightarrow (ii)$. Given a term $t(\xbar,\vbar)$, with $\vbar$ variables from the $V$-sort, we write $t(\xbar,0_V)$ to denote the term obtained by plugging in $0_V$ for every variable in $\vbar$. The next claim can be proved by induction on terms, and is left to the reader.\medskip

\noindent \textit{Claim:} Let $t(\xbar,\vbar)$ be an $\cL^K_f$-term, with the variables $\xbar$ in the $F$-sort and the variables $\vbar$ in the $V$-sort. 
    \begin{enumerate}[$(a)$]
        \item If $t(\xbar,\vbar)$ is $V$-valued,  then $T\models \forall \xbar (t(\xbar,0_V)=0_V)$.
        \item If $t(\xbar,\vbar)$ is $F$-valued, then there is an $\cL_K$-term $t'(\xbar)$ such that 
    \[T\models \forall \xbar (t(\xbar,0_V)=t'(\xbar)).\]    
    \end{enumerate}

Now assume $T$ has quantifier elimination. We prove that $T$ is complete. Fix an $\cL^K_f$-sentence $\varphi$. We want to show that $T\models \varphi$ or $T\models\neg\varphi$. Since $T$ has quantifier elimination, and $\cL^K_f$ has a constant symbol, we may assume that $\varphi$ is a quantifier-free $\cL^K_f$-sentence. It then suffices to assume $\varphi$ is atomic.  Since equality is the only relation in the $V$-sort, there are two cases:
\begin{enumerate}[$(1)$]
\item $\varphi$ is  $t_1(\cbar,0_V)=t_2(\dbar,0_V)$ where $t_1$ and $t_2$  are $V$-valued $\cL^f_K$-terms and $\cbar$, $\dbar$ are constants in $\cL_K$. 
\item $\varphi$ is  $R(t_1(\cbar_1,0_V),\ldots,t_n(\cbar_n,0_V))$ where each $t_i$ is an $F$-valued $\cL^f_K$-term, $R$ is an $n$-ary relation symbol in $\cL_K$, and $\cbar_1,\ldots,\cbar_n$ are constants in $\calL_K$.
\end{enumerate}
In case $(1)$, part $(a)$ of the claim implies that $\varphi$ is $T$-equivalent to $0_V=0_V$ and so $T\models\varphi$. In case $(2)$, part $(b)$ of the claim implies that there are $\cL_K$-terms $t'_1(\xbar),\ldots,t'_n(\xbar)$ such that $\varphi$ is $T$-equivalent to $R(t'_1(\cbar_1),\ldots,t'_n(\cbar_n))$, which is an $\cL_K$-sentence. Since $T\supseteq\Th(K)$ and $\Th(K)$ is complete, we thus have that either $T\models\varphi$ or $T\models\neg\varphi$. \medskip

$(ii)\Rightarrow (iii)$. Assume $T$ is complete and has quantifier elimination. Fix an $\aleph_0$-saturated model $N\models T$ and an extension pair $(\dbar,\wbar)$ in $N$ that is consistent in $T$. Then $q_{(\dbar,\wbar)}(x\vbar)$ is consistent with $T$ by assumption. Since $T$ is complete and $N$ is $\omega$-saturated, $q_{(\dbar,\wbar)}(x\vbar)$ is realized in $N$ by some $y'\wbar'$. Fix $\dbar'$ so that $\boldsymbol{f}(y'\wbar')=\dbar'\boldsymbol{f}(\wbar')$. Let $N_1=\langle \dbar' \wbar'\rangle$ and $N_2=\langle \dbar\wbar\rangle$. By Lemma \ref{lem:fieldsort}, we have $F(N_1)=\langle \dbar'f(\wbar')\rangle_K$ and $F(N_2)=\langle \dbar \boldsymbol{f}(\wbar)\rangle_K$. Since $\dbar'\boldsymbol{f}(\wbar')$ and $\dbar \boldsymbol{f}(\wbar)$ have the same quantifier-free type in $\cL_K$, there is an $\cL_K$-isomorphism $\sigma_F\colon F(N_1)\to F(N_2)$ with $\sigma_F(\dbar'\boldsymbol{f}(\wbar'))=\dbar \boldsymbol{f}(\wbar)$. By Lemma \ref{lem:semisim}, we have an $\cL^K_f$-isomorphism $\sigma=(\sigma_F,\sigma^{\wbar',\wbar}_F)\colon N_1\to N_2$. 
    
    Since $T$ has quantifier elimination, the inclusion $N_2\to N$ is a partial elementary map. Thus, $\sigma$ as a map $N_1\to N$ is partial elementary. Since $N_1$ is finitely generated, and $N$ is $\aleph_0$-saturated (hence $\aleph_0$-homogeneous), we can extend $\sigma$ to a partial elementary map $\tau$ whose domain contains $y'$. Let $y=\tau(y')$. Then $\tau(y'\wbar')=y\wbar$, and so $y\wbar$ must also be linearly independent. Moreover,
\[
\boldsymbol{f}(y\wbar)=\boldsymbol{f}(\tau(y'\wbar'))=\tau(\boldsymbol{f}(y'\wbar'))=\tau(\dbar'\boldsymbol{f}(\wbar'))=\sigma_F(\dbar'\boldsymbol{f}(\wbar'))=\dbar \boldsymbol{f}(\wbar).
\]
So $y$ witnesses that $(\dbar,\wbar)$ is realized in $N$.\medskip

$(iii)\Rightarrow (i)$. 
Assume $T$ is rich.  We use the characterization of quantifier elimination given by Proposition 2.29 in \cite{Pillay-MTnotes}. Fix $\aleph_0$-saturated $M, N\models T$. We want to show that the set $I$ of finite partial isomorphisms between $M$ and $N$ has the back-and-forth property. That is, we want to show that for any finite $x,\zbar\in M$ and $\zbar'\in N$ such that $\qftp_{M}(\zbar)=\qftp_{N}(\zbar')$, there is $y\in N$ so that $\qftp_{M}(x\zbar)=\qftp_{N}(y\zbar')$, and similarly for the dual case. 
    
     Fix finite tuples $\vbar\in V(M)$, $\abar\in F(M)$, $\wbar\in V(N)$ and $\bbar \in F(N)$ with $\qftp_M(\vbar\abar)=\qftp_N(\wbar\bbar)$. We first prove a claim, which essentially says that isomorphisms between these structures are determined by what happens on a basis, as long as we also add enough scalars to generate all the vectors from this basis. For use this in subsequent arguments, we will need a rather technical syntactic formulation of this statement.  \medskip

     \noindent\textit{Claim:} There are linearly independent tuples $\vbar'\seq\vbar$ and $\wbar'\seq \wbar$, and finite tuples $\abar'\supseteq\abar$ from $F(M)$ and $\bbar'\supseteq\bbar$ from $F(N)$, such that:
     \begin{enumerate}[$(i)$]
         \item $\qftp_M(\vbar'\abar')=\qftp_N(\wbar'\bbar')$.
         \item For arbitrary $x\in M$ and $y\in N$, if $\qftp_M(x\vbar'\abar')=\qftp_N(y\wbar'\bbar')$, then $\qftp_M(x\vbar\abar)=\qftp_N(y\wbar\bbar)$.
     
     \end{enumerate}
    \noindent\textit{Proof:}    Let $n=\ell(\vbar)=\ell(\wbar)$. Define $\abar'$ to be $\boldsymbol{g}_{\leq n}(\vbar)\abar$, and define $\bbar'$ to be $\boldsymbol{g}_{\leq n}(\wbar)\bbar$. Since $\qftp_M(\vbar\abar)=\qftp_N(\wbar\bbar)$, we have that
    \[
    \qftp_M(\vbar\abar')=\qftp_M(\vbar\boldsymbol{g}_{\leq n}(\vbar)\abar)=\qftp_N(\wbar\boldsymbol{g}_{\leq n}(\wbar)\bbar)=\qftp_N(\wbar\bbar').
    \]

    To ease notation, let $N_1\coloneqq \langle \vbar\abar\rangle$ and $N_2\coloneqq \langle \wbar\bbar\rangle$. Let $\vbar'\subseteq \vbar$ be a basis of $\Span_{F(N_1)}(\vbar)=V(N_1)$ (by Lemma \ref{lem:substruct}$(a)$), and let $\wbar'$ the corresponding subtuple of $\wbar$. Since $\vbar$ and $\wbar$ have the same quantifier-free type, $\wbar'$ is a basis for $\Span_{F(N_2)}(\wbar)=V(N_2)$ (again by Lemma \ref{lem:substruct}$(a)$). Since we are taking subtuples with the same indices and $\qftp_M(\vbar\abar')=\qftp_N(\wbar\bbar')$, we have $\qftp_M(\vbar'\abar')=\qftp_N(\wbar'\bbar')$, which is part $(i)$ of the claim.

    Now, fix $x\in M$ and $y\in N$ so that $\qftp_M(x\vbar'\abar')=\qftp_N(y\wbar'\bbar')$. Let $s=\ell(\vbar')=\ell(\wbar')$. For each $i\in [n]$, we have that $v_i = \sum_{j=1}^sg_{s,j}(\vbar', v_i)v'_j$ and 
    $w_i = \sum_{j=1}^sg_{s,j}(\wbar', w_i)w'_j$. By the construction of $\vbar'\abar'$ and $\wbar'\bbar'$, we have that there is some function $r:[n]\times[s]\to [\ell(\boldsymbol{g}_{\leq n}(\vbar))]$ such that for all $i\in [n]$ and $j\in [s]$, we have $g_{s,j}(\vbar',v_i)=a'_{r(i,j)}$ and $g_{s,j}(\wbar',w_i)=b'_{r(i,j)}$. Therefore, if we let $t_i(\xbar,\zbar)$ be the $\cL^K_f$-term $\sum_{j=1}^sz_{r(i,j)}x_j$, we get that $v_i=t_i(\vbar',\abar')$ and $w_i=t_i(\wbar',\bbar')$ for each $i\in [n]$. Let $\tbar(\xbar,\zbar)$ denote the tuple $(t_i(\xbar,\zbar))_{1\leq i\leq n}$. Fix a quantifier-free formula $\varphi(\zbar')$ with $\ell(\zbar')=\ell(x\vbar\abar)=\ell(y\wbar\bbar)$. Then we have 
    \[
    M\models \varphi(x\vbar\abar)\leftrightarrow \varphi(x\tbar(\vbar',\abar')\abar)\makebox[.5in]{and} N\models \varphi(y\wbar\bbar)\leftrightarrow \varphi(y\tbar(\wbar',\bbar')\bbar).
    \] 
    Let $\psi(x_1\xbar_2\xbar_3\xbar_4)$ be the formula $\varphi(x_1\tbar(\xbar_2,\xbar_3)\xbar_4)$, where $x_1$ is a singleton variable, $\ell(\xbar_2)=\ell(\vbar')=\ell(\wbar')$, $\ell(\xbar_3)=\ell(\abar')-\ell(\abar)=\ell(\bbar')-\ell(\bbar)$, and $\ell(\xbar_4)=\ell(\abar)=\ell(\bbar)$. Then  $\psi \in \qftp_M(x\vbar'\abar')$ if and only if $\psi \in \qftp_N(y\wbar'\bbar')$ by assumption. Thus, $\varphi \in \qftp_M(x\vbar\abar)$ if and only if $\varphi \in \qftp_N(y\vbar\abar)$. Since $\varphi$ was arbitrary, we have  $\qftp_M(x\vbar\abar)=\qftp_N(y\wbar\bbar)$ as desired. 
    \clqed
    \medskip

Now let $x\in M$ be arbitrary. We want to find some $y\in N$ such that $\qftp_M(x\vbar\abar)=\qftp_N(y\wbar\bbar)$. Let  $\vbar'$, $\abar'$, $\wbar'$, and $\bbar'$ as in the Claim. By part $(i)$ in the Claim, we have that $\qftp_M(\vbar'\abar')=\qftp_N(\wbar'\bbar')$. Also, by part $(ii)$ in the Claim, if we find some $y\in N$ such that $\qftp_M(x\vbar'\abar')=\qftp_N(y\wbar'\bbar')$, then $\qftp_M(x\vbar\abar)=\qftp_N(y\wbar\bbar)$. Therefore, without loss of generality, we can assume $\vbar$ and $\wbar$ are linearly independent (over $F(M)$ and $F(N)$, respectively), and set $\vbar'=\vbar$, $\wbar'=\wbar$, $\abar'=\abar$, and $\bbar'=\bbar$. Moreover,
 we clearly may assume $x\not\in\langle \vbar \abar\rangle$.

   Note that $F(M)$ and $F(N)$ are $\aleph_0$-saturated  (as models of $\Th(K)$), and so by quantifier elimination for $\Th(K)$ we have the back-and-forth property for finite partial isomorphisms between $F(M)$ and $F(N)$. In the subsequent arguments, we will refer to this simply as ``back-and-forth in the $F$-sort".  
       
     We now proceed by cases depending on the sort of $x$. Let $n=\ell(\vbar)=\ell(\wbar)$. \medskip

    \textit{Case 1:}  $x\in V(M)$. 
    Since $x\notin \langle \vbar\abar\rangle$, $x\vbar$ is a linearly independent tuple from $V(M)$, and thus $F(\langle x\vbar\abar\rangle) = \langle \boldsymbol{f}(x\vbar)\abar\rangle_K$ by Lemma \ref{lem:fieldsort}. Given any tuple $\zbar$, let $\boldsymbol{f}_1(\zbar)$ be the subtuple of $\boldsymbol{f}(\zbar)$ consisting of terms involving $z_1$. Then $\boldsymbol{f}_1(x\vbar)$ consists of the elements in $\boldsymbol{f}(x\vbar)$ arising as images of tuples involving $x$. By back-and-forth in the $F$-sort we obtain a tuple $\dbar$ from $F(N)$ of  length  $\ell(\boldsymbol{f}_1(x\vbar))$, and an isomorphism $\sigma_F\colon \langle \boldsymbol{f}(x\vbar)\abar\rangle_K\to\langle \dbar \boldsymbol{f}(\wbar)\bbar\rangle_K$ so that
    $\sigma_F(\boldsymbol{f}(x\vbar)\abar)=\dbar \boldsymbol{f}(\wbar)\bbar$.

    Note that $(\dbar,\wbar)$ is an extension pair in $N$. Moreover, $x\vbar\models q_{(\dbar,\wbar)}$ and $M\models T$ so $(\dbar,\wbar)$ is consistent in $T$. Thus, as $T$ is rich and $N\models T$ is $\aleph_0$-saturated, we can find some $y\in V(N)$ such that $y\wbar$ is linearly independent and  $\boldsymbol{f}(y\wbar)=\dbar \boldsymbol{f}(\wbar)$. By Lemma \ref{lem:fieldsort}, we have $F(\langle y\wbar\bbar\rangle)=\langle \boldsymbol{f}(y\wbar)\bbar\rangle_K$, and so $\sigma_F$ is an isomorphism from $F(\langle x\vbar\abar\rangle)$ to $F(\langle y\wbar\bbar\rangle)$. By Lemma \ref{lem:substruct}$(a)$, $V(\langle x\vbar\abar\rangle)$ is the $F(\langle x\vbar\abar\rangle)$-span of $x\vbar$, and similarly for $\langle y\wbar\bbar\rangle$. Therefore $\sigma\coloneqq(\sigma_F,\sigma_F^{x\vbar, y\wbar})\colon \langle x\vbar\abar\rangle\to\langle y\wbar\bbar\rangle$ is an $\cL^K_f$-isomorphism by Lemma \ref{lem:semisim}. Moreover, $\sigma$ witnesses $\qftp_M(x\vbar\abar)=\qftp_N(y\wbar\bbar)$.\medskip 
 
    \textit{Case 2:}  $x\in F(M)$. By back-and-forth in the $F$-sort, we obtain $y\in F(N)$ and an isomorphism $\sigma_F\colon \langle  \boldsymbol{f}(\vbar)x\abar\rangle_K\to\langle  \boldsymbol{f}(\wbar)y\bbar\rangle_K$ such that 
    $\sigma_F(\boldsymbol{f}(\vbar)x\abar)=\boldsymbol{f}(\wbar)y\bbar$. 
    Note that by Lemma \ref{lem:fieldsort}, $F(\langle \vbar x\abar\rangle)=\langle \boldsymbol{f}(\vbar)x\abar\rangle_K$ and $F(\langle \wbar y\bbar\rangle)=\langle  \boldsymbol{f}(\wbar)y\bbar\rangle_K$.  Moreover, by Lemma \ref{lem:substruct}$(a)$, $V(\langle \vbar x\abar\rangle)$ is the $F(\langle \vbar x\abar\rangle)$-span of $\vbar$ and similarly for $\langle \wbar y\bbar\rangle$. By Lemma \ref{lem:semisim}, $\sigma\coloneqq (\sigma_F,\sigma_F^{\vbar,\wbar})\colon \langle \vbar x\abar\rangle\to\langle \wbar y\bbar\rangle$ is an $\cL_f^K$-isomorphism.  Moreover, $\sigma$ witnesses $\qftp_M(x\vbar \abar)=\qftp_N(y\wbar\bbar)$.
\end{proof}

\begin{remark}\label{rem:QEerror}
Continuing from Remark \ref{rem:errors}$(2)$, we point out that when proving quantifier elimination, Granger fixes a single sufficiently saturated model $M$ and shows that if two finite tuples in $M$ have the same quantifier-free type, then they have the same complete type. However, this only implies quantifier elimination if the theory in question is already known to be complete. Otherwise, one must consider  two different models $M$ and $N$. For example, consider the empty theory in the language of equality, which does not have quantifier elimination but does satisfy the above condition for any single model. Proposition 9.2.8 of \cite{Gr} purports to justify this quantifier elimination test for an arbitrary theory, but the proof implicitly assumes completeness. It is worth saying that working with two models introduces some  technical issues related to the application of Theorem \ref{thm:witt} below (in particular, in controlling an isomorphism between the fields of the two  models). We also mention that in an earlier part of his thesis,  Granger asserts that \cite[Proposition 9.1.5]{Gr} can be used to prove completeness. But details are not provided, and it is not clear how one would proceed without adding further assumptions (such as $m<\infty$ or that $\Th(K)$ has a countable saturated model). 
\end{remark}

Now, we turn to the specific case when $k=2$, i.e., $f$ is a bilinear form. As usual, write $\beta$ instead of $f$ to distinguish it from the general case. We first state some facts from linear algebra, but translated to the context of $\cL^K_\beta$-structures.

\begin{fact}\label{fact:bases}
    Let $M$ be a model of $T^K_{m,\beta}$, with $\dim(V(M))\leq\aleph_0$.
\begin{enumerate}[$(a)$]
\item Assume $\beta$ is non-degenerate and symmetric in $M$, and $K$ has square roots. Then there is an orthonormal basis  for $V(M)$, i.e., a basis $(e_i)_{i\in I}$ such that for all $i,j\in I$, $\beta(e_i,e_j)=1_F$ if $i=j$ and $\beta(e_i,e_j)=0_F$ if $i\neq j$.
\item Assume $\beta$ is non-degenerate and alternating in $M$. Then there is a symplectic basis for $V(M)$, i.e., a partitioned basis $(e_i)_{i\in I}(f_i)_{i\in I}$ such that for all $i,j\in I$, $\beta(e_i,e_j)=\beta(f_i,f_j)=0_F$,  $\beta(e_i,f_j)=1_F$ if $i=j$, and $\beta(e_i,f_j)=0_F$ if $i\neq j$.
\end{enumerate}
\end{fact}

The next algebraic result is a restatement of Theorem 5.3.5 in Granger's thesis \cite{Gr}, which seems to be based on Remark 8 in \cite[Section I.7]{Gross}. We state the result in the specific case we need, which is when the form is a bilinear form that is non-degenerate symmetric or alternating. In the case that the form is non-degenerate symmetric, we need to assume $\textnormal{char}(K)\neq 2$ to get our desired conclusion without any extra assumptions. However, when the form is non-degenerate alternating, we can allow the characteristic to be arbitrary since an alternating form is always trace-valued (see Chapter I in \cite{Gross} for definitions and  details). Definition 10 in \cite[Section I.7]{Gross} defines a ``similitude", which is used in the statement of Remark 8. We note that an $\cL^K_\beta$-isomorphism is easily seen to be the same as a ``similitude with multiplier $1$". Therefore, in our terminology we have:

\begin{theorem}\label{thm:witt} 
    Suppose $M,N\models T^K_{m,\beta}$ and $\sigma_F\colon F(M)\to F(N)$ is an $\cL_K$-isomorphism. Assume that either $\beta$ is non-degenerate symmetric in $M, N$ and $\operatorname{char}(K)\neq 2$, or that $\beta$ is non-degenerate alternating in $M,N$. Suppose also that there is some $\sigma_V\colon V(M)\to V(N)$ so that $(\sigma_F,\sigma_V)\colon M\to N$ is an $\cL_\beta^K$-isomorphism. Then for any subspaces $W_1\subseteq V(M)$ and $W_2\subseteq V(N)$, and any $\rho\colon W_1\to W_2$ so that $(\sigma_F,\rho)\colon (F(M), W_1)\to (F(N), W_2)$ is an $\cL_\beta^K$-isomorphism, there is an extension of $(\sigma_F,\rho)$ to an $\cL^K_\beta$-isomorphism from $M$ to $N$.
\end{theorem}

\begin{remark}\label{rem:wittissue}
Continuing from Remark \ref{rem:errors}$(3)$, we note that Granger does not explicitly justify the existence of the $\cL^K_\beta$-isomorphism $(\sigma_F,\sigma_V)$ when applying Theorem $\ref{thm:witt}$. To address this, we use an approach suggested by \cite[Proposition 9.1.5]{Gr}. In particular, Fact \ref{fact:bases} provides a method for building $\cL^K_\beta$-isomorphisms between spaces of $\emph{countable}$ dimension. However, one must take care to preserve  the existence of an $\cL_K$-isomorphism between the respective fields. This introduces some tension between applying downward Lowenheim-Skolem to obtain countable dimension, and  working with saturated models to obtain isomorphic fields. One must also be sure to preserve the non-degeneracy of the form when passing to a countable dimension subspace. 
\end{remark}

The next lemma includes an extra set-theoretic assumption regarding the existence of saturated models, which we will discuss further after the proof. As suggested by the previous remark, we will also temporarily assume that $K$ is in a countable language.

\begin{lemma}\label{lem:chain}
Assume $\cL_K$ is countable. Let $T$ be either $T^K_{m,\textnormal{sym}}$ or  $T^K_{m,\textnormal{alt}}$, and suppose $\kappa$ is a cardinal such that any countable model of $T$ has  a saturated elementary extension of cardinality $\kappa$.  Fix $M\models T$ and a countable set $A\seq M$. Then there is an $\cL^K_\beta$-structure $N$ satisfying the following properties:
\begin{enumerate}[$(i)$]
\item $A\seq N$ and  $N\models T$.
\item $N$ is a substructure of a model of $\Th_A(M)$.
\item $\dim(V(N))\leq\aleph_0$.
\item Either $K$ is finite or $F(N)$ is a saturated model of $\Th(K)$ of cardinality $\kappa$.
\end{enumerate}
\end{lemma}
\begin{proof}
    Since $\cL^K_{\beta}$ is countable, we may let $M'\preceq M$ be a countable elementary substructure of $M$ containing $A$. Then $\dim(M')\leq\aleph_0$ and so by Fact \ref{fact:bases}, we may choose a countable basis $\ebar\in V(M')$ which is either orthonormal or sympletic (depending on $T$). Let $M''\succeq M'$ be a saturated elementary extension of cardinality $\kappa$. Finally, let $N=\langle  F(M'')\ebar\rangle$. Then $(iii)$ clearly holds by Lemma \ref{lem:substruct}$(b)$, and $N$ is a substructure of $M''\models \Th_A(M)$, so we have $(ii)$. Since $F(N)=F(M'')$, we have $(iv)$ by choice of $M''$. It remains to verify $(i)$. First, note that using Lemma \ref{lem:substruct}$(a)$ and the fact that $F(M')\subseteq F(M'')$ we have $$A\seq V(M')=\Span_{F(M')}(\ebar)\seq \Span_{F(M'')}(\ebar)=V(N).$$ Since $N$ is a substructure of $M''\models T$, and $F(N)=F(M'')$, we immediately get that $N$ satisfies all axioms of $T$ except possibly non-degeneracy. So it remains to show that $\beta$ is non-degenerate in $N$. Fix some nonzero $v\in V(N)=\Span_{F(M'')}(\ebar)$. Let $v=\sum_{i\in I}a_i e_i$ where $I$ is a (nonempty) finite subset of the index set of $\ebar$ and each $a_i$ is nonzero. Fix $t\in I$. Then since $\ebar$ is either orthogonal or symplectic, we can choose some coordinate $e_j$ such that $\beta(e_t,e_j)=1$ and $\beta(e_i,e_j)=0$ for all $i\neq t$. Therefore $\beta(v,e_j)=a_t\neq 0$, as desired. 
\end{proof}

\begin{remark}\label{rem:HalKap}
By \cite[Corollary 2.5]{HalKap}, in order to give a $\mathsf{ZFC}$ proof of quantifier elimination for a given theory (in a countable language), it suffices to exhibit a proof assuming the existence of some $\kappa$ as in the previous lemma. 
\end{remark}

\begin{remark}
    It seems reasonable that Lemma \ref{lem:chain} could be proved without the countability assumption on $\cL_K$, either through more finely tuned tools from linear algebra, or a careful Skolem hull argument in this particular language. We also note that this issue does not arise in Granger's thesis, since he assumes $\cL_K$ is a definitional expansion of the field language, hence is countable.
\end{remark}

Now we can prove the main goal in the bilinear case. 

\begin{corollary}\label{cor:bilinearQE}
Let $T$ be either $T^K_{m,\textnormal{sym}}$ or  $T^K_{m,\textnormal{alt}}$ (as  defined in Definition \ref{def:bilinears}). Then $T$ is complete and has quantifier elimination.
\end{corollary}
\begin{proof}
We first explain that we can reduce to the case that $\cL_K$ is countable. In particular, since $\Th(K)$ has quantifier elimination, we can write $\cL_K$ as a union $\bigcup_{\alpha<\lambda}\cL_\alpha$ where each $\cL_\alpha$ is countable extension of the field language and, if $K_\alpha\coloneqq K{\upharpoonright}\cL_\alpha$, then $\Th(K_\alpha)$ also has quantifier elimination. Now we have $T=\bigcup_{\alpha<\lambda}T^{K_\alpha}_{m,c}$, where $c$ denotes either ``sym" or ``alt". If we can prove quantifier elimination for each $T^{K_\alpha}_{m,c}$, then this will establish quantifier elimination for $T$ (and thus also completeness by Theorem \ref{thm:rich}).

So without loss of generality, we assume $\cL_K$ is countable. To prove quantifier elimination for $T$, it suffices by Theorem \ref{thm:rich} to show that $T$ is rich. Let $N$ be an $\aleph_0$-saturated model of $T$ and fix an extension pair $(\dbar,\wbar)$ in $N$ which is consistent in $T$. We want to show that $(\dbar,\wbar)$ is realized in $N$.

By assumption, $T\cup q_{(\dbar,\wbar)}$ is consistent. Thus we can find $M\models T$ and  $y'\wbar'\in V(M)$ linearly independent so that $\qftp_{F(M)}(\boldsymbol{\beta}(y'\wbar'))=\qftp_{F(N)}(\dbar\boldsymbol{\beta}(\wbar))$. Let $\dbar'\in F(M)$ be such that $\boldsymbol{\beta}(y'\wbar')=\dbar'\beta(\wbar')$. It follows that we have an $\cL_K$-isomorphism $\sigma^*_F\colon \langle \dbar'\boldsymbol{\beta}(\wbar')\rangle_K\to \langle \dbar\boldsymbol{\beta}(\wbar)\rangle_K$.

Let $A\seq M$ be a countable set containing $y'\wbar'$, along with either a basis for $V(M)$ if $m<\infty$, or a countably infinite independent set in $V(M)$ if $m=\infty$. Let $B\seq N$ be a countable set containing $\dbar\wbar$, along with either a basis for $V(N)$ if $m<\infty$, or a countably infinite independent set in $V(N)$ if $m=\infty$. 

Applying Lemma \ref{lem:chain} to $M$ and $N$ with $A$ and $B$ (and employing Remark \ref{rem:HalKap}), we  obtain  $\cL^K_\beta$-structures $M',N'$ satisfying the following properties.
\begin{enumerate}[$(i)$]
\item $A\seq M'$, $B\seq N'$, and $M',N'\models T$.
\item $N'$ is a substructure of a model of $\Th_B(N)$.
\item $\dim(V(M'))$ and $\dim(V(N'))$ are countable.
\item $F(M')$ and $F(N')$ are saturated models of $\Th(K)$ of the same cardinality.\footnote{Note that if $K$ is finite then this is automatic since models of $\Th(K)$ are saturated and pairwise isomorphic. In the case of finite $K$, it would also be reasonable to add constants for all elements of $K$, in which case $\langle X\rangle_K$ can be identified with $K$ for any set $X$.}
\end{enumerate}

Recall that $\sigma^*_F$ is an $\cL_K$-isomorphism between (finitely generated) substructures of $F(M')$ and $F(N')$. Since $\Th(K)$ has quantifier elimination, $\sigma_F$ is a partial elementary map from $F(N')$ to $F(N')$. Since $F(M')$ and $F(N')$ are saturated models of $\Th(K)$ of the same  cardinality, we can extend $\sigma^*_F$ to an  $\cL_K$-isomorphism $\sigma_F\colon F(M')\to F(N')$.

 Now let $W_1\seq V(M')$ and $W_2\seq V(N')$ be the subspaces generated by $\wbar'$ and $\wbar$, respectively. Then $W_1$ and $W_2$ have the same dimension since $\wbar'$ and $\wbar$ are linearly independent tuples of the same length. Since $\sigma_F(\boldsymbol{\beta}(\wbar'))=\boldsymbol{\beta}(\wbar)$, it follows from Lemma \ref{lem:semisim} that $(\sigma_F,\sigma^{\wbar',\wbar}_F)\colon (F(M'),W_1)\to (F(N'),W_2)$ is an $\cL^K_\beta$-isomorphism. 

Next, note that by choice of $A$ and $B$, condition $(ii)$ implies  that $V(M')$ and $V(N')$ have the same countable dimension (in particular, both dimensions are $m$ if $m<\infty$, and otherwise both are $\aleph_0$). Since $M',N'\models T$, we can use Fact \ref{fact:bases} to obtain bases $\ebar_1$ of $V(M')$ and $\ebar_2$ of $V(N')$ of the same length and with the property that \emph{any}  field automorphism from $F(M')$ to $F(N')$ (e.g., $\sigma_F$) will map $\beta(e_{1,i},e_{1,j})$ to $\beta(e_{2,i},e_{2,j})$ for all $i,j$. Thus  $(\sigma_F,\sigma_F^{\ebar_1,\ebar_2})$ is an $\cL^K_\beta$-isomorphism from $M'$ to $N'$ by Lemma \ref{lem:semisim}.

We now have the $\cL^K_\beta$-isomorphisms $(\sigma_F,\sigma^{\wbar',\wbar}_F)\colon (F(M'),W_1)\to (F(N'),W_2)$ and $(\sigma_F,\sigma_F^{\ebar_1,\ebar_2})\colon M'\to N'$. Applying Theorem \ref{thm:witt} with $M'$, $N'$, $\sigma_V=\sigma_F^{\ebar_1,\ebar_2}$, and $\rho=\sigma^{\wbar',\wbar}_F$, we obtain an $\cL^K_\beta$-isomorphism $\sigma\colon M'\to N'$ extending $(\sigma_F,\sigma_F^{\wbar',\wbar})$. Let $y^*=\sigma(y')$. Then we have that
\[
\boldsymbol{\beta}(y^*\wbar)=\boldsymbol{\beta}(\sigma(y'\wbar'))=\sigma(\boldsymbol{\beta}(y'\wbar'))=\sigma_F(\dbar'\boldsymbol{\beta}(\wbar'))=\dbar\boldsymbol{\beta}(\wbar).
\]
 Let $p(z)$ be the partial quantifier-free $\cL^K_\beta$-type, with parameters $\dbar\wbar\in N$ saying that $z\wbar$ is linearly independent and $\boldsymbol{\beta}(z\wbar)=\dbar\boldsymbol{\beta}(\wbar)$. Then we have just shown that $p(z)$ is realized in $N'$ (by $y^*$). Since $N'$ is a substructure of a model of $\Th_{\dbar\wbar}(N)$, and $\dbar\wbar\in B$, we have that $\Th(N)\cup p(z)$ is consistent. Since $N$ is $\aleph_0$-saturated, it follows that $p(z)$ is realized in $N$. Therefore, $(\dbar,\wbar)$ is realized in $N$, as desired.
 \end{proof}

We end this subsection with some further  applications of Theorem \ref{thm:rich} that are worth mentioning. 
First, we observe that Theorem \ref{thm:rich} can be used to analyze  the case of \emph{linear} forms. So now assume $k=1$ and so, in particular, $T^K_{m,f}$ asserts that $f$ is a linear form.  For $n_1,n_2\in \Z^+\cup\{\infty\}$, let $T^K_{(n_1,n_2)}$ be the $\cL^K_f$-theory extending $T^K_{n_1+n_2,f}$ which says that the kernel of $f$ has dimension $n_1$ and codimension $n_2$. Note that \textit{any} complete extension $T\supseteq T^K_{m,f}$ will specify the dimension and codimension of the kernel of $f$, so the following result characterizes all such theories. 

\begin{corollary}\label{cor:linearQE}
    $T^K_{(n_1,n_2)}$ is complete and has quantifier elimination.
\end{corollary}
\begin{proof}
    By Theorem \ref{thm:rich}, it is enough to show $T^K_{(n_1,n_2)}$ is rich. Fix an $\aleph_0$-saturated $N\models T^K_{(n_1,n_2)}$ and an extension pair $(\dbar,\wbar)$ from $N$ consistent in $T^K_{(n_1,n_2)}$. We want to show that $(\dbar,\wbar)$ is realized in $N$. 
    
    By the consistency assumption, there is some $M\models T^K_{(n_1,n_2)}$ with $y'\wbar'\in V(M)$ linearly independent so that $\qftp_{F(M)}(\boldsymbol{f}(y'\wbar'))=\qftp_{F(N)}(\dbar\boldsymbol{f}(\wbar))$. Thus, we can choose an $\cL_K$-isomorphism $\sigma_F^*:\langle \boldsymbol{f}(y'\wbar')\rangle_K\to\langle \dbar\boldsymbol{f}(\wbar)\rangle_K$. Choose saturated elementary extensions $N'\succeq N$ and $M'\succeq M$ of the same cardinality (using Remark \ref{rem:HalKap}). Then $F(N')$ and $F(M')$ are saturated models of  $\Th(K)$ with the same cardinality. Since $\Th(K)$ has quantifier elimination in $\cL_K$, the partial isomorphism $\sigma_F^*$ is a partial elementary map from $F(M')$ to $F(N')$. Therefore, there is an $\cL_K$-isomorphism $\sigma_F:F(M')\to F(N')$ extending $\sigma_F^*$. 
    
    Now, we choose a basis $\ebar=\ebar_1\ebar_2$ of $V(M')$ and a basis $\ebar^* = \ebar^*_1\ebar^*_2$ of $V(N')$ so that $\ebar_1$ and $\ebar^*_1$ span the kernel of $f$ in $M'$ and $N'$, respectively. Note that regardless of whether $n_1$ and $n_2$ are finite or infinite, we have  $\ell(\ebar_1)=\ell(\ebar^*_1)$ and $\ell(\ebar_2)=\ell(\ebar^*_2)$ by saturation. Moreover, we have  $f(e_{1,i})=0_F$ for every $i\leq \ell(\ebar_1)$ and $f(e_{2,i})\neq0_F$ for every $i\leq \ell(\ebar_2)$. Without loss of generality, we can replace each $e_{2,i}$ by $f(e_{2,i})\inv e_{2,i}$, and thus assume  $f(e_{2,i})=1_F$ for every $i\leq \ell(\ebar_2)$. Similarly, we can assume that $f(e^*_{1,i})=0_F$ for $i\leq \ell(\ebar^*_1)$ and $f(e^*_{2,i})=1_F$ for $i\leq \ell(\ebar^*_2)$. Therefore, $\sigma_F(\boldsymbol{f}(\ebar))=\boldsymbol{f}(\ebar^*)$, and so by Lemma \ref{lem:semisim}, we have an isomorphism $\sigma=(\sigma_F,\sigma_F^{\ebar,\ebar^*})\colon M'\to N'$. Now $\sigma(y'\wbar')\in N'$ is a linearly independent tuple of vectors so that $\boldsymbol{f}(\sigma(y'\wbar'))=\sigma(\boldsymbol{f}(y'\wbar'))=\dbar\boldsymbol{f}(\wbar)$. Therefore, as $N'\succeq N$ and $N$ is $\aleph_0$-saturated, we can find $y\in N$ so that $y\wbar$ is linearly independent and $\boldsymbol{f}(y\wbar)=\dbar\boldsymbol{f}(\wbar)$. That is, $(\dbar,\wbar)$ is realized in $N$ as desired.
\end{proof}

Finally, recall the theory $\widetilde{T}^K_m$ (defined for $m=\infty$ in Definition \ref{def:TK}) which axiomatizes vector spaces of dimension $m$ over models of $\Th(K)$ (in the two-sorted setting with a binary map for scalar multiplication). At the beginning of this subsection, we also defined the theory $T^K_m$ in the expanded language with the $g_{n,i}$ functions and associated axioms. In particular, $T^K_m$ is a definitional extension of $\widetilde{T}^K_m$. We  now see that this expansion is sufficient to obtain quantifier elimination. The idea is that we can further expand $T^K_m$ to an $\cL^K_\beta$-theory asserting that $\beta$ is identically $0$, in which case it is easy to show that we get a rich theory.

\begin{corollary}\label{cor:VSqe}
For any $m\in\Z^+\cup\{\infty\}$, $T^K_m$ is complete and has quantifier elimination.
\end{corollary}
\begin{proof}
Note that $T^K_{(m,0)}$ is the definitional extension of $T^K_m$ by a constantly $0$ linear form. Since any model of $T^K_m$ expands uniquely to a model of $T^K_{(\infty,0)}$, the result follows from Corollary \ref{cor:linearQE}.
\end{proof}

\section{Verification that $\bH_k$ is a Ramsey class}\label{sec:ramsey}

The goal of this section is prove Corollary \ref{cor:Gkmod}, which says that the Fra\"{i}ss\'{e} limit $\cH_k$ of the class $\bH_k$ from Definition \ref{def:Hk} has the modeling property.  By Theorem \ref{thm:scow}, this is equivalent to showing $\bH_k$ has the Ramsey property, and that is the main result of this section. 

\begin{theorem}\label{thm:hkramsey}
$\bH_k$ has the Ramsey Property.
\end{theorem}

For the reader's convenience, we recall the description of $\bH_k$ from Section \ref{sec:Hk}. Define a first-order language $\cL_k=\{P_1,\ldots,P_{k+1},<,<_k,R\}$, where $P_1,\ldots, P_{k+1}$ are unary relations, $<$ is a binary relation, $<_k$ is a $2k$-ary relation, and $R$ is a $(k+1)$-ary relation. 
Recall that $Q$ denotes the $k$-ary relation $P_1\times\ldots\times P_k$. Then $\bH_k$ is the class of finite $\cL_k$-structures satisfying the following axioms.

\begin{enumerate}[$(1)$]
\item $P_1,\ldots, P_{k+1}$ is a partition.
\item $<$ is a linear order with $P_1<\ldots< P_{k+1}$.
\item $R$ only holds on $P_1\times\ldots\times P_{k+1}$ (which we also view as $Q\times P_{k+1}$).
\item $<_k$ only holds on $Q\times Q$, and is a linear order on $Q$. 
\item For any $\xbar,\ybar\in Q$ and $w,z\in P_{k+1}$, $\big(\xbar\leq_k\ybar \wedge R(\ybar,w)\wedge w\leq z\big)\rightarrow R(\xbar,z)$.
\end{enumerate}

As the reader will see, the proof of Theorem \ref{thm:hkramsey} is quite long and could be ripe for optimization using a different approach. For example, it would be interesting to give a direct proof that the automorphism group of the Fra\"{i}ss\'{e} limit $\cH_k$ is extremely amenable (which, by a result of Kechris, Pestov, and Todorcevic \cite{KPT}, characterizes the Ramsey property for a Fra\"{i}ss\'{e} class). See also Problem \ref{prob:genram} below, which describes a more general approach that might shed light on streamlining the proof.

On the other hand, despite its length, the proof we give does demonstrate some nice connections to state of the art tools in Ramsey theory. Moreover, the proof proceeds by showing that a series of auxiliary classes have the Ramsey property. These classes are defined in \emph{functional} languages, which is compelling given the definition of $\FOP_k$. 

The key ingredients we will use are two general theorems about Ramsey classes, one due to Hubi\v{c}ka and Ne\v{s}et\v{r}il \cite{HN} and the other due to Evans, Hubi\v{c}ka, and Ne\v{s}et\v{r}il \cite{EHN}. Both results take place slightly outside the framework of classical first order logic.  In particular, \cite{HN} allows functions with partial domains, while in \cite{EHN}, functions are both partial and set-valued.  For this reason, we require some preliminaries in order to state the relevant theorems.

\textbf{In this section, we will work with a nonstandard notion of first-order structure in which function symbols are allowed to be interpreted as partial functions.}  Our terminology follows \cite{HN}.  Since this differs from the standard framework, we will give complete definitions. Suppose $\calL$ is a first-order language.  We use $\rho(S)$ to denote the arity of a symbol $S\in\cL$.  We define an \textbf{$\calL$-structure $\calM$} to consist of a set $M$ along with, for each relation symbol $R\in \calL$, a set $R^{\calM}\subseteq M^{\rho(R)}$, and for each function symbol $f\in \calL$, a set $\dom^{\calM}(f)\subseteq M^{\rho(f)}$ and a function $f^{\calM}:\dom(f)\rightarrow M$.

\begin{definition}
Suppose $\calM$ and $\calN$ are $\calL$-structures and $\eta\colon M\to N$ is a function.
\begin{enumerate}[$(1)$]
    \item $\eta$ is a \textbf{homomorphism (from $\cM$ to $\cN$}) if the following hold.
\begin{enumerate}[$(i)$]
\item For every relation symbol $R\in \calL$ and every $\abar\in M^{\rho(R)}$, if $\calM\models R(\abar)$ then $\calN\models R(\eta(\abar))$.
\item For every function symbol $f\in \calL$ and $\abar\in M^{\rho(f)}$, if  $\abar\in \dom(f^{\cM})$ then $\eta(\abar)\in \dom(f^{\cN})$ and  $\eta(f^{\calM}(\abar))=f^{\calN}(\eta(\abar))$.
\end{enumerate} 
\item  $\eta$ is a \textbf{embedding (from $\cM$ to $\cN$)} if it is injective and the following hold.
\begin{enumerate}[$(i)$]
\item For every relation symbol $R\in \calL$ and every $\abar\in M^{\rho(R)}$, $\calM\models R(\abar)$ if and only if $\calN\models R(\eta(\abar))$.
\item For every function symbol $f\in \calL$ and every $\abar\in M^{\rho(f)}$, $\abar\in \dom(f^{\cM})$ if and only if $\eta(\abar)\in\dom(f^{\cN})$, and in this case $\eta(f^{\calM}(\abar))=f^{\calN}(\eta(\abar))$.
\end{enumerate}
\end{enumerate}
\end{definition}

 As usual, we write $\eta\colon\cM\to \cN$ to indicate that $\eta$ is a homomorphism (or an embedding) from an $\cL$-structure $\cM$ to an $\cL$-structure $\cN$.

\begin{definition}\label{def:ramseyarrow}
Given $\calL$-structures, $\cA$ and $\cB$, define ${\cA\choose \cB}$ to be the set of embeddings $\sigma\colon \cB\to \cA$. We write $\cC\rightarrow (\cB)_n^{\cA}$ to mean that for all functions $\chi:{\cC\choose \cA}\rightarrow [n]$, there exists some embedding  in  $ {\cC\choose \cB}$ with image $\cB'$ such that $\chi$ is constant on ${\cB'\choose \cA}$.

A class $\bK$  of finite $\calL$-structures has the \textbf{Ramsey property} if for every  $\cA,\cB\in\bK$ and positive integer $n$ there is $\cC \in \bK$ such that $\cC\rightarrow (\cB)_n^{\cA}$.  
\end{definition}

\begin{remark}
Note that if $\cL$ is relational then the above definitions are the same as the usual first-order setting. Thus we recover the standard definition of the Ramsey property for a class of finite relational structures. In particular, this applies to $\bH_k$, which is a class of structures in the relational language $\cL_k$.
\end{remark}

We now state several standard definitions from Fra\"{i}ss\'{e} theory.

\begin{definition}
Suppose $\calA,\calB,\calC$ are $\calL$-structures, and $\alpha_1:\cC\rightarrow \cA$, $\alpha_2:\cC\rightarrow \cB$ are embeddings.  We say an $\calL$-structure $\calD$ is an \textbf{amalgamation of $\calA$ and $\calB$ over $\calC$ with respect to $\alpha_1$ and $\alpha_2$} if there are embeddings $\beta_1:\cA\rightarrow \cD$, $\beta_2:\cB\rightarrow \cD$ so that $\beta_1\circ \alpha_1=\beta_2\circ\alpha_2$.

We say the amalgamation is \textbf{strong} if $\beta_1(A\backslash \alpha_1(C))\cap \beta_2(B\backslash \alpha_2(C))=\emptyset$. We say the amalgamation is \textbf{free} if for all $\xbar\in D$, if $\xbar\cap \beta_1(A\backslash \alpha_1(C))\neq \emptyset$ and $\xbar\cap \beta_2(B\backslash \alpha_2(C))\neq \emptyset$, then $\xbar\notin \dom^{\calC}(f)$ for any function symbol $f\in\cL$ and $\xbar\notin R^{\calC}$ for any relation symbol $R\in \calL$. 
\end{definition}

\begin{remark}
    The previous notion of free amalgamation illustrates the extra leverage obtained by working with this more general notion of $\cL$-structure in which partial functions are allowed. 
\end{remark}

The Hereditary Property (HP) and Joint Embedding Property (JEP) are defined identically to the classical case (see Section 1.1 of \cite{HN}). 

\begin{definition}
Suppose $\bK$ is a class of finite $\calL$-structures.  We say $\bK$ has the \textbf{(strong/free) amalgamation property} if for any $\calA,\calB,\calC\in \bK$, and any embeddings $\alpha_1:\cC\rightarrow \cA$, $\alpha_2:\cC\rightarrow \cB$, there exists a (strong/free) amalgamation of $\calA$ and $\calB$ over $\calC$ with respect to $\alpha_1$ and $\alpha_2$.

We say $\bK$ is a \textbf{(strong/free) amalgamation class} if it has the (strong/free) amalgamation property, along with HP and JEP.
\end{definition}

\begin{definition} An $\calL$-structure is \textbf{irreducible} if it is not the free amalgamation of two of its proper substructures.  
\end{definition}

\begin{definition}Given a class $\bK$ of finite $\calL$-structures, let $\vec{\bK}$ be the class of all finite $(\calL\cup\{<\})$-structures obtained by expanding some structure in $\bK$ by a linear order.  
\end{definition}

The next result, from \cite{EHN}, is a generalization of the Ne\v{s}et\v{r}il-R\"{o}dl Theorem \cite{NR77, NR77b}, which says that  if $\bK$ is a free amalgamation class of finite structures in a relational language, then $\vec{\bK}$ has the Ramsey property.

\begin{theorem}[Evans-Hubi\v{c}ka-Ne\v{s}et\v{r}il \cite{EHN}]\label{thm:ehn}
Suppose $\bK$ is a free-amalgamation class of finite $\calL$-structures.  Then $\vec{\bK}$ has the Ramsey Property.
\end{theorem}

We note here that Theorem \ref{thm:ehn} in \cite{EHN} is in fact more general, applying to the setting of structures with set-valued functions.  

The second theorem we will use, Theorem \ref{thm:hn} below, comes from  \cite{HN}, and gives us a sufficient criterion for a subclass of a free amalgamation class to inherit the Ramsey property.  To state Theorem \ref{thm:hn}, we require a few more definitions.

\begin{definition} Suppose $\calA$ and $\calB$ are $\calL$-structures.  A homomorphism $\eta:\calA\rightarrow \calB$ is a \textbf{homomorphism-embedding} if, for every irreducible substructure $\calC\subseteq \calA$, $\eta|_{\calC}$ is an embedding. 
\end{definition}

\begin{definition} Given an $\calL$-structure $\calA$, a \textbf{completion of $\cA$} is an irreducible $\cL$-structure $\calC$ so that there exists a homomorphism-embedding from $\calA$ into $\calC$. If, moreover, $\cC$ is from a given class $\bR$, then we call $\cC$ an \textbf{$\bR$-completion} of $\cA$.
\end{definition}

\begin{definition}
Let $\bR$ be a class of finite  $\calL$-structures. A subclass $\bK$ of $\bR$ is \textbf{locally finite} if there is  $n\colon \bR\to \N$ such that for any $\cL$-structure $\cA$, if:
\begin{enumerate}[$(i)$]
\item every irreducible substructure of $\cA$ is in $\bK$,
\item $\cA$ has an $\bR$-completion $\cC$, and
\item every substructure of $\cA$ of size at most $n(\cC)$ has a $\bK$-completion,
\end{enumerate}
then $\cA$ has a $\bK$-completion.
\end{definition}

\begin{theorem}[Hubi\v{c}ka-Ne\v{s}et\v{r}il \cite{HN}]\label{thm:hn}
Let $\bR$ be a Ramsey class of  finite irreducible $\calL$-structures and let $\bK$ be a hereditary locally finite subclass of $\bR$ with the strong amalgamation property. Then $\bK$ is a Ramsey class.
\end{theorem}

We will apply Theorem \ref{thm:hn} twice.  We end this subsection with Fact \ref{fact:irred} below, which will simplify these applications.  In particular, under certain conditions, to check a map is a homomorphism-embedding, it suffices to show it is an embedding on all irreducible substructures.  Given an $\calL$-structure $\calM$, a set $X\subseteq M$, and a function symbol in $\calL$, we say that $X$ is \emph{closed under $f^{\calM}$} if, for every $\xbar\in X^{\rho(f)}\cap \dom(f^{\calM})$, $f^{\calM}(\xbar)\in X$.  If $X$ is the domain of a substructure of $\calM$, we write $\calM[X]$ to denote this substructure. 

\begin{fact}\label{fact:irred}
Suppose $\calA$ is an $\calL$-structure such that for all function symbols $f\in \calL$, and all $\xbar\in \dom(f^{\calA})$, $\xbar\cup \{f^{\calA}(\xbar)\}$ is closed under $f^{\calA}$. 
Let $\calB$ be an $\calL$-structure, and suppose $h\colon A\rightarrow B$ has the property that for every irreducible substructure $\calC\leq \calA$, $h|_C\colon \calC\rightarrow \calB$ is an embedding.  Then $h\colon \calA\rightarrow \calB$ is a homomorphism. 
\end{fact}
\begin{proof}  
Suppose $R(x_1,\ldots, x_m)$ is a relation of $\calL$, $a_1,\ldots, a_m\in A$, and $R^{\calA}(a_1,\ldots, a_m)$.  Then $\calA'=\calA[\{a_1,\ldots, a_m\}]$ is irreducible, so $h|_{\calA'}$ is an embedding, and thus $\calB\models R(h(a_1),\ldots, h(a_m))$. 

Suppose now $f\in\cL$ is an $m$-ary function symbol and $a_1,\ldots, a_m\in \dom(f^{\calA})$.  By assumption, $\calA'=\calA[\{a_1,\ldots, a_m,f^{\calA}(a_1,\ldots, a_m)\}]$ is an irreducible substructure of $\calA$. So $h|_{\calA'}$ is an embedding.  This implies $(h(a_1),\ldots, h(a_m))\in \dom(f^{\calB})$ and $h(f^{\calA}(a_1,\ldots, a_m))=f^{\calB}(h(a_1),\ldots, h(a_m))$.  Thus, $h$ is a homomorphism to $\calB$.
\end{proof}

\subsection{Proof of the main Ramsey result.}

In this subsection, we prove Theorem \ref{thm:hkramsey}.  We will do this by defining several auxiliary classes of finite structures, and using Theorems \ref{thm:ehn} and \ref{thm:hn}. 

Our first auxiliary class is defined in the following functional language.

\begin{definition}
 Let $\calL_{f,k}=\{U_1,\ldots, U_{k+1}, f\}$, where $U_1,\ldots,U_{k+1}$ are unary relation symbols, and $f$ is a $k$-ary function symbol.  
 \end{definition}
 
Now we define our first auxiliary class $\bR_k$, which tells us that the relations $U_1,\ldots, U_{k+1}$ are a partition and gives restrictions on the domain and range of $f$. 

\begin{definition}
Let $\bR_k$ be the class of finite $\calL_{f,k}$-structures $\calM$ where the following hold.
\begin{enumerate}
\item $M=U_1^{\calM}\cup \ldots\cup U_{k+1}^{\calM}$ is a partition.
\item $\dom(f^{\cM})\subseteq U_1^{\calM}\times \ldots \times U_k^{\calM}$ and $\im(f^{\cM})\subseteq U_{k+1}^{\calM}$.
\end{enumerate}
\end{definition}

Due to its simplicity, it is not difficult to show that $\bR_k$ is a free-amalgamation class.  We include a verification for completeness.  We note to the reader that for this to hold, it is crucial that our setting allows function symbols to have partial domains.

\begin{lemma}
$\bR_k$ is a free amalgamation class.
\end{lemma}
\begin{proof}
Suppose $\calA$, $\calB_1$, $\calB_2$ are in $\bR_k$, and for each $i\in [2]$,  $\alpha_i:\calA\rightarrow \calB_i$ is an embedding.  Define 
$$
C\coloneqq(B_1\backslash \alpha_1(A))\sqcup (B_2\backslash \alpha_2(A))\sqcup A.
$$
For $i\in [2]$, let $\beta_i\colon B_i\rightarrow C$ be the identity of $B_i\backslash \alpha_i(A)$ and $\alpha_i\inv$ on $\alpha_i(A)$.  Clearly each $\beta_i$ is an injection from $B_i$ into $C$, and $\beta_1(B_1)\cap \beta_2(B_2)=A$.  We now define a $\calL_{f,k}$-structure $\calC$ with domain $C$ as follows. For each $i\in [k+1]$, interpret
$$
U_i^{\calC}=(U_i^{\calB_1}\backslash \alpha_1(A))\cup (U_i^{\calB_2}\backslash \alpha_2(A)) \cup U_i^{\calA}.
$$
Then define $\dom(f^{\calC})=\beta_1(\dom(f^{\calB_1}))\cup \beta_2(\dom(f^{\calB_2}))$, and for each $i\in [2]$ and $\xbar\in \beta_i(\dom(f^{\calB_i}))$, set $f^{\calC}(\xbar)=\beta_i(f^{\calB_i}(\beta_i^{-1}(\xbar)))$.  This is well defined since, if $\xbar\in \beta_1(\dom(f^{\calB_1}))\cap \beta_2(\dom(f^{\calB_2}))$, then $\xbar\subseteq A$, and thus for each $i\in [2]$, 
\[
\beta_i(f^{\calB_i}(\beta_i^{-1}(\xbar)))=\beta_i(f^{\calB_i}(\alpha_i(\xbar))=\beta_i(\alpha_i(f^{\calA}(\xbar)))=f^{\calA}(\xbar).
\]

It is straightforward to check that $\beta_1$ and $\beta_2$ are $\calL_{f,k}$-embeddings into $\calC$, and $\beta_1\circ \alpha_1=\beta_2\circ\alpha_2=Id_A$.  Thus $\calC$ is an amalgam of $\calB_1$ and $\calB_2$ over $\calA$. By construction, $\calC$ is moreover a free amalgam. It is clear that $\calC$ is an element of $\bR_k$, so we have shown that $\bR_k$ has the free amalgamation property.

That $\bR_k$ has the hereditary property is obvious.  For the Joint Embedding Property\footnote{Again, we follow the convention that structures are nonempty.}, suppose $\calA,\calB\in \bR_k$.  Define $\calC$ to be the structure with domain $A\sqcup B$, so that for each $i\in [k+1]$, $U_i^{\calC}=U_i^{\calA}\sqcup U_i^{\calB}$, so that $\dom(f^{\calC})=\dom(f^{\calA})\sqcup \dom(f^{\calB})$, and so that $f^{\calC}$ is equal to $f^{\calA}$ on $\dom(f^{\calA})$ and equal to $f^{\calB}$ on $\dom(f^{\calB})$. Its clear $\calC$ is in $\bR_k$, and the natural inclusion maps give $\calL_{f,k}$-embedding from $\calA$ and $\calB$ into $\calC$, respectively.  Thus $\bR_k$ has the Joint Embedding Property.
\end{proof}

We obtain the following as an immediate corollary of Theorem \ref{thm:ehn}.  

\begin{corollary}\label{cor:rk}
$\vec{\bR}_k$ has the Ramsey property.
\end{corollary}

Recall that the elements of $\vec{\bR}_k$ are $\vec{\cL}_{f,k}$-structures, where $\vec{\cL}_{f,k}=\calL_{f,k}\cup \{<\}$. We now define a special subclass of $\vec{\bR}_k$ in which the domain of $f$ is all of $U_1\times \ldots \times U_k$ and where $U_1<\ldots<U_{k+1}$.

\begin{definition} Let $\bS_k$ be the class of all $\calM$ in $\vec{\bR}_k$ such that $\dom(f^{\cM})=U_1^{\calM}\times \ldots \times U_k^{\calM}$ and $U_1^{\calM}<\ldots <U_{k+1}^{\calM}$. 
\end{definition}

To ease notation, given an $\calL_{f,k}$-structure (or $\vec{\cL}_{f,k}$-structure) $\calM$, we let $Q^{\calM}=\prod_{i=1}^kU_i^{\calM}$.  With this notation, we have that for any $\calM$ in $\vec{R}_k$, $\dom(f^{\calM})\subseteq Q^{\calM}$, and for any $\calN\in \bS_k$, $\dom(f^{\calN})=Q^{\calN}$.  This notation is a purposeful analogy to notation set out in previous sections, where, given an $\calL_k$-structure $\calN$, we let $Q^{\calN}=\prod_{i=1}^kP_i^{\calM}$ (see the remarks directly before Definition \ref{def:Tk}). 

Our next goal is to show Lemma \ref{lem:skramsey} below, which says $\bS_k$ is an irreducible locally finite subclass of $\vec{\bR}_k$.  Combined with Theorem \ref{thm:hn}, this will show $\bS_k$ is a Ramsey class.

\begin{lemma}\label{lem:skramsey}
$\bS_k$ is an irreducible locally finite subclass of $\vec{\bR}_k$. 
\end{lemma}

\begin{proof}
That $\bS_k$ is irreducible is immediate from the fact that its elements are all linearly ordered by $<$.  We now check local finiteness.  Fix $\calC_0\in \vec{\bR}_k$, and let $n=k+3$.  Suppose $\calC$ is a finite $\vec{\calL}_{f,k}$-structure such that: 
\begin{enumerate}
\item $\calC_0$ is an $\vec{\bR}_k$-completion of $\calC$,
\item every irreducible substructure of $\calC$ is in $\bS_k$, and
\item every substructure of $\calC$ with at most $n$ vertices has an $\bS_k$-completion.
\end{enumerate}
We will show $\calC$ has a $\bS_k$-completion.  Let $h\colon \calC\rightarrow \calC_0$ be a homomorphism-embedding from $\calC$ to $\calC_0$, which is guaranteed to exist by (1).  We begin with a few observations about $\calC$.\medskip

\noindent\textit{Claim 1.} 
Given $\abar\in \dom(f^{\calC})$, $X_{\abar}:=\{a_1,\ldots, a_k,f^{\calC}(\abar)\}$ is closed under $f^{\calC}$.

\noindent\textit{Proof.} Since $h$ is a homomorphism-embedding into $\calC_0$,  $h(f^{\calC}(\abar))=f^{\calC_0}(h(\abar))$.  Since $\calC_0\in \vec{\bR}_k$, $h(\abar)\in \dom(f^{\calC_0})$ implies $h(\abar)\in Q^{\calC_0}$ and $h(f^{\calC}(\abar))=f^{\calC_0}(h(\abar))\in U_{k+1}^{\calC_0}$.  Since  $C_0=U_1^{\calC_0}\cup \ldots \cup U_{k+1}^{\calC_0}$ is a partition, this tells us that $h(\abar)$ is a tuple of $k$ distinct elements, none of which are $f^{\calC_0}(h(\abar))$.   In other words, we must have that $|h(X_{\abar})|=k+1$.  Clearly, $|X_{\abar}|\leq k+1$, so it follows that $|X_{\abar}|=k+1$ and $h|_{X_{\abar}}$ is an injection.  We now show that $X_{\abar}$ is closed under $f^{\calC}$.  Suppose towards a contradiction that there is some $\xbar\in X_{\abar}^k$ so that $\xbar\in \dom(f^{\calC})$, but $f^{\calC}(\xbar)\notin X_{\abar}$.  This implies that $\xbar\neq \abar$.  By the same argument as above, we must have $h(\xbar)\in Q^{\calC_0}$ and $h(f^{\calC}(\xbar))\in U_{k+1}^{\calC_0}$.  Since $\xbar\neq \abar$ and $h|_{X_{\abar}}$ is injective, $h(\xbar)\neq h(\abar)$. Since $h(\xbar)$ and $h(\abar)$ are distinct elements of $Q^{\calC_0}$, there must be some element of $h(X_{\abar})$ appearing in the tuple $h(\xbar)$ which does not appear in $h(\abar)$. The only element in $h(X_{\abar})$ not appearing in $h(\abar)$ is $h(f^{\calC}(\abar))=f^{\calC_0}(h(\abar))$, so we must have that $f^{\calC_0}(h(\abar))$ appears in $h(\xbar)$.  But  $h(\xbar)\in Q^{\calC_0}$, while $f^{\calC_0}(h(\abar))\in U_{k+1}^{\calC_0}$.  This is a contradiction since $U_{k+1}^{\calC_0}$ is disjoint from $U_1^{\calC_0},\ldots, U_k^{\calC_0}$.  Thus, $X_{\abar}$ is closed under $f^{\calC}$.\clqed\medskip

We can now prove the following facts about $\calC$. 
\begin{enumerate}[$(a)$]
\item $C=U_1^{\calC}\cup\ldots\cup U_{k+1}^{\calC}$ is a partition of $C$,
\item $\dom(f^{\calC})\subseteq Q^{\calC}$ and $\im(f^{\calC})\subseteq U_{k+1}^{\calC}$,
\item For all $x,y\in C$, if $x<^{\calC}y$, then $x\in U_i^{\calC}$ and $y\in U_j^{\calC}$ for some $i,j\in [k+1]$ satisfying $i\leq j$.
\end{enumerate}

For $(b)$, suppose $\abar\in \dom(f^{\calC})$. Claim 1 implies that there is a substructure $\calC[X_{\abar}]\leq \calC$ with domain $X_{\abar}$.  It is not difficult to verify that $\calC[X_{\abar}]$ is irreducible, and thus, by (2), is in $\bS_k$.  Therefore $\abar\in Q^{\calC}$ and $f^{\calC}(\abar)\in U_{k+1}^{\calC}$.  This shows $(b)$.

For $(a)$, fix $x\in C$. We claim that there is a set $A_x\subseteq C$ so that $x\in A_x$ and so that $\calC[A_x]$ is an irreducible substructure of $\calC$.  If $x\in X_{\abar}$ for some $\abar\in \dom(f^{\calC})$, let $A_x=X_{\abar}$.  If $x\in C\backslash (\bigcup_{\abar\in \dom(f^{\calC})}X_{\abar})$, let $A_x=\{x\}$.  In either case, $\calC[A_x]$ contains $x$ and is an irreducible substructure of $\calC$, and is therefore in $\bS_k$ by (2).  Thus, there is a unique $i\in [k+1]$ so that $x\in U_i^{\calC}$.  This shows $(a)$. 

For $(c)$, suppose $x,y\in C$ and $x<^{\calC}y$. Since we have already shown $(a)$, we know $x\in U_i^{\calC}$ and $y\in U_j^{\calC}$ for some $i,j\in [k+1]$.  We want to show $i\leq j$.  We claim there is a subset $A_{x,y}\subseteq C$ so that $x,y\in A_{x,y}$, $|A_{x,y}|\leq n$, and $\calC[A_{x,y}]$ is a substructure of $\calC$. If $x$ and $y$ appear in $X_{\abar}$ for some $\abar\in \dom(f^{\calC})$, let $A_{x,y}=X_{\abar}$.   If $x,y\in C\backslash (\bigcup_{\abar\in \dom(f^{\calC})}X_{\abar})$, let $A_{x,y}=\{x,y\}$.  If $x\in X_{\abar}$ for some $\abar\in \dom(f^{\calC})$ and $y\in C\backslash (\bigcup_{\abar\in \dom(f^{\calC})}X_{\abar})$, let $A_{x,y}=X_{\abar}\cup \{y\}$.  In each of these cases, $\calC[A_{x,y}]$ is a substructure of $\calC$ of size at most $k+2\leq n$.  By (3), $\calC[A_{x,y}]$ has an $\bS_k$-completion. Say $g\colon\calC[A_{x,y}]\rightarrow \calY$ is a homomorphism-embedding into some  $\calY\in \bS_k$.  Then $x<^{\calC}y$ implies $g(x)<^{\calY}g(y)$ and $x\in U_i^{\calC},y\in U_j^{\calC}$ implies $g(x)\in U_i^{\calY}$ and $g(y)\in U_j^{\calY}$. Since $\calY\in \bS_k$, this implies $i\leq j$, as desired.  This shows $(c)$.

For each $\xbar\in Q^{\calC_0}$, let $d_{\xbar}$ be a new constant symbol and set 
$$
D:=C_0\cup \{d_{\xbar}: \xbar\in Q^{\calC_0}\backslash \dom(f^{\calC_0})\}.
$$
We now define an $\vec{\cL}_{f,k}$-structure $\calD$ with underlying set $D$ as follows.  For each $i\in [k]$, set $U_i^{\calD}=U_i^{\calC_0}$, and let
$$
U_{k+1}^{\calD}:=U_{k+1}^{\calC_0}\cup \{d_{\xbar}: \xbar\in Q^{\calC_0}\backslash \dom(f^{\calC_0})\}.
$$
Let $\dom(f^{\calD})=Q^{\calD}=Q^{\calC_0}$, and for $\xbar\in Q^{\calD}$, set 
\[
f^{\calD}(\xbar)=\begin{cases} f^{\calC_0}(\xbar)&\text{ if }\xbar\in \dom(f^{\calC_0})\\
d_{\xbar}&\text{ if }\xbar\notin \dom(f^{\calC_0}).\end{cases}
\]
We now define $<^{\calD}$. If $x,y\in U_i^{\calD}$ for some $i\in [k]$, we let  $x<^{\calD}y$ hold if and only if $x<^{\calC_0}y$.  If $x\in U_i^{\calD}$ and $y\in U_j^{\calD}$ for some $1\leq i<j\leq k+1$, define $(x<^{\calD}y)\wedge \neg (y<^{\calD}x)$.  If $x\in U_{k+1}^{\calD}\backslash U_{k+1}^{\calC_0}$ and $y\in U_{k+1}^{\calC_0}$, define $(x<^{\calD}y)\wedge \neg (y<^{\calD}x)$.  Now let $<'$ be any linear order on $U_{k+1}^{\calD}\backslash U_{k+1}^{\calC_0}$.  Given $x,y\in U_{k+1}^{\calD}\backslash U_{k+1}^{\calC_0}$, define $x<^{\calD}y$ if and only if $x<'y$.  It is not difficult to check that by construction, $\calD\in \bS_k$.

Recall that $h\colon \cC\rightarrow \cC_0$ is a homomorphism-embedding.  By construction $C_0\subseteq D$, so we can consider $h$ as a map from $C$ to $D$.  We show that considered as such, $h$ is a homomorphism-embedding from $\calC$ to $\calD$.  By Claim 1 and Fact \ref{fact:irred}, it suffices to show that for every irreducible $\calA\leq \calC$, $h|_{\calA}$ is an embedding into $\calD$.

Suppose $\calA\leq \calC$ is an irreducible substructure of $\calC$.  By assumption, $h|_A$ is an embedding into $\calC_0$.  Therefore, $h|_A$ is an injection, and the following hold.
\begin{enumerate}[$(i)$]
\item $h(\dom(f^{\calA}))=\dom(f^{\calC_0})\cap h(A)^k$.
\item For all $\xbar\in \dom(f^{\calA})$, $h(f^{\calA}(\xbar))=f^{\calC_0}(h(\xbar))$.
\item For all $a\in A$ and $i\in [k+1]$, $a\in U_i^{\calC}$ if and only if $h(a)\in U_i^{\calC_0}$.
\item For all $x,y\in A$, $x<^{\calA}y$ if and only if $h(x)<^{\calC_0}h(y)$.
\end{enumerate}
We want to show the following.
\begin{enumerate}[$(i)'$]
\item $h(\dom(f^{\calA}))=\dom(f^{\calD})\cap h(A)^k$.
\item For all $\xbar\in \dom(f^{\calA})$, $h(f^{\calA}(\xbar))=f^{\calD}(h(\xbar))$.
\item For all $a\in A$ and $i\in [k+1]$, $a\in U_i^{\calA}$ if and only if $h(a)\in U_i^{\calD}$.
\item For all $x,y\in A$, $x<^{\calA}y$ if and only if $h(x)<^{\calD}h(y)$.
\end{enumerate}

That $(iii)'$ holds follows from the fact that $\im(h)\subseteq C_0$, the definition of $\calD$, and $(iii)$.  That $(ii)'$ holds follows from $(ii)$ and the fact that, by construction, $\dom(f^{\calC_0})\subseteq \dom(f^{\calD})$, and $f^{\calC_0}$ and $f^{\calD}$ agree on $\dom(f^{\calC_0})$. 

We now show $(iv)'$.  Suppose $x,y\in A$.  If $x,y\in U_i^{\calC}$ for some $i\in [k+1]$, then by construction, $h(x)<^{\calC_0}h(y)$ if and only if $h(x)<^{\calD}h(y)$.  Combining this with $(iii)$, we have that $x<^{\calA}y$ if and only if $h(x)<^{\calD}h(y)$, as desired.  Assume now $x\in U_i^{\calC}$ and $y\in U_j^{\calC}$ for some $1\leq i\neq j\leq k+1$.  Since $<^{\calC_0}$ is a linear order, exactly one of $h(x)<^{\calC_0}h(y)$ or $h(y)<^{\calC_0}h(x)$ holds. Without loss of generality, say $(h(x)<^{\calC_0}h(y))\wedge (\neg (h(y)<^{\calC_0}h(y)))$. By $(iii)$, this implies $(x<^{\calA}y)\wedge \neg(y<^{\calA}x)$.  By $(c)$, and since $i\neq j$ by assumption, we have $i<j$.  Since we already showed that $(iii)'$ holds, we know that $h(x)\in U_i^{\calD}$ and $h(y)\in U_j^{\calD}$. Thus, $(h(x)<^{\calD}h(y))\wedge (\neg (h(y)<^{\calD}h(y)))$ holds by construction of $\calD$ and because $i<j$. Therefore,  $x<^{\calA}y$ if and only if $h(x)<^{\calD}h(y)$, as desired.

We have left to show $(i)'$.  By construction, $\dom(f^{\calD})\cap h(A)^k=Q^{\calD}\cap h(A)^k$. Since we already know $(iii)'$, $Q^{\calD}\cap h(A)^k=h(Q^{\calA})$.  Thus we want to show $h(\dom(f^{\calA}))=h(Q^{\calA})$.  Since  $h$ is injective, this is equivalent to showing $\dom(f^{\calA})=Q^{\calA}$.  

Suppose towards a contradiction that $\dom(f^{\calA})\neq Q^{\calA}$. By $(b)$,  $\dom(f^{\calA})\subseteq Q^{\calA}$, so this means $\dom(f^{\calA})\subsetneq Q^{\calA}$.  Let $\abar=(a_1,\ldots, a_k)\in Q^{\calA}\backslash \dom(f^{\calA})$, and let $X=\{a_1,\ldots, a_k\}$. Since $\dom(f^{\calC})\subseteq Q^{\calC}$, and $X^k\cap Q^{\calC}=\{\abar\}$, we must have that $X$ is closed under $f^{\calC}$.  Let $\calX$ be the substructure of $\calA$ whose underlying set is just $X$. Since $<^{\calC_0}$ is a linear order, $(ii)$ implies that $<^{\calA}$ linearly orders $\calA$.  Thus $<^{\calX}$ linearly orders $\calX$, so $\calX$ is irreducbile.  But now $\calX$ is an irreducible substructure of $\calC$ of size $k\leq n$ which is not in $\bS_k$ (since $\dom(f^{\calX})=\emptyset\neq Q^{\calX}$).  This contradicts (2).  This finishes our proof that $h$ is a homomorphism-embedding into $\calD$.
\end{proof}

We obtain the following immediate corollary of Theorem \ref{thm:hn} and Lemma \ref{lem:skramsey}.

\begin{corollary}\label{cor:ramseysk}
$\bS_k$ has the Ramsey property.
\end{corollary}

We are now ready to introduce our third and final auxiliary class, $\bH_k^{\textnormal{pre}}$, which is defined below.  In contrast to $\vec{\bR}_k$ and $\bS_k$, which consist of $\vec{\calL}_{f,k}$-structures, $\bH_k^{\textnormal{pre}}$ will be a class of  $\calL_k$-structures.  To make the definition, we require the following notion of a total strict pre-order.  

\begin{definition}
Suppose $X$ is a set.  A \textbf{total strict  pre-order} on $X$ is a binary relation $\prec$ on $X$ such that the following hold.
\begin{enumerate}
\item For all $x\in X$, $\neg (x\prec x)$,
\item For all $x\neq y\in X$, $x\prec y$ or $y\prec x$ (or both),
\item For all $x,y,z\in X$, if $x\prec y$ and $y\prec z$, then $x\prec z$.
\end{enumerate}
In this case, we also let $\preceq $ denote the binary relation $x\prec y\vee x=y$. 
\end{definition}
 
It is easy to see that given a total strict pre-order $\prec $ on $X$, and $x,y\in X$, the relation $x\sim y$ given by 
$$
x\sim y \Leftrightarrow (x\preceq y \wedge y\preceq x))
$$
is an equivalence relation. Moreover $\prec$ induces a linear order on the equivalence classes $[x]$, which we will also denote by $\prec$.

We now define $\bH_k^{\textnormal{pre}}$, which is very similar to $\bH_k$.  The difference between the two classes is that $\bH_k^{\textnormal{pre}}$ will allow $<_k$ to be a strict pre-order on $Q$, while $\bH_k$ insists that $<_k$ is a linear order on $Q$. 

\begin{definition} Let $\bH_k^{\textnormal{pre}}$ be the class of finite $\calL_k$-structures $\calM$ satisfying the following theory, which we call $T_k^{\textnormal{pre}}$.
\begin{enumerate}[$(1)$]
\item $P_1,\ldots, P_{k+1}$ is a partition.
\item $<$ is linear order with $P_1<\ldots< P_{k+1}$.
\item $R$ only holds on $P_1\times\ldots\times P_{k+1}$ (which we also view as $Q\times P_{k+1}$).
\item $<_k$ only holds on $Q\times Q$.
\item $<_k$ is a total strict pre-order on $Q$,
\item For any $\xbar,\ybar\in Q$ and $w,z\in P_{k+1}$, $\big(\xbar\leq_k\ybar \wedge R(\ybar,w)\wedge w\leq z\big)\rightarrow R(\xbar,z)$.
\end{enumerate}
\end{definition}

As usual, the notation $x\leq_ky$ is shorthand for $x<_ky \vee x=y$ and $x\leq y$ is shorthand for $x<y\vee x=y$.  It is easy to check that $\bH_k^{\textnormal{pre}}$ has the hereditary property, as the axioms above can be expressed universally. Given an element in $\calY\in \bH_k^{\textnormal{pre}}$ and $y\in Q^{\calY}$, we let $[y]_{\calY}$ denote the equivalence class of $y$ in the strict pre-order $<_k^{\calY}$ on $ Q^{\calY}$, i.e. 
$$
[y]_{\calY}=\{z\in Q^{\calY}:  (y\leq_k^{\calY}z)\wedge (z\leq_k^{\calY}y)\}.
$$
When $a,b\in Q^{\calY}$, we write $a\sim^{\calY}_kb$ if and only if $[a]_{\calY}=[b]_{\calY}$.  Clearly $\sim_k$ is a quantifier-free definable relation in $\calL_k$.  

We will show $\bH_k^{\textnormal{pre}}$ has the Ramsey property by using a strong correspondence between structures in $\bH_k^{\textnormal{pre}}$ and $\bS_k$ (see Lemma \ref{lem:corr} below). This correspondence essentially maps structures in $\bS_k$ to structures in $\bH_k^{\textnormal{pre}}$ in such a way that no information is lost.  However, it is not, formally speaking, a bi-interpretation.  In what follows, we will necessarily be passing frequently between structures in the languages $\vec{\calL}_{f,k}$ and $\calL_k$.  We will use substructure/isomorphism to mean either $\vec{\calL}_{f,k}$-substructure/$\vec{\calL}_{f,k}$-isomorphism or $\calL_k$-substructure/$\calL_k$-isomorphism depending on the context, which should be sufficiently clear to the reader.

We begin by defining a way to generate an $\vec{\calL}_{f,k}$-structure from a structure in $\bH_k^{\textnormal{pre}}$. 
\begin{definition}
Suppose $\calB\in \bH_k^{\textnormal{pre}}$.  For each equivalence class $[\xbar]_{\calB}$ from the strict pre-order $<_k^{\calB}$, let $d_{[\xbar]_{\calB}}$ be a new constant symbol.  Define $\calB_f$ to be the $\calL_{f,k}$-structure with domain $B_f=B\cup \{d_{[\xbar]_{\calB}}: \xbar\in Q^{\calB}\}$, and where the following hold.
\begin{enumerate}[\hspace{5pt}$\ast$]
\item For $i
\in [k]$, $U_i^{\calB_f}=P_i^{\calB}$ (so $Q^{\calB_f}=Q^{\calB}$).
\item $U_{k+1}^{\calB_f}=P_{k+1}^{\calB}\cup \{d_{[\xbar]_{\calB}}: \xbar\in Q^{\calB}\}$. 
\item $U_1<^{\calB_f}\ldots<^{\calB_f}U_{k+1}$. 
\item For each $i\in [k+1]$ and $x,y\in P_i^{\calB}$, $x<^{\calB_f}y$ if and only if $x<^{\calB}y$.
\item For $y\in P_{k+1}^{\calB}$ and $\xbar\in Q^{\calB}$, $d_{[\xbar]}<^{\calB_f}y$ if and only if $R^{\calB}(\xbar,y)$.
\item $\dom(f^{\calB_f})=Q^{\calB}$ and for each $\xbar\in Q^{\calB}$, $f^{\calB_f}(\xbar)=d_{[\xbar]}$.
\item For $\xbar,\xbar'\in Q^{\calB}$, $d_{[\xbar]}<^{\calB_f}d_{[\xbar']}$ if and only if $(\xbar<_k\xbar')\wedge \neg(\xbar'<_k\xbar)$. 
\end{enumerate}
\end{definition}

It is straightforward to verify that for any $\calB\in \bH_k^{\textnormal{pre}}$, we have $\calB_f\in \bS_k$, and that the pre-images of points under $f^{\calB_f}$ are exactly the equivalence classes in $Q^{\calB}$ under  $<_k^{\calB}$.  We have a corresponding way of generating an $\calL_k$-structure from a structure in $\bS_k$.

\begin{definition}
Given $\calC\in \bS_k$, define $\calC_{R}$ to be the $\calL_k$-structure with domain $C_R=C\backslash \im(f^{\calC})$ such that the following hold.
\begin{enumerate}[\hspace{5pt}$\ast$]
\item For each $i\in [k]$, $P_i^{\calC_R}=U_i^{\calC}$. 
\item $P_{k+1}^{\calC_R}=U_{k+1}^{\calC}\backslash \mathrm{Im}(f^{\calC})$.
\item For all $x,y\in C_R$, $x<^{\calC_R}y$ if and only if $x<^{\calC}y$.
\item For  $\xbar,\xbar'\in Q^{\calC_R}$, $\xbar<_k^{\calC_R}\xbar'$ if and only if $\xbar\neq \xbar'$ and $f^{\calC}(\xbar)\leq f^{\calC}(\xbar')$.
\item For all $\xbar\in Q^{\calC_R}$ and $z\in P_{k+1}^{\calC_R}$, $R^{\calC_R}(\xbar, z)$ if and only if $f(\xbar)<^{\calC}z$. 
\end{enumerate} 
\end{definition}

It is straightforward to verify that if $\calC\in \bS_k$, then $\calC_R\in \bH_k^{\textnormal{pre}}$, and that the equivalence classes under the pre-order $<_k^{\calC_R}$ are exactly the pre-images of points under $f^{\calC}$. Furthermore, for any $\xbar,\xbar'\in Q^{\calC_R}$, $f^{\calC}(\xbar)<^{\calC}f^{\calC}(\xbar')$ holds if and only if $x\nsim_k^{\calC_R} y$ and $\xbar<^{\calC_R}_k\xbar'$.

Given  $\calC\in \bS_k$, we will need to consider a correspondence between substructures of $\calC$ and substructures of $\calC_R$.  For this we will need the following notation.

\begin{definition}\label{def:phipsi}
Suppose $\calC\in \bS_k$ and $\calA\in \bH_k^{\textnormal{pre}}$.  
\begin{enumerate}
\item Given $\calD\leq \calC$, let $\phi_{\calC}(\calD)$ be the substructure of $\calC_R$ with domain $D\setminus \im(f^{\calD})$,
\item Given $\calW\leq \calC_R$, let $\psi_{\calC}(\calW)$ be the substructure of $\calC$ with domain 
$$
W\cup \{f^{\calC}(\xbar): \xbar\in Q^{\calW}\}.
$$
\end{enumerate}
\end{definition}

If $\calC\in \bS_k$ then, since $\bS_k$ is hereditary, we have $\psi_{\calC}(\calW)\in \bS_k$ for any $\calW\leq \calC_R$.  Similarly, since $\calC_R\in \bH_k^{\textnormal{pre}}$ and $\bH_k^{\textnormal{pre}}$ is hereditary, we have that $\phi_{\calC}(\calD)\in \bH_k^{\textnormal{pre}}$ for any $\calD\leq \calC$.  We will also use the following lemma, which is easily verified from Definition \ref{def:phipsi}.

\begin{lemma}\label{lem:phipsi}
Suppose $\calC\in \bS_k$ and $\calA\in \bH_k^{\textnormal{pre}}$. 
\begin{enumerate}[$(a)$]
\item Given $\calD\leq \calC$ and $\calW\leq \calC_R$, $\psi_{\calC}(\phi_{\calC}(\calD))=\calD$ and  $\phi_{\calC}(\psi_{\calC}(\calW))=\calW$.
\item  $\calE\leq \calD\leq \calC$ implies $\phi_{\calC}(\calE)\leq \phi_{\calC}(\calD)\leq \calC_R$ and $\calX\leq \calY\leq \calC_R$ implies $\psi_{\calC}(\calX)\leq \psi_{\calC}(\calY)\leq \calC$.
\end{enumerate}
\end{lemma}

We can now prove our crucial lemma, which shows that, relative to a fixed $\calC\in \bS_k$, the maps $\phi_{\calC}$ and $\psi_{\calC}$ preserve isomorphism in a certain sense.

\begin{lemma}\label{lem:corr}
Suppose $\calC\in \bS_k$ and $\calA\in \bH_k^{\textnormal{pre}}$.  
\begin{enumerate}[$(a)$]
\item Given $\calW\leq \calC_R$, if $\calW\cong \calA$ then $\psi_{\calC}(\calW)\cong \calA_f$.
\item Given $\calD\leq \calC$, if $\calD\cong \calA_f$ then $\phi_{\calC}(\calD)\cong \calA$.
\end{enumerate}
\end{lemma}
\begin{proof}
Fix $\calC\in \bS_k$ and $\calA\in \bH_k^{\textnormal{pre}}$.  To ease notation, for the rest of the proof, we will write $\phi$ and $\psi$ to mean $\phi_{\calC}$ and $\psi_{\calC}$.

We first  show part $(a)$.  Fix $\calW\leq \calC_R$, and assume $\calW\cong \calA$.  Say $h\colon \calA\rightarrow \calW$ is an $\calL_k$-isomorphism.  We define a map $h'\colon \calA_f\rightarrow \psi(\calW)$ as follows.    Recall that by definition of $\calW$ and $\calA_f$, for each $i\in [k]$, $U_i^{\psi(\calW)}=P_i^{\calW}$ and $U_i^{\calA_f}=P_i^{\calA}$, and 
\begin{align*}
U_{k+1}^{\calA_f}=P_{k+1}^{\calA}\cup \{d_{[\xbar]_{\calA}}: \xbar\in Q^{\calA}\}\text{ and } U_{k+1}^{\psi(\calW)}&=P_{k+1}^{\calW}\cup \{f^{\calC}(w): w\in Q^{\calW}\}.
\end{align*}
Thus, $h$ induces a bijection from $U_i^{\calA_f}$ to $U_i^{\psi(\calW)}$ for all $i\in [k]$, as well as a bijection from  $P_{k+1}^{\calA}$ into $P_{k+1}^{\calW}$. So for all $x\in U_1^{\calA_f}\cup \ldots \cup U_k^{\calA_f}\cup P_{k+1}^{\calA}$, we define $h'(x)=h(x)$. We then extend $h'$ to include the rest of $U_{k+1}^{\calA_f}$ by setting, for each $\xbar\in Q^{\calA}$, $h'(d_{[\xbar]_{\calA}})=f^{\calC}(h(\xbar))$. Note that for each $\xbar\in Q^{\calA}$, $h(\xbar)\in \calW$, so $f^{\calC}(h(\xbar))\in \psi(\calW)$ since $\psi(\calW)\leq \calC$.  Thus $h'$ is indeed a map from $A_f$ to $\psi(\calW)$.  To show $h'$ is a bijection, we just need to show it induces a bijection from $\{d_{[\xbar]_{\calA}}: \xbar\in Q^{\calA}\}$ to $\{f^{\calC}(\wbar): \wbar\in Q^{\calW}\}$.

Assume $d_{[\xbar]_{\calA}}\neq d_{[\xbar']_{\calA}}$.  Then $[\xbar']_{\calA}\neq [\xbar]_{\calA}$, so 
\[
\calA\models \big((\xbar<_k\xbar')\wedge \neg (\xbar'<_k\xbar)\big)\vee \big((\xbar'<_k\xbar)\wedge \neg (\xbar<_k\xbar')\big) 
\]
  Since $h$ is an isomorphism, this implies that 
  \[
  \calW\models \big((h(\xbar)<_kh(\xbar'))\wedge \neg (h(\xbar')<_kh(\xbar))\big)\vee \big((h(\xbar')<_kh(\xbar))\wedge \neg (h(\xbar)<_kh(\xbar'))\big).
  \]
    Since $\calW\leq \calC_R$, we have that  
    \[
    \calC_R\models \big((h(\xbar)<_kh(\xbar'))\wedge \neg (h(\xbar')<_kh(\xbar))\big)  \vee \big((h(\xbar')<_kh(\xbar))\wedge \neg (h(\xbar)<_kh(\xbar'))\big) .
    \]
      By  definition of $\calC_R$, this implies $f^{\calC}(h(\xbar))\neq f^{\calC}(h(\xbar'))$, and thus, $h'(d_{[\xbar]_{\calA}})\neq h'(d_{[\xbar']_{\calA}})$.  Now fix $f^{\calC}(\wbar)$ for some $\wbar\in Q^{\calW}$.  Since $h$ is an isomorphism from $\calA$ to $\calW$, there is $\xbar\in Q^{\calA}$ so that $h(\xbar)=\wbar$.  Then by construction, $f^{\calC}(\wbar)=f^{\calC}(h(\xbar))=h'(d_{[\xbar]_{\calA}})$.  

We now have that $h':\calA_f\rightarrow \psi(\calW)$ is a bijection, and its easy to see it respects the predicates $U_1,\ldots, U_{k+1}$ by construction.  To finish showing it is an isomorphism, we just need to show it respects $<$ and $f$.

By construction, $\dom(f^{\psi(\calW)})=Q^{\psi(\calW)}=Q^{\calW}=h'(Q^{\calA})=Q^{\calA_f}=\dom(f^{\calA_f})$. So $h'(\dom(f^{\calA_f}))=\dom(f^{\psi(\calW)})$.  By construction, for any $\xbar\in Q^{\calA_f}$, $h'(f^{\calA_f}(\xbar))=h'(d_{[\xbar]_{\calB}})=f^{\calC}(h(\xbar))=f^{\psi(\calW)}(h'(\xbar))$.  This shows that $h'$ respects $f$. 

Now fix $a,b\in A_f$.  We show $a<^{\calA_f}b$ if and only if $h'(a)<^{\psi(\calW)}h'(b)$.  We will have a few cases depending on where $a,b$ live inside $\calA_f$.

Suppose first $a,b\in U_1^{\calA_f}\cup \ldots\cup U_k^{\calA_f}\cup P_{k+1}^{\calA}$.  By definition of $\calA_f$, $a<^{\calA_f}b$ if and only if $a<^{\calA}b$.  Since $h$ is an isomorphism, this holds if and only if  $h(a)<^{\calW}h(b)$.  By definition of $\psi(\calW)$, this holds if and only if $h(a)<^{\psi(\calW)}h(b)$, i.e., if and only if $h'(a)<^{\psi(\calW)}h'(b)$ (since $h$ and $h'$ agree on $a,b$). 

Now, suppose that $a,b\in U_{k+1}^{\calA_f}\backslash P_{k+1}^{\calA}$. This means $a=d_{[\xbar]_{\calA}}$ and $b=d_{[\xbar']_{\calA}}$ for some $\xbar,\xbar'\in Q^{\calA}$.  By definition of $\calA_f$,  $d_{[\xbar]_{\calA}}\leq^{\calA_f}d_{[\xbar']_{\calA}}$ if and only if $\xbar\leq _k^{\calA}\xbar'$.  Since $h$ is an isomorphism, this holds if and only if $h(\xbar)\leq_k^{\calW}h(\xbar')$.  Since $\calW\leq \calC_R$, this holds if and only if $h(\xbar)\leq_k^{\calC_R}h(\xbar')$.  By definition of $\calC_R$, this holds if and only if $f^{\calC}(h(\xbar))\leq^{\calC}f^{\calC}(h(\xbar))$.  By definition of $h'$, this holds if and only if $h'(d_{[\xbar]_{\calA}})\leq^{\calC}h'(d_{[\xbar']_{\calA}})$.  Since $\psi(\calW)\leq \calC$, this holds if and only if $h'(d_{[\xbar]_{\calA}})\leq^{\psi(\calW)}h'(d_{[\xbar']_{\calA}})$.  We have now shown $a\leq^{\calA_f}b$ if and only if $h'(a)\leq^{\psi(\calW)}h'(b)$.  Since $h'$ is injective, this means $a<^{\calA_f}b$ if and only if $h'(a)<^{\psi(\calW)}h'(b)$. 

Assume next that $b\in U_{k+1}^{\calA_f}\backslash P_{k+1}^{\calA}$ and $a\in P_1^{\calA}\cup \ldots \cup P_k^{\calA}$. Since $P_i^{\calA_f}<^{\calA_f}P_{k+1}^{\calA}$ holds for each $i\in [k]$, we have $a<^{\calA_f}b$.  On the other hand, $h'(a)=h(a)\in U_1^{\psi(\calW)}\cup \ldots \cup U_k^{\psi(\calW)}$, and $h'(b)\in U_{k+1}^{\psi(\calW)}$.  Thus, since $U_i^{\psi(\calW)}<^{\psi(\calW)}U_{k+1}^{\psi(\calW)}$ holds for each $i\in [k]$, $h'(a)<^{\psi(\calW)}h'(b)$. 

Lastly, assume that $b\in U_{k+1}^{\calA_f}\backslash P_{k+1}^{\calA}$ and $a\in P_{k+1}^\calA$.  So $b=d_{[\xbar]_{\calA}}$ for some $\xbar\in Q^{\calA}$.  By construction,  $d_{[\xbar]_{\calA}}<^{\calA_f}a$ if and only if $R^{\calA}(\xbar,a)$.  Since $h$ is an isomorphism, this holds if and only if $R^{\calW}(h(\xbar),h(a))$. Since $\calW\leq \calC_R$, this holds if and only if $R^{\calC_R}(h(\xbar),h(a))$.  By definition of $\calC_R$, this holds if and only if $f^{\calC}(h(\xbar))<^{\calC}h(a)$.   By definition of $h'$, this holds if and only if $h'(d_{[\xbar]_{\calA}})<^{\calC}h'(a)$.  Since $\psi(\calW)\leq \calC$, this holds if and only if $h'(d_{[\xbar]_{\calA}})<^{\psi(\calW)}h'(a)$.  This finishes our verficiation that $h'$ respects $<$, and is thus an $\vec{\calL}_{f,k}$-isomorphism.  This finishes the proof of part $(a)$. 

We now show part $(b)$.  Fix $\calD\leq \calC$ and assume $\calD\cong \calA_f$.  Say $g\colon \calD\rightarrow \calA_f$ is an $\vec{\calL}_{f,k}$-isomorphism.  We define a map $g'\colon \phi(\calD)\rightarrow \calA$.  Note that by construction, for each $i\in [k]$, $U_i^{\calD}=P_i^{\phi(\calD)}$ and $U_i^{\calA_f}=P_i^{\calA}$.  Therefore, $g$ induces a bijection from $U_i^{\calD}$ to $U_i^{\calA_f}$ for each $i\in [k+1]$.  

Recall that by definition of $\phi(\calD)$, $P_{k+1}^{\phi(\calD)}=U_{k+1}^{\calD}\backslash \im(f^{\calD})$.   Since $g$ is an $\vec{\calL}_{f,k}$-isomorphism,
\[
\im(f^{\calD})=g\inv(\im(f^{\calA_f}))=g\inv(\{d_{[\xbar]_{\calB}}: \xbar\in Q^{\calA}\}),
\]
where the second equality is because  $\im(f^{\calA_f})=\{d_{[\xbar]_{\calB}}: \xbar\in Q^{\calA}\}$ by construction.  
Thus, $P_{k+1}^{\phi(\calD)}=U_{k+1}^{\calD}\backslash \im(f^{\calD})=U_{k+1}^{\calD}\backslash g\inv(\{d_{[\xbar]_{\calB}}: \xbar\in Q^{\calA}\})$, and therefore 
\begin{multline*}
g(P_{k+1}^{\phi(\calD)})=g(U_{k+1}^{\calD}\backslash g\inv(\{d_{[\xbar]}: \xbar\in Q^\calA\})) \\
=g(U_{k+1}^{\calD})\backslash \{d_{[\xbar]}: \xbar\in Q^{\calA}\}
=U_{k+1}^{\calA_f}\backslash \im(f^{\calA_f}) =P_{k+1}^{\calA},
\end{multline*}
where the last equality uses that $U_{k+1}^{\calA_f}=P_{k+1}^{\calA}\cup \{d_{[\xbar]_{\calB}}: \xbar\in Q^{\calA}\}$ by construction.  Thus, $g$ also induces a bijection from $P_{k+1}^{\phi(\calD)}$ to $P_{k+1}^{\calA}$.  Therefore, setting $g'(x)=g(x)$ for all $x\in P_1^{\phi(\calD)}\cup \ldots \cup P_{k+1}^{\phi(\calD)}$, we have a bijection from $\phi(\calD)$ to $\calA$.

We will show $g'$ is an $\calL_k$-isomorphism. Above, we have seen that it respects the predicates $P_1,\ldots, P_{k+1}$ in the desired way, so we just have to check that it respects $<_k$, $R$, and $<$.

Fix $a,b\in \phi(\calD)$.  It is clear from the definition of $\phi(\calD)$ that $a<^{\phi(\calD)}b$ if and only if $a<^{\calD}b$.  Since $g$ is an isomorphism, this holds if and only if  $g(a)<^{\calA_f}g(b)$.  By above, $g'(a)=g(a)$ and $g'(b)=g(b)$, and $g(a),g(b)\in \calA$.  Therefore, by definition of $\calA_f$, $g(a)<^{\calA_f}g(b)$ if and only if $g'(a)<^{\calA}g'(b)$. Thus, $a<^{\phi(\calD)}b$ if and only if $g'(a)<^{\calA}g'(b)$.

Fix $\xbar, \xbar'\in Q^{\phi(\calD)}$.  If $\xbar=\xbar'$, then clearly $\neg (\xbar<_k^{\phi(\calD)}\xbar')$ and $\neg (g(\xbar)<^{\calA}_kg(\xbar'))$, as $<_k$ is a strict pre-order on $\calA$ and $\phi(\calD)$.

Assume now that $\xbar\neq \xbar'$. We have by definition of $\phi(\calD)$ that $\xbar<_k^{\phi(\calD)}\xbar'$ if and only if $f^{\calD}(\xbar)\leq^{\calD}f^{\calD}(\xbar')$.  Since $g$ is an isomorphism, this holds if and only if $g(f^{\calD}(\xbar))\leq^{\calA_f}g(f^{\calD}(\xbar'))$ if and only if $f^{\calA_f}(g(\xbar'))\leq^{\calA_f}f^{\calA_f}(g(\xbar'))$.    By definition of $\calA_f$, this holds if and only if $g(\xbar)\leq_k^{\calA}g(\xbar')$.  Since $g$ is an injection, $\xbar\neq \xbar'$, and $<^{\calA}_k$ is a strict pre-order, this holds if and only if $g(\xbar)<_k^{\calA}g(\xbar')$.  Thus, $\xbar<_k^{\phi(\calD)}\xbar'$ if and only if $g'(\xbar)<_k^{\calA}g'(\xbar')$.

Now fix $\xbar\in Q^{\phi(\calD)}$ and $z\in P_{k+1}^{\phi(\calD)}$.  By definition of $\phi(\calD)$, $R^{\phi(\calD)}(\xbar,z)$ if and only if $f^{\calD}(\xbar)<^{\calD}z$.  Since $g$ is an isomorphism, this holds if and only if $f^{\calA_f}(g(\xbar))<^{\calA_f}g(z)$.  By definition of $\calA_f$, this holds if and only if  $R^{\calA}(g(\xbar),g(z))$, i.e., if and only if $R^{\calA}(g'(\xbar),g'(z))$.  Thus $R^{\phi(\calD)}(\xbar,z)$ if and only if $R^{\calA}(g'(\xbar),g'(z))$.  This finishes our verification that $g'$ is an isomorphism, and thus finishes the proof of part $(b)$. 
\end{proof}

We can now obtain a corollary which will allow us to transfer the Ramsey property from $\bS_k$ to $\bH_k^{\textnormal{pre}}$. First, both of these classes consist of linearly ordered structures, and thus embeddings can be conflated with isomorphic copies. In other words, given two structures $\cA$ and $\cB$ from one of these classes, we identify ${\cA\choose \cB}$ with the set of substructures of $\cA$ isomorphic to $\cB$. \emph{This identification will be used without further mention in the next two results.} Note also that given $\calC\in \bS_k$, we can think of $\phi_{\calC}$ as a map from the set of substructures of $\calC$ into the set of substructures of $\calC_R$, and $\psi_C$ as a map from the set of substructures of $\calC_R$ into the set of substructures of $\calC$.  

\begin{corollary}\label{cor:phipsi}
Suppose $\calC\in \bS_k$ and $\calA\in \bH_k^{\textnormal{pre}}$.   For any $\calD\leq \calC$, the restriction of $\phi_{\calC}$ to ${\calD\choose \calA_f}$ is a bijection from ${\calD\choose \calA_f}$ to ${\phi_{\calC}(\calD)\choose \calA}$, whose inverse is the restriction of $\phi_{\calC}$ to ${\phi_{\calC}(\calD)\choose \calA}$. 
\end{corollary}
\begin{proof}

Fix $\calC\in \bS_k$ and $\calA\in \bH_k^{\textnormal{pre}}$.  For the rest of the proof, we will write $\phi$ and $\psi$ to mean $\phi_{\calC}$ and $\psi_{\calC}$.  Suppose $\calD\leq \calC$.  Fix $\calW\in {\calD\choose \calA_f}$.  Then $\phi(\calW)\leq \phi(\calD)$ by Lemma \ref{lem:phipsi}, and by Lemma \ref{lem:corr}, $\calW\cong \calA_f$ implies $\phi(\calW)\cong \calA$.  Thus $\phi(\calW)\in {\phi(\calD)\choose \calA}$ so the restriction of $\phi$ to ${\calD\choose \calA_f}$ maps into ${\phi(\calD)\choose \calA}$.  Suppose now $\calZ\in {\phi(\calD)\choose \calA}$.  Lemmas \ref{lem:phipsi} and \ref{lem:corr}, along with $\calZ\leq \phi(\calD)$, imply $\psi(\calZ)\leq \psi(\phi(\calD))=\calD$.  By Lemma \ref{lem:corr}, $\calZ\cong \calA$ implies $\psi(\calZ)\cong \calA_f$.  Combining we have that $\psi(\calZ)\in {\calD\choose \calA_f}$.  Thus the restriction of $\psi$ to ${\phi(\calD)\choose \calA}$ maps into ${\calD\choose \calA_f}$.  By Lemma \ref{lem:corr}, $\phi{\upharpoonright}{\calD\choose \calA_f}$ and $\psi{\upharpoonright}{\phi(\calD)\choose \calA}$ are then inverses of each other.  Hence, both maps are bijections. 
\end{proof}

\begin{proposition}\label{prop:sramsey}
$\bH_k^{\textnormal{pre}}$ has the Ramsey property.
\end{proposition}
\begin{proof}
Fix $\calA,\calB\in \bH_k^{\textnormal{pre}}$ and a positive integer $\ell$.  We show that there exists a structure $\calE\in \bH_k^{\textnormal{pre}}$ such that for all $\rho\colon {\calE\choose \calA}\rightarrow [\ell]$ there is a $\calX\in {\calE\choose \calB}$ so that $\rho$ is constant on ${\calX\choose \calA}$.  Since $\calA_f,\calB_f\in \bS_k$, Corollary \ref{cor:ramseysk} implies that there exists $\calC\in \bS_k$ so that for every  $\tau\colon{\calC\choose \calA_f}\rightarrow [\ell]$ there is a $\calX\in {\calC\choose \calB_f}$ so that $\tau$ is constant on ${\calX\choose \calA_f}$. In the rest of the proof, we write $\phi$ and $\psi$ for $\phi_{\calC}$ and $\psi_{\calC}$.

Recall $\calC_R\in \bH_k^{\textnormal{pre}}$, since $\calC\in \bS_k$.  This $\calC_R$ will be our desired $\calE$.  Fix any $\rho\colon{\calC_R\choose \calA}\rightarrow [\ell]$.   For each $\calX\in {\calC\choose \calA_f}$, define $\tau(\calX)\coloneqq\rho(\phi_{\calC}(\calX))$. Since $\calC_R=\phi_{\calC}(\calC)$ by definition,  Corollary \ref{cor:phipsi} (applied with $\calD=\calC$) implies that $\tau$ is a well defined coloring $\tau:{\calC\choose \calA_f}\rightarrow [\ell]$. 

By our choice of $\calC$,  there is some $\calD\in {\calC\choose \calB_f}$ so that $\tau$ is constant on ${\calD\choose \calA_f}$, say $\tau(\calX)=i$ for all $\calX\in {\calD\choose \calA_f}$. By Lemma \ref{lem:corr}, $\calD\cong \calB_f$ implies $\phi_{\calC}(\calD)\cong \calB$.  Thus, $\phi_{\calC}(\calD)\in {\calC_R\choose \calB}$.  By Corollary \ref{cor:phipsi} and our definition of $\tau$, for every $\calY\in {\phi_{\calC}(\calD)\choose \calA}$, $\psi_{\calC}(\calY)\in {\calD\choose \calA_f}$, so $\rho(\calY)=\tau(\psi_{\calC}(\calY))=i$.  Therefore, $\rho$ is constant on ${\phi_{\calC}(\calD)\choose \calA}$. 
\end{proof}

We now finish by showing that $\bH_k$ is an irreducible locally finite subclass of $\bH_k^{\textnormal{pre}}$.  With Proposition \ref{prop:sramsey} and Theorem \ref{thm:hn}, this will show $\bH_k$ is  a Ramsey class.

\begin{lemma}\label{lem:ramseylem}
$\bH_k$ is an irreducible locally finite subclass of $\bH_k^{\textnormal{pre}}$.
\end{lemma}
\begin{proof}
It is clear that $\bH_k$ is irreducible, because its elements are linearly ordered by $<$.  We now check local finiteness.  Fix $\calC_0\in \bH_k^{\textnormal{pre}}$, and let $n=\max\{k|C_0|^k,2k\}$.  Suppose $\calC$ is a finite $\calL_k$-structure such that the following hold:
\begin{enumerate}
\item $\calC_0$ is an $\bH_k^{\textnormal{pre}}$-completion of $\calC$.
\item every irreducible substructure of $\calC$ is in $\bH_k$, and
\item every substructure of $\calC$ with at most $n$ vertices has an $\bH_k$-completion.
\end{enumerate}
We show $\calC$ has an $\bH_k$-completion.  For every $a\in C$, $\calC[\{a\}]$ is an irreducible substructure of $\calC$, so by (2),  there is a unique $i\in [k+1]$ for which $a\in P_i^{\calC}$.  Therefore, $C=P_1^{\calC}\cup\ldots\cup P_{k+1}^{\calC}$ is a partition $\calC$.   Similarly, (2) implies that ${<_k^{\calC}}\subseteq Q^{\calC}\times Q^{\cC}$ and $R^{\calC}\subseteq Q^{\calC}\times P_{k+1}^{\calC}$.

We define an $\calL_k$-structure $\calD$ with domain $D=C_0$ as follows.  First, define $\calD$ to agree with $\calC_0$ on all symbols in $\calL_k\backslash \{<_k\}$, i.e. $\calD|_{\calL_k\backslash \{<_k\}}=(\calC_0)|_{\calL_k\backslash \{<_k\}}$.  We now define $<_k^{\calD}$. For all $\xbar,\ybar\in D^k$, if $\xbar\notin Q^{\calD}$ or $\ybar\notin Q^{\calD}$, we let $\neg (\xbar<_k^{\calD}\ybar)$ hold. For all $\xbar,\ybar\in Q^{\calD}$ satisfying $\xbar\nsim_k\ybar$ in $\calC_0$, define $\xbar<_k^{\calD}\ybar$ if and only if $\xbar<_k^{\calC_0}\ybar$.  We now only have left to define $<^{\calD}_k$ on pairs of elements from $Q^{\calD}$ which come from the same $\sim_k$-class in $\calC_0$.  

Let $E_1,\ldots, E_m$ be the equivalence classes in $Q^{\calC_0}$ under the strict pre-order $<_k^{\calC_0}$.  Given $1\leq i\leq m$, let 
$$
W_i=\{(\xbar,\ybar)\in E_i\times E_i: \text{there are 
 $\abar,\bbar\in Q^{\calC}$ with $f(\abar)=\xbar$, $f(\bbar)=\ybar$,  and $\abar<_k^{\calC}\bbar$}\},
$$
and let $W_i'$ be the transitive closure of $W_i$ in $E_i$.  We claim $W_i'$ is a partial order on $E_i$.  Suppose towards a contradiction this is not the case. Then there exists some $\ell\leq |E_i|$ and $\xbar_1,\ldots, \xbar_{\ell}\in E_i$ so that for each $1\leq i\leq \ell-1$, $(\xbar_i,\xbar_{i+1})\in W_i$ and $(\xbar_{\ell},\xbar_1)\in W_i$.   Thus, there are $\abar_1,\ldots, \abar_{\ell}\in  Q^{\calC}$ such that for each $1\leq i\leq \ell-1$, $\abar_i<^{\calC}_k\abar_{i+1}$ and $\abar_{\ell}<_k^{\calC}\abar_1$.   But now $\calC[\abar_1,\ldots, \abar_{\ell}]$ is a substructure of $\calC$ which has no $\calH_k$-completion.  Note that $\ell\leq |Q^{\calC_0}|\leq |C_0|^k$, so $\calC[\abar_1,\ldots, \abar_{\ell}]$ has size at most $k\ell\leq k|C_0|^k\leq n$, a contradiction. It follows that $W_i'$ is a partial order on $E_i$.  Let $W_i''$ be any linear order extending $W_i'$ on $E_i$.  Finally, for each $\xbar,\ybar\in E_i$, define $\xbar<_k^{\calD}\ybar$ if and only if $(\xbar,\ybar)\in W_i''$. 

This finishes our definition of $\calD$.  It is straightforward to verify that by construction, $\calD\in \bH_k$.  By (1), there is a map $h\colon C\rightarrow C_0$ which is a homomorphism-embedding from $\calC$ to $\calC_0$.  By construction, $D=C_0$, so we may consider $h$ as a map from $C$ to $D$.  It now suffices to check that $h$ is a homomorphism-embedding as a map from $\calC$ to $\calD$.  Since $\calL_k$ is relational, Fact \ref{fact:irred} tells us that it suffices to show that for every irreducible substructure $\calA$ of $\calC$, $h|_{\calA}$ is an embedding into $\calD$.

Let $\calA$ be an irreducible substructure of $\calC$.  Since $h$ is a homomorphism-embedding from $\calC$ to $\calC_0$, $h|_{\calA}$ is an embedding from $\calA$ to $\calC_0$.  This implies  $h|_{\calA}$ is an injection and, since $\calD|_{\calL_k\setminus\{<_k\}}=\calC|_{\calL_k\setminus\{<_k\}}$, $h$ respects the predicates $P_1,\ldots, P_{k+1}$ and the relations $<,R$ in the desired way, when considered as a map to $\calD$. 

Fix $\xbar,\ybar\in A^k$.  We want to show $\xbar<_k^{\calA}\ybar$ if and only if $h(\xbar)<^{\calD}h(\ybar)$. Suppose first  $\xbar=\ybar$, $\xbar\notin Q^{\calA}$, or $\ybar\notin Q^{\calA}$ hold. Then (2) implies $\neg(\xbar<_k^{\calA}\ybar)$.  By construction, in any of these cases we also have $\neg (h(\xbar)<_k^{\calD}h(\ybar))$, so we are done. 

Now assume $\xbar\neq \ybar$, and $\xbar,\ybar\in Q^{\calA}$.  We claim $h(\xbar)\nsim_kh(\ybar)$ in $\calC_0$.  Suppose toward a contradiction $h(\xbar)\sim_kh(\ybar)$ in $\calC_0$. Then $(h(\xbar)<_k^{\calC_0}h(\ybar))\wedge (h(\ybar)<_k^{\calC_0}h(\xbar))$ holds.  Since $h|_{\calA}$ is an embedding, this implies $(\xbar<_k^{\calA}\ybar)\wedge (\ybar<_k^{\calA}\xbar)$.  But $\calA[\xbar,\ybar]$ is an irreducible substructure of $\calC$ of size at most $2k$ which is not in $\bS_k$, a contradiction.  Thus, $h(\xbar)\nsim_kh(\ybar)$ in $\calC_0$.  By construction, $h(\xbar)<_k^{\calC_0}h(\ybar)$ if and only if $h(\xbar)<_k^{\calD}h(\ybar)$, and since $h|_{\calA}$ is an embedding to $\calC_0$, $\xbar<^{\calA}_k\ybar$ holds if and only if  $h(\xbar)<_k^{\calC_0}h(\ybar)$.  Thus $\xbar<^{\calA}_k\ybar$ holds if and only if  $h(\xbar)<_k^{\calD}h(\ybar)$, as desired.
\end{proof}

It is an interesting question whether the fact that $\bH_k$ is a Ramsey class can be proved via more direct methods.  For instance, the Ramsey property for $\bH_1$ follows from the Ne\v{s}et\v{r}il-R\"{o}dl Theorem mentioned above, and the classes $\bH_k$ for $k>1$ are all built from $\bH_1$ via the construction described in Definition \ref{def:fraisseop}.  This raises the question of whether, in the notation of Definition \ref{def:fraisseop}, one can always transfer the Ramsey property from a class $\bK$ to $\bK_k(U)$.  We make this problem precise below.  

\begin{problem}\label{prob:genram}
Suppose $\calL$ is a finite relational language containing a binary relation $<$ and a unary relation $U$.  Suppose $\bK$ is a Fra\"{i}ss\'{e} class of finite $\calL$-structures with disjoint amalgamation, and whose elements are linearly ordered by $<$, and where, given any $\calA\in \bK$ and $n\in \mathbb{N}$, there is $\calB\in \bK$ with $\calA\leq \calB$ and $|U(\calB)|\geq n$.
Show that if $\bK$ has the Ramsey property, so does $\bK_k(U)$ for any $k\geq 1$.
\end{problem}

While the proof we gave above of the Ramsey property for $\bH_k$ is tailored to this particular class, we believe it may still shed light on a strategy for Problem \ref{prob:genram}.

\end{document}